\documentclass[11pt,a4paper]{amsart}
\usepackage{newtx}
\usepackage{xcolor}
\usepackage{hyperref}
\usepackage{graphicx}
\usepackage{stmaryrd}
\SetSymbolFont{stmry}{bold}{U}{stmry}{m}{n}
\usepackage{tikz}
\usepackage{thmtools}
\usepackage{mathrsfs}
\usepackage[inline]{enumitem}
\usepackage[T1]{fontenc}
\usepackage[left=2.65cm,right=2.65cm,top=2.65cm,bottom=2.65cm,headheight=15pt,headsep=0.3cm,footskip=0.5cm]{geometry}
\usepackage[style=alphabetic,backend=biber,backref=true]{biblatex}

\newenvironment{acknowledgements}{\subsection*{Acknowledgements}}{}

\DefineBibliographyStrings{english}{%
    page             = {},
    pages            = {},
}
\DeclareFieldFormat{labelalpha}{\mkbibbold{#1}}
\DeclareFieldFormat{extraalpha}{\mkbibbold{\mknumalph{#1}}}

\renewbibmacro{in:}{}
\DefineBibliographyStrings{english}{
    backrefpage  = {pg.},
    backrefpages = {pg.}
}

\makeatletter
\DeclareFontFamily{U}{mathx}{}
\DeclareFontShape{U}{mathx}{m}{n}{%
    <5><6><7><8><9><10><10.95><12><14.4><17.28><20.74><24.88> mathx10%
}{}
\DeclareSymbolFont{mathx}{U}{mathx}{m}{n}
\DeclareMathAccent{\widecheck}{0}{mathx}{"71}
\makeatother

\usetikzlibrary{cd,calc,arrows.meta}

\tikzset{
    rootsys/.style={
        scale=1.5,
        line cap=round,
        line join=round,
        every node/.style={font=\Large}
    },
    rootwall/.style={
        red!70,
        line width=0.9pt
    },
    facet/.style={
        green!50!black,
        line width=2pt
    },
    root/.style={
        black,
        line width=1.4pt,
        -{Stealth[length=7pt,width=8pt]}
    },
    polygon/.style={
        fill=blue!15,
        draw=blue!60!black,
        line width=1.0pt
    },
    thickpolygon/.style={
        fill=blue!15,
        draw=blue!60!black,
        line width=3.0pt
    },
    blueweylchamber/.style={
        fill=blue!15,
        draw=none
    },
    pinkweylchamber/.style={
        fill=pink!60,
        draw=none
    },
    greenweylchamber/.style={
        fill=green!90!black,
        draw=none
    }
}

\newtheorem{theorem}{Theorem}[subsection]

\newtheorem{lemma}[theorem]{Lemma}

\theoremstyle{definition}
\newtheorem{definition}[theorem]{Definition}
\newtheorem{remark}[theorem]{Remark}
\newtheorem{notation}[theorem]{Notation}
\newtheorem{example}[theorem]{Example}

\declaretheorem[style=plain, name=Abridged Theorem, numbered=no]{mytheorem}

\newcounter{subfig}[theorem]
\renewcommand{\thesubfig}{\textbf{Figure~\thetheorem.} (\alph{subfig})}

\newcommand{\semileq}{\mathrel{(\leqslant)}}

\begin{document}

\title[Jordan--H{\"o}lder theory for complementary
polyhedra]{Jordan--H{\"o}lder theory for complementary polyhedra and moduli of
parahoric Higgs torsors}

\author{Emanuel Roth}
\email{eroth2@ed.ac.uk}
\address{School of Mathematics and Maxwell Institute\\
    University of Edinburgh\\
    James Clerk Maxwell Building, Peter Guthrie Tait Road\\
    Edinburgh EH9 3FD\\
    United Kingdom\\
    ORCID: \href{https://orcid.org/0009-0002-1112-6453}{0009-0002-1112-6453}}

\subjclass[2020]{14D20 (primary), 14D23 (secondary)}
\keywords{complementary polyhedra, good moduli
    spaces, Higgs bundles,
    Jordan--H{\"o}lder, moduli theory, parahoric torsors, principal bundles,
S-equivalence, vector bundles.}
\thanks{The author was supported by the UKRI Centre for Doctoral Training
    in \href{https://www.agq-cdt.org/}{Algebra, Geometry and Quantum
    Fields (AGQ)}, Grant Number
    EP/Y035232/1. For the purpose of open access, the author has
    applied a
    \href{https://creativecommons.org/licenses/by/4.0/}{Creative
    Commons Attribution (CC BY)} license to any Author Accepted
    Manuscript version arising from this submission. Data openly
    available in a public repository. The author declares the use of
    artificial intelligence (AI) for proofreading, formatting, and
    finding references, and the author declares that no AI was used
    in writing the mathematical content of this paper.
}

\begin{abstract}
    We extend Behrend's theory of complementary polyhedra to include
    Jordan--H{\"o}lder theory by introducing the notion of Jordan--H{\"o}lder
    facets. As a consequence, we can construct Jordan--H{\"o}lder
    parabolic subgroup schemes for reductive group schemes over a
    smooth projective curve. We use these results to provide a
    uniform approach to describing Jordan--H{\"o}lder filtrations and
    S-equivalence of moduli stacks of semistable bundles, including
    vector bundles, principal bundles, parabolic bundles, parahoric
    torsors, and parahoric Higgs torsors, following Zhang. We
    explicitly give a description of Jordan--H{\"o}lder filtrations and
    S-equivalence in the semistable locus of parahoric torsors and
    parahoric Higgs torsors, coming from their good moduli spaces, as
    constructed by Alper--Halpern-Leistner--Heinloth and R{\'e}ga. This
    application uses the Rees construction, as given by Heinloth, to
    identify filtrations of moduli stacks of parahoric torsors with
    reductions of parahoric torsors to parabolic subgroups, inducing
    Jordan--H{\"o}lder stratifications of the moduli stacks.
\end{abstract}

\maketitle

\section{Introduction}\label{sec_intro}

{
    \renewcommand{\thesubsubsection}{\thesection.\arabic{subsubsection}}
    \renewcommand{\thetheorem}{\thesection.\arabic{theorem}}

    \subsubsection{History of Jordan--H{\"o}lder
    filtrations}\label{intro_jhhistory}
    In the late 19th century, Jordan in
    \cite{jordan_traitedessubstitutionsetdesequationsalgebriques}
    asked how groups $\mathsf{G}$ can be decomposed into
    \textit{simple groups}, i.e., groups without nontrivial normal
    subgroups. Specifically, he showed for a finite permutation group
    $\mathsf{G}$ the existence of a filtration:
    \begin{equation*}
        \{e\}=\mathsf{H}_0\triangleleft\ldots\triangleleft\mathsf{H}_l=\mathsf{G}
    \end{equation*}
    \sloppy
    whose elements are successively normal subgroups
    $(\mathsf{H}_{i-1}\triangleleft\mathsf{H}_i)$, and whose quotient
    groups $(\mathsf{H}_i/\mathsf{H}_{i-1})$ are simple. Jordan
    compared two such filtrations $(\mathsf{H}_i)$ and
    $(\mathsf{K}_j)$ of $\mathsf{G}$, and showed that
    $(\mathsf{H}_i/\mathsf{H}_{i-1})$ and
    $(\mathsf{K}_j/\mathsf{K}_{j-1})$ are isomorphic up to
    permutation. H{\"o}lder in
    \cite{holder_zuruckfuhrungeinerbeliebigenalgebraischengleichung}
    found a more direct proof, removing Jordan's use
    of permutation groups and generalizing to finite groups.

    Such filtrations are now called \textit{Jordan--H{\"o}lder
    filtrations}, and have been constructed for many algebraic and
    geometric objects, such as for Noetherian, Artinian rings, and
    finite-length modules. For a
    Noetherian, Artinian object $A$ inside an \textit{Abelian
    category} $\mathcal{A}$,
    there exists a filtration of $A$, called a \textit{Jordan--H{\"o}lder
    filtration}:
    \begin{equation*}
        0=A_0\subset\ldots\subset A_l=A
    \end{equation*}
    whose quotients $B_i:=A_i/A_{i-1}$ are simple objects, see e.g.
    \cite[{\href{https://stacks.math.columbia.edu/tag/0FCD}{0FCD}}]{stacks-project}.
    For two such filtrations, the \textit{(associated) gradeds}
    $\bigoplus_{i=1}^lB_i$ are isomorphic.

    Jordan--H{\"o}lder filtrations are
    important tools in the study of moduli problems in algebraic
    geometry. Using Jordan--H{\"o}lder filtrations for Abelian
    categories, Seshadri proved the following, using
    \textit{slope-stability conditions}, as defined in
    \cite[\S3]{seshadri_spaceofunitaryvectorbundles} and
    \cite[Definitions 1.2.11 and
    1.2.12]{huybrechts-lehn_geometryofmodulispaces}.

    \begin{theorem}[(Jordan--H{\"o}lder theory for vector bundles
                {\cite[Proposition
        3.1]{seshadri_spaceofunitaryvectorbundles}})]\label{thm_jhvectorbundles}
        For a slope-semistable vector bundle $E$ over a complex
        smooth projective curve $X$, there exists a filtration of subbundles:
        \begin{equation*}
            0=E_0\subset\ldots\subset E_l=E
        \end{equation*}
        whose quotients $F_i=E_{i}/E_{i-1}$ are all slope-stable, and
        all have the same slope as $E$. Furthermore, two such
        filtrations have isomorphic graded bundles $\bigoplus_{i=1}^lF_i$.
    \end{theorem}

    Two slope-semistable vector bundles are \textit{S-equivalent}
    when the graded bundles of their Jordan--H{\"o}lder filtrations
    are isomorphic. Seshadri applied S-equivalence to the moduli
    theory of vector bundles in \cite[Theorem
    8.1]{seshadri_spaceofunitaryvectorbundles}, by showing that the set of
    S-equivalence classes of slope-semistable bundles is naturally
    equipped with a structure of a normal projective complex
    variety. Ramanathan generalized these results to principal
    bundles in \cite[3.12
    Proposition]{ramanathan_moduliforprincipalbundlesI} and \cite[5.9
    Theorem]{ramanathan_moduliforprincipalbundlesII}, where the latter
    generalizes Seshadri's moduli space of S-equivalence classes in
    \cite[Theorem 8.1]{seshadri_spaceofunitaryvectorbundles}.

    \subsubsection{Motivation for this paper}\label{intro_motivation}
    We are interested in generalizing Jordan--H{\"o}lder filtrations,
    stratifications, and
    S-equivalence results to more general moduli problems, such as
    that of \textit{parabolic bundles}, \textit{parahoric torsors},
    and \textit{parahoric Higgs torsors}.
    We take inspiration from Behrend's paper on \textit{complementary
    polyhedra}
    \cite{behrend_semi-stabilityofreductivegroupschemesovercurves},
    which provided a uniform approach to proving the existence and
    uniqueness of \textit{Harder--Narasimhan filtrations} of vector
    bundles (or \textit{canonical reductions} of principal bundles).

    \subsubsection{Complementary polyhedra}\label{intro_complementarypolyhedra}
    \textit{Complementary polyhedra}, denoted by $d$, are maps in the
    theory of root
    systems (Definition
    \ref{def_complementarypolyhedra}) introduced by Behrend in
    \cite[Definition
    2.1]{behrend_semi-stabilityofreductivegroupschemesovercurves},
    whose stability conditions
    (Definition \ref{def_stabilityofcomplementarypolyhedra}) can
    store the stability profile of vector bundles and principal
    bundles. We present a Jordan--H{\"o}lder theory for complementary
    polyhedra, and extend their notions of (semi)-stability from
    \cite{behrend_phdthesis,behrend_semi-stabilityofreductivegroupschemesovercurves},
    to \textit{polystability} (Definition
        \ref{def_stabilityofcomplementarypolyhedra}
    \ref{def_stabilityofcomplementarypolyhedra3}). These stability
    conditions can be viewed visually through the convex hulls $F$ of
    complementary polyhedra, depending on whether $0$ lies in the
    interior of $F$, \textit{relative interior} of $F$
    (Definition \ref{def_interior}), or simply $F$ itself. This gives
    the following correspondences and implications (Lemmas
        \ref{lem_stabilityofcomplementarypolyhedra} and
    \ref{lem_polystabilityequivalencepolyhedra}):
    \begin{align*}\label{eq_intro}
        \tag{1.3}
        \mathrm{int}(F)\subset\mathrm{relint}(F)\subset F &  &
        \text{stability}\Rightarrow\text{polystability}\Rightarrow\text{semistability}
    \end{align*}
    To prove the existence and uniqueness of Harder--Narasimhan
    filtrations of vector bundles (or canonical reductions of
    principal bundles), Behrend in
    \cite[Corollary
    3.14]{behrend_semi-stabilityofreductivegroupschemesovercurves}
    determines a unique \textit{special facet} of complementary
    polyhedra (Definition \ref{def_specialfacet}) with analogous
    properties. Thus, to construct Jordan--H{\"o}lder filtrations or
    reductions, we similarly
    introduce \textit{Jordan--H{\"o}lder facets} (Definition
    \ref{def_jhfacet}), prove they exist (Theorem
    \ref{thm_jhfacetexistence}), and show that they are unique up to
    its graded or \textit{Levi-factor} (Theorem
    \ref{thm_jhfacetuniqueness}), analogous to the properties in
    \ref{intro_jhhistory}.

    \subsubsection{Analogy to geometric invariant theory (GIT)}

    Complementary polyhedra and their stability conditions are
    analogous to \textit{GIT-stability conditions}. For a variety $X$
    with an action by a reductive group $\mathsf{G}$, GIT-stability
    conditions at a point $x\in X$ can be reformulated in terms of
    actions of cocharacters
    $\lambda:\mathbb{G}_m\rightarrow\mathsf{G}$ inside orbits
    $\mathsf{G}\cdot x$ through \textit{Hilbert--Mumford weights}
    $\mu(\lambda,x)$, leading to the \textit{Hilbert--Mumford
    criterion} in \cite[Theorem 2.1]{mumford-fogarty-kirwan_GIT}. For
    complementary polyhedra $d$, the analogous notion is its
    \textit{degree} (Definition
        \ref{def_stabilityofcomplementarypolyhedra}
    \ref{def_stabilityofcomplementarypolyhedra1}), built into the
    definition of stability (Definition
    \ref{def_stabilityofcomplementarypolyhedra}). By asking how
    GIT-instability can be measured, Kempf in
    \cite{kempf_instabilityininvarianttheory} showed that for
    non-GIT-semistable points $x\in X\setminus X^{\mathrm{sst}}$,
    there exists unique \textit{maximally destabilizing} cocharacter
    $\lambda$ with respect to $\mu(\lambda,x)$ (up to normalization).
    By analogy, for non-semistable $d$, Behrend's \textit{special
    facet} (Definition \ref{def_specialfacet}) plays the same role in
    measuring the instability of $d$, which induces canonical
    reductions of principal bundles in
    \cite[\S8]{behrend_semi-stabilityofreductivegroupschemesovercurves}.
    These notions align when constructing moduli spaces of bundles using GIT.

    Through Jordan--H{\"o}lder facets (Definition \ref{def_jhfacet}),
    we aim to extend these analogies to when $x\in
    X^{\mathrm{sst}}\setminus X^{\mathrm{st}}$ is GIT-semistable, but
    non-GIT-stable. Recall that a cocharacter $\lambda$
    \textit{degenerates} or \textit{specializes} $x$ to an
    \textit{associated graded point} $x_0\in X$. When $\lambda$
    degenerates $x$ to a GIT-polystable $x_0\in X^{\mathrm{pst}}$,
    Woodward in \cite[Theorem 5.5.9]{woodward_momentmapsGIT} proves
    Jordan--H{\"o}lder properties for $\lambda$, analogous to the
    properties in \ref{intro_jhhistory}. By analogy,
    Jordan--H{\"o}lder facets should identify Jordan--H{\"o}lder
    filtrations of moduli problems, such as in Theorem
    \ref{thm_jhvectorbundles}, where in GIT construction of the
    moduli space of vector bundles, S-equivalence in GIT corresponds
    to S-equivalence of vector bundles.

    Moreover, expecting the correspondences in (\ref{eq_intro}) can
    be justified through GIT, since for a vector space $V$ with a
    $\mathsf{G}$-representation, complementary polyhedra $d$ play the
    role of \textit{weight polytopes} in the sense of Keller in
    \cite[Definition 5.1]{keller_phdthesis}, so (\ref{eq_intro}) is
    analogous to \cite[Proposition 5.2]{keller_phdthesis}, which
    comes directly from the Hilbert--Mumford criterion.

    \subsubsection{Applications and results}\label{intro_applications}
    By associating complementary polyhedra to \textit{reductive group
    schemes} over a curve, such as to \textit{automorphism group schemes} of
    bundles, seen in \cite[Proposition
    6.6]{behrend_semi-stabilityofreductivegroupschemesovercurves},
    many authors have constructed complementary polyhedra for
    various moduli problems:
    \begin{enumerate}[label=(\alph*)]
        \item\label{intro_applications1} Vector bundles and principal
            bundles, by Behrend in
            \cite[\S8]{behrend_semi-stabilityofreductivegroupschemesovercurves}.

        \item\label{intro_applications2} Arakelov bundles, by Stuhler
            in \cite{stuhler_canonicalparabolicsubgroupsofarakelovgroupschemes}.

        \item\label{intro_applications3} Principal parabolic bundles,
            by Heinloth and Schmitt in
            \cite[\S4]{heinloth-schmitt_cohomologyringsofmodulistacksprincbundles}.

        \item\label{intro_applications4} Higgs bundles, by Wi{\ss}dorf in
            \cite[\S3.5]{wissdorf_phdthesis}.

        \item\label{intro_applications5} Parahoric torsors, by Heinloth in
            \cite[\S3.F]{heinloth_hilbertmumfordstability}.

        \item\label{intro_applications6} Parahoric Higgs torsors, by
            R{\'e}ga in \cite[\S5.2.2]{rega_phdthesis}.
    \end{enumerate}
    They have identified stability conditions of the moduli problems
    with that of complementary polyhedra, allowing them to study
    their Harder--Narasimhan filtrations, canonical reductions, and in
    more modern language, to study the \textit{$\Theta$-stratifications} of
    their moduli stacks, in the sense of Halpern-Leistner,
    Alper--Halpern-Leistner--Heinloth in
    \cite{danielhalpernleistner_onthestructureofinstabilityinmodulitheory,alperDHLheinloth_existenceofmodulispacesforalgebraicstacks}.

    Using our Jordan--H{\"o}lder theory for complementary polyhedra, we
    can develop the results we wished for in \ref{intro_motivation}. We
    explain this explicitly for principal bundles (Theorem
    \ref{thm_jhprincipal}), recovering Theorem
    \ref{thm_jhvectorbundles}, and then we explain the generalization
    to \textit{parahoric torsors}
    (Definition \ref{def_parahorictorsor}), as studied by Heinloth,
    Pappas--Rapoport, Balaji--Seshadri, et al.\ in
    \cite{heinloth_uniformizationofG-bundles,pappas-rapoport_someqnsaboutG-bundles,balaji-sesh_parahorictorsors,heinloth_hilbertmumfordstability}.

    \begin{mytheorem}[{\hyperref[thm_jhparahoric]{\ref*{thm_jhparahoric}}}]
        Let $\underline{\chi}$ be an admissible stability parameter in the
        sense of \cite[\S3.F]{heinloth_hilbertmumfordstability}, and let
        $\xi$ be an $\mathcal{L}_{\det,\underline{\chi}}$-semistable parahoric
        $\mathcal{G}$-torsor. There exists a reduction
        $s:X\rightarrow\xi/\mathcal{G}_{\mathsf{P}}$ to a parabolic subgroup
        $\mathcal{G}_{\mathsf{P}}$ of $\mathcal{G}$ such that:
        \begin{enumerate}[label=(\roman*)]
                \sloppy
            \item The reduction $s$ is
                \textit{admissible}, i.e., for every vertex
                $\mathfrak{v}\in\pi_0(t(\mathrm{Ad}(s^*\xi)))$, we have
                $n_{\underline{\chi}}(\mathrm{Ad}(s^*\xi),\mathfrak{v})=0$,
                using (\ref{eq_numericalinvariant}).

            \item The extension to the
                Levi-factor $(s^*\xi)(\mathcal{L}_{\mathsf{P}})$, through the
                quotient
                $\mathcal{G}_{\mathsf{P}}\rightarrow\mathcal{L}_{\mathsf{P}}$,
                is
                $\mathcal{L}_{\det,\underline{\chi}^{\mathcal{L}_\mathsf{P}}}$-stable,
                with
                $\underline{\chi}^{\mathcal{L}_\mathsf{P}}$ induced by the
                isomorphism
                $\mathcal{L}_\mathsf{P}\cong\mathcal{L}(\boldsymbol{\lambda}_\eta)$
                as the
                restriction of $\underline{\chi}$.
        \end{enumerate}
        These reductions have the same extension
        $(s^*\xi)(\mathcal{L}_{\mathsf{P}})$ to the Levi-factor, up to an
        isomorphism.
    \end{mytheorem}

    We then extend these Jordan--H{\"o}lder results to
    \textit{parahoric Higgs torsors} (Definition
        \ref{def_parahorichiggs} \ref{def_parahorichiggs1},
    Theorem \ref{thm_jhparahorichiggs}), as
    studied by Yun, Baraglia--Kamgarpour--Varma,
    Kydonakis--Sun--Zhao, R{\'e}ga, et al.\ in
    \cite{yun_globalspringertheory,baraglia-kamgarpour-varma_completeintegrabilityoftheparahorichitchinsystem,kydonakis-sun-zhao_logahorichiggstorsors,rega_phdthesis}.
    The two parahoric cases are new to the literature and
    generalize the Jordan--H{\"o}lder filtrations for parabolic bundles by
    Henke in \cite[\S7.2]{henke_phdthesis}, and for Higgs bundles by Gra{\~n}a
    Otero in \cite{granaotero_jhreductionsforprincipalhiggsbundlesoncurves}.

    Through Theorems \ref{thm_jhstratificationprincipal},
    \ref{thm_jhstratificationparahoric}, and
    \ref{thm_jhstratificationparahorichiggs}, we connect our
    Jordan--H{\"o}lder filtrations of bundles to the
    \textit{good moduli spaces}
    of Alper, Alper--Halpern-Leistner--Heinloth in
    \cite{alper_goodmodulispacesforartinstacks,alperDHLheinloth_existenceofmodulispacesforalgebraicstacks},
    and to the \textit{(quasi)-S-equivalence} of algebraic stacks of Zhang in
    \cite[Definition 2.3]{zhang_Sequivalenceforalgebraicstacks}. The
    parahoric version states the following.

    \begin{mytheorem}[{\hyperref[thm_jhstratificationparahoric]{\ref*{thm_jhstratificationparahoric}}}]
        Let $K$ be algebraically closed of characteristic $0$, and let
        $\underline{\chi}$ be an admissible stability parameter in the
        sense of \cite[\S3.F]{heinloth_hilbertmumfordstability}. The
        Jordan--H{\"o}lder stratification of the moduli stack of
        parahoric $\mathcal{G}$-torsors
        $\mathcal{M}\textsc{Bun}_{\mathcal{G}}$,
        induced by Theorem \ref{thm_jhparahoric}, consists of locally
        closed substacks. Moreover, the quasi-S-equivalence of the
        good moduli space
        $\pi:(\mathcal{M}\textsc{Bun}_{\mathcal{G}})^{d^{\underline{\chi}}-\mathrm{sst}}\rightarrow
        (M\textsc{Bun}_{\mathcal{G}})^{d^{\underline{\chi}}-\mathrm{sst}}$
        from \cite[Theorem
        8.1]{alperDHLheinloth_existenceofmodulispacesforalgebraicstacks},
        in the sense of \cite[Definition
        2.3]{zhang_Sequivalenceforalgebraicstacks}, is the same as
        S-equivalence induced by Theorem \ref{thm_jhparahoric}.
    \end{mytheorem}

    We
    achieve this by introducing \textit{bundle moduli problems} (Definition
    \ref{def_bundlemoduliproblem}). In explicit cases, we use the
    \textit{Rees construction} to relate \textit{filtrations}
    $f:\Theta:=[\mathbb{A}^1/\mathbb{G}_m]\rightarrow\mathcal{M}$ (Definition
    \ref{def_filtration}) of moduli stacks of
    bundles to filtrations or parabolic reductions of
    bundles. We make the Rees construction explicit for
    parahoric Higgs torsors (Lemma
    \ref{lem_reesconstructionparahorichiggs}), extending Heinloth's
    explanation in \cite[\S3.C]{heinloth_hilbertmumfordstability} for
    parahoric torsors (Lemma \ref{lem_reesconstructionparahoric}).

    \subsubsection{Structure of the paper}\label{intro_structure}

    Section \ref{sec_jhfacet},
    we recall Behrend's complementary
    polyhedra, their stability conditions and their special facets,
    where in Subsection \ref{subsec_specialfacets}, we define Jordan--H{\"o}lder
    facets and develop Jordan--H{\"o}lder theory for complementary
    polyhedra. In Section
    \ref{sec_jhtheoryreductivegroupschemes}, we use complementary
    polyhedra to study the Jordan--H{\"o}lder theory of reductive group
    schemes, in order to apply these results to bundle moduli problems.
    In Section \ref{sec_modulistacks}, we construct
    Jordan--H{\"o}lder filtrations and stratifications for moduli stacks,
    with respect to
    complementary polyhedra, packaging these notions together as
    bundle moduli problems in Subsection
    \ref{subsec_bundlemoduliproblem}.
    Section
    \ref{sec_applicationstobundles} gives explicit applications of
    our results, such as to the moduli problems of principal bundles
    (Subsection \ref{subsec_principal}) parahoric torsors
    (Subsection \ref{subsec_parahoric}), and parahoric Higgs
    torsors (Subsection \ref{subsec_parahorichiggs}).
}

\section{Jordan--H{\"o}lder facets of complementary
polyhedra}\label{sec_jhfacet}

We review the relevant material from
\cite{behrend_semi-stabilityofreductivegroupschemesovercurves} to
define \textit{complementary polyhedra}, and extend Behrend's notion
of special facets to \textit{Jordan--H{\"o}lder} facets.

\subsection{Root systems}

We briefly review the notion of root systems:
Let $(V,\langle\_,\_\rangle)$ be a \textit{Euclidean space}, i.e.,
let $V$ be a finite-dimensional $\mathbb{R}$-vector space with an
inner-product $\langle\_,\_\rangle$, and let $\Phi$ be a \textit{root system}
in $(V,\langle\_,\_\rangle)$, as defined in
\cite[\S9.2]{humphreys_introductiontoliealgebrasandrepresenationtheory}.
The \textit{reflections} $s_\alpha\in\mathrm{Iso}_{\mathbb{R}}(V)$ of roots
$\alpha\in\Phi$, as defined in
\cite[\S9.1]{humphreys_introductiontoliealgebrasandrepresenationtheory},
generate a finite subgroup of $\mathrm{Iso}_{\mathbb{R}}(V)$, called the
\textit{Weyl group} $\mathfrak{W}_\Phi$.

In
\cite[\S10.1]{humphreys_introductiontoliealgebrasandrepresenationtheory},
\textit{bases} or subsets of \textit{simple roots}
$\Delta\subset\Phi$ are defined, which are $\mathbb{R}$-bases of $V$,
and in \cite[\S10.1,
Theorem]{humphreys_introductiontoliealgebrasandrepresenationtheory},
they are proven to exist. For the hyperplanes:
\begin{equation*}
    \mathrm{Hy}_\alpha:=\{\mathbf{v}\in
    V\ \vert\ \langle\alpha,\mathbf{v}\rangle=0\}
\end{equation*}
of roots $\alpha\in\Phi$, the \textit{(open) Weyl chambers} of $\Phi$
are connected components of
$V\setminus\bigcup_{\alpha\in\Phi}\mathrm{Hy}_\alpha$. It is easy to
see that Weyl chambers are of the form:
\begin{equation*}
    \{\mathbf{v}\in
    V\ \vert\ \forall\alpha\in\Delta:\langle\alpha,\mathbf{v}\rangle>0\}
\end{equation*}
for a choice of simple roots $\Delta\subset\Phi$. The Weyl group
$\mathfrak{W}_\Phi$ acts transitively on the set of Weyl chambers, and on
the set of simple roots, as shown in \cite[\S10.3,
Theorem]{humphreys_introductiontoliealgebrasandrepresenationtheory}.

\subsection{Definition of complementary
polyhedra}\label{subsec_complementarypolyhedra}

We partition $V=\bigsqcup_{P\in F(\Phi)}P$ into \textit{facets}
$P$, distinguished by the fact that two vectors
$\mathbf{v},\mathbf{w}\in V$ lie in
the same facet if for all roots $\alpha\in\Phi$, one of the
following is true:
\begin{enumerate}[label=(\roman*)]
    \item\label{def_facet1} $\mathbf{v},\mathbf{w}\in L_\alpha:=\{\mathbf{v}\in
        V\ \vert\ \langle\alpha,\mathbf{v}\rangle=0\}$.

    \item\label{def_facet2} $\mathbf{v},\mathbf{w}\notin L_\alpha$
        and both are in the
        same connected component of $V\setminus L_\alpha$.
\end{enumerate}

Condition \ref{def_facet2} can be interpreted as $\mathbf{v}$
and $\mathbf{w}$ lying on the same ``side'' away from $L_\alpha$. There are
finitely many facets, including Weyl chambers. Behrend
partially orders the set of facets of $\Phi$ by:
\begin{equation*}
    P\geqslant Q:\Leftrightarrow P\subset\mathrm{cl}(Q)
\end{equation*}
where $\mathrm{cl}(Q)$ denotes the \textit{closure} of $Q$ in $V$.

\begin{definition}[(Vertices {\cite[Definition
    1.2]{behrend_semi-stabilityofreductivegroupschemesovercurves}})]\label{def_vertices}
    \leavevmode
    \begin{enumerate}[label=(\alph*)]
        \item\label{def_vertices1} A \textit{fundamental dominant weight}
            $\lambda$ is an element of $V$ such that there exists simple
            roots $\Delta\subset\Phi$ and $\alpha\in\Delta$ with:
            \begin{align*}
                \langle\alpha,\lambda\rangle=1 &  &
                \forall\beta\in\Delta\setminus\{\alpha\}:\langle\beta,\lambda\rangle=0
            \end{align*}

        \item\label{def_vertices2} For a facet $P$, a fundamental
            dominant weight $\lambda$ is a \textit{vertex} of $P$ if
            $\lambda$ is in the closure $\mathrm{cl}(P)$ of $P$. We denote
            the set of vertices of $P$ by $\mathrm{vert}(P)$.
    \end{enumerate}
\end{definition}

\begin{example}[(Facets of $A_2$)]\label{ex_facetsA2}
    We depict the root system $A_2$ in the Euclidean space
    $(\mathbb{R}^2,\langle\_,\_\rangle)$:
    \begin{equation*}
        \begin{tikzpicture}[rootsys]
            %draw Weyl chambers
            \draw[blueweylchamber] (30:1.35) -- (0,0) -- (90:1.35) -- cycle;
            \draw[pinkweylchamber] (90:1.35) -- (0,0) -- (150:1.35) -- cycle;

            %draw facet P
            \draw[facet] (0,0) -- (90:1.35)
            node[anchor=south,xshift=2pt]{$P$};

            %label blue chamber
            \node[text=blue!80!black] at ($(0,0)!1!(60:1.35)$) {$\mathfrak{c}$};

            %label pink chamber
            \node[text=pink!70!black,yshift=2pt] at
            ($(0,0)!1!(120:1.34)$) {$\mathfrak{d}$};

            %draw root walls
            \foreach \angle in {30,150,210,330}
            {
                \draw[rootwall] (0,0) -- (\angle:1.35);
            }
            \draw[rootwall] (0,0) -- (270:1);

            %draw root vectors
            \draw[root] (0,0) -- (0:1.0)
            node[anchor=west,xshift=2pt]{$\alpha$};
            \foreach \angle in {60,120,180,240,300}
            {
                \draw[root] (0,0) -- (\angle:1.0);
            }

            %draw dominant weights
            \fill[purple!80] (30:2.0/3.0) circle (1.3pt)
            node[anchor=west,xshift=2pt,yshift=-3.1pt]{$\eta$};
            \fill[purple!80] (90:2.0/3.0) circle (1.3pt)
            node[anchor=west,xshift=0.3pt]{$\mu$};
            \fill[purple!80] (150:2.0/3.0) circle (1.3pt)
            node[anchor=east,xshift=-2pt,yshift=-2pt]{$\lambda$};
            \foreach \angle in {210,270,330}
            {
                \fill[purple!80] (\angle:2.0/3.0) circle (1.3pt);
            }

            %draw origin
            \fill[red!80] (0,0) circle (1.3pt);

            %label picture
            \node at (0,-1.15) {\footnotesize
            \textbf{Figure~\thetheorem.}\relax};
        \end{tikzpicture}
    \end{equation*}
    The facets $\mathfrak{c}$, $\mathfrak{d}$, $P$, and $\{0\}$ have
    the ordering
    $\{0\}\geqslant P\geqslant \mathfrak{c},\mathfrak{d}$, where $\mathfrak{c}$
    and $\mathfrak{d}$ are
    Weyl chambers. The six fundamental weights are depicted as the six
    nonzero dots in Figure \ref{ex_facetsA2}, such that:
    \begin{align*}
        \mathrm{vert}(\mathfrak{c})=\{\eta,\mu\} &  &
        \mathrm{vert}(\mathfrak{d})=\{\mu,\lambda\} &  &
        \mathrm{vert}(P)=\{\mu\} &  &
        \mathrm{vert}(\{0\})=\emptyset
    \end{align*}
\end{example}

\begin{definition}[(Complementary polyhedra {\cite[Definition
    2.1]{behrend_semi-stabilityofreductivegroupschemesovercurves}})]\label{def_complementarypolyhedra}
    A map from the set of Weyl chambers $d:\mathfrak{W}_\Phi\rightarrow V^*$
    is a \textit{complementary polyhedron} of $\Phi$ if:
    \begin{enumerate}[label=(A\arabic*)]
        \item\label{def_complementarypolyhedra1} For all
            $\mathfrak{c},\mathfrak{d}\in\mathfrak{W}_\Phi$ with a common vertex
            $\lambda\in\mathrm{vert}(\mathfrak{c})\cap\mathrm{vert}(\mathfrak{d})$,
            we have:
            \begin{equation*}
                d(\mathfrak{c})(\lambda)=d(\mathfrak{d})(\lambda)
            \end{equation*}

        \item\label{def_complementarypolyhedra2} For all
            $\mathfrak{c},\mathfrak{d}\in\mathfrak{W}_\Phi$, and
            $\alpha\in\Phi$, such that
            $\mathfrak{c}$ is \textit{$\alpha$-conjugate} to
            $\mathfrak{d}$, i.e.,
            $\mathfrak{c}$ is positive with respect to $\alpha$, and
            $\mathfrak{d}$ is
            negative with respect to $\alpha$, then:
            \begin{equation*}
                d(\mathfrak{c})(\alpha)\leqslant d(\mathfrak{d})(\alpha)
            \end{equation*}
    \end{enumerate}
    We denote
    $F(P):=\mathrm{conv}\{d(\mathfrak{c})\ \vert\ P\geqslant\mathfrak{c}\}$
    for a
    facet $P$, and call $F:=F(\{0\})$ the \textit{convex hull}.
\end{definition}

Since complementary polyhedra are valued in the dual-space $V^*$,
recall that for a root $\alpha\in\Phi$, its \textit{dual-root} in the
\textit{dual-Euclidean space} $(V^*,(\_,\_))$ is given by:
\begin{align*}
    \widecheck{\alpha}:V\rightarrow\mathbb{R} &  & \mathbf{v}\mapsto2\langle
    \mathbf{v},\alpha\rangle/\langle\alpha,\alpha\rangle
\end{align*}
forming a root system $\widecheck{\Phi}$ in $(V^*,(\_,\_))$, as seen
in \cite[\S III.
9.4]{humphreys_introductiontoliealgebrasandrepresenationtheory}.
Through this identification, every fundamental dominant weight
$\lambda$ induces a \textit{dual-fundamental dominant weight}
$\widecheck{\lambda}$ in $V^*$, and for
a facet $P$ of $\Phi$, this induces a \textit{dual-facet}
$\widecheck{P}$ of $\widecheck{\Phi}$, giving a bijection of facets
$F(\Phi)\cong F(\widecheck{\Phi})$. Using this, we give two examples
of complementary polyhedra for the root system $A_2$.

\begin{example}[(Complementary polyhedra of
    $A_2$)]\label{ex_complementarypolyhedra}
    For the root system $A_2$ in $(\mathbb{R}^2,\langle\_,\_\rangle)$, we
    label the Weyl chamber containing the root $\gamma\in A_2$ by
    $\mathfrak{c}_\gamma$. We choose and label simple roots
    $\Delta=\{\alpha,\beta\}$, then define
    $d_a:\mathfrak{W}_{A_2}\rightarrow\mathbb{R}^{2*}$:
    \begin{align*}
        d_a(\mathfrak{c}_{-\beta})        =d_a(\mathfrak{c}_{\alpha}) &
        =\widecheck{\beta}   & d_a(\mathfrak{c}_{-\alpha-\beta})
        =d_a(\mathfrak{c}_{-\alpha}) & =\widecheck{\alpha+\beta}     \\
        d_a(\mathfrak{c}_{\alpha+\beta})                        &
        =\widecheck{\beta}/2 & d_a(\mathfrak{c}_{\beta})
        & =(\widecheck{\alpha+\beta})/2
    \end{align*}
    We also define $d_b:\mathfrak{W}_{A_2}\rightarrow\mathbb{R}^{2*}$:
    \begin{align*}
        d_b(\mathfrak{c}_{-\beta}) &
        =d_b(\mathfrak{c}_{\alpha})=d_b(\mathfrak{c}_{\alpha+\beta})=-\widecheck{\alpha}
        & d_b(\mathfrak{c}_{-\alpha-\beta}) &
        =d_b(\mathfrak{c}_{-\alpha})=d_b(\mathfrak{c}_{\beta})=\widecheck{\alpha}
    \end{align*}
    It can be verified that $d_a$ and $d_b$ are complementary polyhedra
    with the convex hulls:
    \begin{equation*}
        \begin{tikzpicture}[rootsys]
            %draw trapezium
            \filldraw[polygon] (60:0.5) -- (60:1) -- (120:1) --
            (120:0.5) -- cycle;

            %draw coroot vectors
            \draw[root] (0,0) -- (0:1.0)
            node[anchor=west,xshift=2pt]{$\widecheck{\alpha}$};
            \draw[root] (0,0) -- (60:1.0)
            node[pos=0.94,anchor=south
            west,xshift=1pt,yshift=1pt]{$\widecheck{\alpha+\beta}$};
            \draw[root] (0,0) -- (120:1.0)
            node[pos=0.94,anchor=south
            east,xshift=-1pt,yshift=1pt]{$\widecheck{\beta}$};
            \draw[root] (4.0,0) -- ($(4.0,0)+(0:1)$)
            node[anchor=west,xshift=2pt] {$\widecheck{\alpha}$};
            \foreach \angle in {180,240,300}
            {
                \draw[root] (0,0) -- (\angle:1.0);
            }
            \foreach \angle in {60,120,180,240,300}
            {
                \draw[root] (4.0,0) -- ($(4.0,0)+(\angle:1.0)$);
            }

            %draw coroot walls
            \foreach \angle in {30,150,210,330}
            {
                \draw[rootwall] (4.0,0) -- ($(4.0,0)+(\angle:1.35)$);
                \draw[rootwall] (0,0) -- (\angle:1.35);
            }
            \draw[rootwall] (4.0,0) -- ($(4.0,0)+(90:1)$);
            \draw[rootwall] (0,0) -- (90:1);
            \draw[rootwall] (4.0,0) -- ($(4.0,0)+(270:1)$);
            \draw[rootwall] (0,0) -- (270:1);

            %draw interval
            \draw[thickpolygon] ($(4.0,0)+(0:1.0)$) --
            ($(4.0,0)+(0:-1.0)$) -- cycle;
            \fill[blue!80] ($(4.0,0)+(0:1.0)$) circle (1.3pt);
            \fill[blue!80] ($(4.0,0)+(0:-1.0)$) circle (1.3pt);

            %draw origins
            \fill[red!80] (0,0) circle (1.3pt);
            \fill[red!80] (4.0,0) circle (1.3pt);

            %first label
            \stepcounter{subfig}
            \node at (0,-1.15) {\footnotesize \protect\thesubfig};
            \node at (0,-1.35) {\footnotesize Convex hull $F_a$ of $d_a$};

            %second label
            \stepcounter{subfig}
            \node at (4.0,-1.15) {\footnotesize \protect\thesubfig};
            \node at (4.0,-1.35) {\footnotesize Convex hull $F_b$ of $d_b$};
        \end{tikzpicture}
    \end{equation*}
\end{example}

\subsection{Stability of complementary
polyhedra}\label{subsec_stabilitycomplementarypolyhedra}

In \cite{behrend_phdthesis,
behrend_semi-stabilityofreductivegroupschemesovercurves}, Behrend
introduces notions of \textit{(semi)-stability} for complementary
polyhedra, which we extend naturally to \textit{polystability}.
Recall that Behrend defines projections of complementary polyhedra:
For a facet $P$, let $\mathrm{pr}_P:V\rightarrow P^\perp$ denote the projection
with respect to the decomposition
$V=\mathrm{span}_{\mathbb{R}}(P)\oplus P^\perp$.
Through the isomorphism $V\cong V^*$ induced by
$\langle\_,\_\rangle$, we obtain a decomposition
$V^*=\mathrm{span}_{\mathbb{R}}(\widecheck{P})\oplus\widecheck{P}^\perp$ that
induces the \textit{dual-projection} $\widecheck{\mathrm{pr}}_P:V^*\rightarrow
\widecheck{P}^\perp$. As explained in \cite[Lemma 1.10 and Definition
2.3]{behrend_semi-stabilityofreductivegroupschemesovercurves}, $\Phi$
can be projected to $P^\perp$ to obtain a root system $\Phi_P$ of the
Euclidean subspace $(P^\perp,\langle\_,\_\rangle_P)$ of
$(V,\langle\_,\_\rangle)$, with the complementary polyhedron
$d_P:\mathfrak{W}_{\Phi_P}\rightarrow\widecheck{P}^\perp$ and convex hull
$F_P=\widecheck{\mathrm{pr}}_P(F(P))$. We call $d_P$ the
\textit{Levi-factor} of $d$.

\begin{definition}[(Degree and stability {\cite[Remark
                5.3.4]{behrend_phdthesis}\label{def_stabilityofcomplementarypolyhedra};\cite[Definitions
    3.1 and 3.3]{behrend_semi-stabilityofreductivegroupschemesovercurves}})]
    Let $d:\mathfrak{W}_\Phi\rightarrow V^*$ be a complementary polyhedron.
    \begin{enumerate}[label=(\alph*)]
        \item\label{def_stabilityofcomplementarypolyhedra1}
            For a facet $P$, let:
            \begin{equation*}
                R(P):=\{\alpha\in\Phi\ \vert\ \forall
                \lambda\in\mathrm{vert}(P):\langle\alpha,\lambda\rangle\geqslant0\}
            \end{equation*}
            be the \textit{parabolic subset}, then for any Weyl chamber
            with $P\geqslant\mathfrak{c}$, we define the \textit{degree}:
            \begin{equation*}
                \deg(P):=\sum_{\alpha\in R(P)}d(\mathfrak{c})(\alpha)
            \end{equation*}

        \item\label{def_stabilityofcomplementarypolyhedra2} $d$ is
            \textit{(semi)-stable} if for all nonzero facets
            $P\neq\{0\}$, we have
            $\deg(P)\semileq 0$, i.e., $\deg(P)\leqslant0$ for
            semistability and $\deg(P)<0$ for stability.

        \item\label{def_stabilityofcomplementarypolyhedra3} $d$ is
            \textit{polystable} if for all facets $P$ with $\deg(P)=0$, we have:
            \begin{enumerate}[label=(P\arabic*)]
                \item\label{def_stabilityofcomplementarypolyhedra31} The
                    Levi-factor
                    $d_P:\mathfrak{W}_{\Phi_P}\rightarrow\widecheck{P}^\perp$
                    is a semistable complementary polyhedron of $\Phi_P$.
                \item\label{def_stabilityofcomplementarypolyhedra32}
                    $\deg(-P)=0$ for the facet $-P$ opposite to $P$.
            \end{enumerate}
    \end{enumerate}
\end{definition}

\begin{notation}\label{rem_semistabilitynotation}
    As in Definition \ref{def_stabilityofcomplementarypolyhedra}
    \ref{def_stabilityofcomplementarypolyhedra2}, we write
    \textit{(semi)-stability} and $\semileq$ to account for both cases, i.e.,
    semistability means $\leqslant$ and stability means $<$.
\end{notation}

Definition \ref{def_stabilityofcomplementarypolyhedra}
\ref{def_stabilityofcomplementarypolyhedra1} does not depend on the
choice of Weyl chamber $\mathfrak{c}$, due to
\ref{def_complementarypolyhedra1}. Our definition for polystability
is analogous to \textit{slope polystability} in \cite[Definition
7.2.1 (d)]{woodward_momentmapsGIT}, formulated with Hilbert--Mumford weights.

Behrend constructs an equivalent
characterization of (semi)-stability through convex hulls, which
requires looking at their \textit{interiors} in $V^*$.

\begin{lemma}[(Equivalent characterization of (semi)-stability
            {\cite[Remark
                5.3.4]{behrend_phdthesis};\cite[Proposition
    3.2]{behrend_semi-stabilityofreductivegroupschemesovercurves}})]\label{lem_stabilityofcomplementarypolyhedra}
    Let $d:\mathfrak{W}_\Phi\rightarrow V^*$ be a complementary polyhedron
    with convex hull $F$.
    \begin{enumerate}[label=(\roman*)]
        \item\label{lem_stabilityofcomplementarypolyhedra1} $d$ is
            semistable if and only if $F$ contains $0$.

        \item\label{lem_stabilityofcomplementarypolyhedra2} $d$ is stable
            if and only if $0\in \mathrm{int}(F)$.
    \end{enumerate}
\end{lemma}

In Example \ref{ex_complementarypolyhedra}, Lemma
\ref{lem_stabilityofcomplementarypolyhedra} implies that $d_a$ is
neither stable nor semistable, and implies that $d_b$ is semistable
but non-stable. Once we have constructed Jordan--H{\"o}lder theory for
complementary
polyhedra, we will prove a variant of Lemma
\ref{lem_stabilityofcomplementarypolyhedra} for polystability, which
shows how polystability sits between semistability and stability.

\subsection{Special facets}\label{subsec_specialfacets}

Behrend shows in
\cite{behrend_semi-stabilityofreductivegroupschemesovercurves} that a
complementary polyhedron induces a unique facet, called the
\textit{special facet}, with properties akin to Harder--Narasimhan
filtrations and canonical reductions. For this, he defines
\textit{numerical invariants}.

\begin{definition}[(Numerical invariants {\cite[Definitions 3.5 and
    3.7]{behrend_semi-stabilityofreductivegroupschemesovercurves}})]\label{def_numericalinvariant}
    Let $d:\mathfrak{W}_\Phi\rightarrow V^*$ be a complementary polyhedron,
    let $P$ be a facet, and let $\lambda\in\mathrm{vert}(P)$ be a vertex of
    $P$, then we define:
    \begin{equation*}
        \Psi(P,\lambda):=\{\alpha\in\Phi\ \vert\ \langle\alpha,\lambda\rangle=1,\ \forall\mu\in\mathrm{vert}(P)\setminus\{\lambda\}:\langle\alpha,\mu\rangle=0\}
    \end{equation*}
    Let $\mathfrak{c}$ be any Weyl chamber with
    $P\geqslant\mathfrak{c}$, then the
    \textit{numerical invariant} is:
    \begin{equation*}
        n(P,\lambda):=\sum_{\alpha\in\Psi(P,\lambda)}d(\mathfrak{c})(\alpha)
    \end{equation*}
\end{definition}

This definition does not depend on the choice of Weyl chamber
$\mathfrak{c}$, due to \ref{def_complementarypolyhedra1}. Numerical
invariants recover the complementary polyhedron itself, since:
\begin{equation*}
    d(\mathfrak{c})=\sum_{\lambda\in\mathrm{vert}(\mathfrak{c})}n(\mathfrak{c},\lambda)\widecheck{\lambda}
\end{equation*}
as explained in \cite[Note
3.8]{behrend_semi-stabilityofreductivegroupschemesovercurves}.

\begin{definition}[(Special facets {\cite[Definition
    3.10]{behrend_semi-stabilityofreductivegroupschemesovercurves}})]\label{def_specialfacet}
    For a complementary polyhedron $d:\mathfrak{W}_\Phi\rightarrow V^*$, a
    facet $P$ is a \textit{special facet} if:
    \begin{enumerate}[label=(B\arabic*)]
        \item\label{def_specialfacet1} For all $\lambda\in\mathrm{vert}(P)$
            we have $n(P,\lambda)>0$.
        \item\label{def_specialfacet2}
            $d_P:\mathfrak{W}_{\Phi_P}\rightarrow\widecheck{P}^\perp$ is a
            semistable complementary polyhedron of $\Phi_P$.
    \end{enumerate}
\end{definition}

In \cite[Lemma
    3.11 and Corollary
3.14]{behrend_semi-stabilityofreductivegroupschemesovercurves},
Behrend provides an equivalent characterization of special facets,
using the following remark, used to prove the existence and
uniqueness of special facets.

\begin{remark}[($y(P)$ of a facet $P$ {\cite[Lemma
    3.11]{behrend_semi-stabilityofreductivegroupschemesovercurves}})]\label{rem_facety}
    For any facet $P$, there exists a unique point $y(P)\in V^*$ such that:
    \begin{equation*}
        \{y(P)\}=\mathrm{span}_{\mathbb{R}}(\widecheck{P})\cap\bigcap_{\lambda\in\mathrm{vert}(P)}H_\lambda
    \end{equation*}
    where $H_\lambda:=\{\mathbf{x}\in
    V^*\ \vert\ (\widecheck{\lambda},\mathbf{x})=(\widecheck{\lambda},d(\mathfrak{c}))\}$
    for any Weyl chamber $\mathfrak{c}$ with $P\geqslant\mathfrak{c}$ such that
    $\lambda\in\mathrm{vert}(\mathfrak{c})$. The hyperplane $H_\lambda$ does not
    depend on the choice of Weyl chamber $\mathfrak{c}$, due to
    \ref{def_complementarypolyhedra1}.
\end{remark}

\begin{theorem}[(Existence and uniqueness of special facets
            {\cite[Lemma 3.11 and Corollary
    3.14]{behrend_semi-stabilityofreductivegroupschemesovercurves}})]\label{thm_specialfacetuniqueness}
    A complementary polyhedron $d:\mathfrak{W}_{\Phi}\rightarrow V^*$ with
    convex hull $F$ has a unique special facet $P$, equivalently
    characterized by $y(P)\in\widecheck{P}\cap F(P)$. Furthermore,
    $y(P)$ is the minimum $\min(F)$.
\end{theorem}

Due to Theorem \ref{thm_specialfacetuniqueness}, a complementary
polyhedron is semistable if and only if its special facet is $\{0\}$.

\subsection{Jordan--H{\"o}lder facets}\label{subsec_jhfacets}

Similar to Behrend's results on special facets, we introduce
\textit{Jordan--H{\"o}lder facets} of semistable complementary polyhedra,
with properties similar to Jordan--H{\"o}lder filtrations and reductions.
We prove its existence and uniqueness up to Levi-factor.

\begin{definition}[(Jordan--H{\"o}lder facets)]\label{def_jhfacet}
    For a semistable complementary polyhedron $d:\mathfrak{W}_\Phi\rightarrow
    V^*$, a facet $P$ is a \textit{Jordan--H{\"o}lder facet} if:
    \begin{enumerate}[label=(J\arabic*)]
        \item\label{def_jhfacet1} For all $\lambda\in\mathrm{vert}(P)$, we
            have $n(P,\lambda)=0$.

        \item\label{def_jhfacet2}
            $d_P:\mathfrak{W}_{\Phi_P}\rightarrow\widecheck{P}^\perp$ is a
            stable complementary polyhedron of $\Phi_P$.
    \end{enumerate}
\end{definition}

Similar to \cite[Lemma
3.11]{behrend_semi-stabilityofreductivegroupschemesovercurves} for
special facets, we obtain an equivalent geometric characterization of
Jordan--H{\"o}lder facets, using \textit{relative interiors}. First, we
recall from \cite[\S1, pg. 8]{rockafellar_convexanalysis} that every
subset $F\subset V^*$ has an \textit{affine hull}:
\begin{align*}
    n:=\dim_{\mathbb{R}}(V)& &\mathrm{aff}_{\mathbb{R}}(F):=\left\{\mathbf{x}\in
        V^*\ \middle\vert\ \exists\lambda_i\in\mathbb{R}, \exists\mathbf{f}_i\in
        F:\mathbf{x}=\sum_{i=1}^{n}\lambda_i\mathbf{f}_i,
    1=\sum_{i=1}^{n}\lambda_i\right\}
\end{align*}

\begin{definition}[(Relative interiors
    {\cite[\protect\S6]{rockafellar_convexanalysis}})]\label{def_interior}
    For a convex subset $F$ of $V^*$, the \textit{relative interior}
    $\mathrm{relint}(F)\subset F$ is the union of subsets of $F$ open in the
    affine hull
    $\mathrm{aff}_{\mathbb{R}}(F)$.
\end{definition}

\begin{lemma}[(Equivalent characterization of Jordan--H{\"o}lder
    facets)]\label{lem_jhfacetequivalence}
    Let $d:\mathfrak{W}_\Phi\rightarrow V^*$ be a semistable complementary
    polyhedron with convex hull $F$. For a facet $P$, we define the conditions:
    \begin{enumerate}[label=(J\arabic*')]
        \item\label{lem_jhfacetequivalence1}
            $\mathrm{span}_{\mathbb{R}}(\widecheck{P})\cap F(P)=\{0\}$.

        \item\label{lem_jhfacetequivalence2}
            $\mathrm{span}_{\mathbb{R}}(\widecheck{P})\cap\mathrm{relint}(F(P))\neq\emptyset$
            and
            $V^*=\mathrm{span}_{\mathbb{R}}(\widecheck{P},F(P))$.
    \end{enumerate}
    Then \ref{def_jhfacet1} is equivalent to
    \ref{lem_jhfacetequivalence1}, and \ref{def_jhfacet2} is
    equivalent to \ref{lem_jhfacetequivalence2}. Thus, $P$ is
    a Jordan--H{\"o}lder facet if and only if $P$ fulfills
    \ref{lem_jhfacetequivalence1} and
    \ref{lem_jhfacetequivalence2}.
\end{lemma}

\begin{proof}
    Due to \cite[Propositions 1.9 and
    3.2]{behrend_semi-stabilityofreductivegroupschemesovercurves},
    \ref{def_jhfacet1} implies that $\deg(P)=0$, so \cite[Lemma
    3.4]{behrend_semi-stabilityofreductivegroupschemesovercurves}
    implies that $d_P$ is semistable, since $d$ is semistable. Then
    \cite[Lemma
    3.11]{behrend_semi-stabilityofreductivegroupschemesovercurves}
    implies that $y(P)\in F(P)$. In the proof of \cite[Lemma
    3.11]{behrend_semi-stabilityofreductivegroupschemesovercurves}, it
    is shown that:
    \begin{equation*}
        y(P)=\sum_{\lambda\in\mathrm{vert}(P)}\frac{n(P,\lambda)}{|\Psi(P,\lambda)|}\widecheck{\lambda}
    \end{equation*}
    Since the vertices are linearly independent, \ref{def_jhfacet1} is
    equivalent to $y(P)=0$. Then since $y(P)\in F(P)$, \cite[Lemma
    2.5]{behrend_semi-stabilityofreductivegroupschemesovercurves} and
    $y(P)=0$ imply \ref{lem_jhfacetequivalence1}. In the
    other direction, $\mathrm{span}_{\mathbb{R}}(\widecheck{P})\cap
    F(P)=\{0\}$ implies that $y(P)=0$, and by using $y(P)\in F(P)$, we
    obtain \ref{def_jhfacet1}.

    Let $P$ have \ref{def_jhfacet2}, if
    $\mathrm{span}_{\mathbb{R}}(\widecheck{P},F(P))\subsetneq V^*$,
    then $d_P$ cannot
    be stable as $F_P$ would not span $P^\perp$, contradicting Lemma
    \ref{lem_stabilityofcomplementarypolyhedra}
    \ref{lem_stabilityofcomplementarypolyhedra2}. Since $d_P$ is
    stable, Lemma \ref{lem_stabilityofcomplementarypolyhedra}
    \ref{lem_stabilityofcomplementarypolyhedra2} gives us that
    $\mathrm{int}(F_P)=\mathrm{relint}(F_P)$ is a neighborhood of $0$ in
    $P^\perp$. By \cite[Theorem 6.6]{rockafellar_convexanalysis},
    relative interiors of convex bodies are preserved by linear
    transformations, so
    $0\in\mathrm{relint}(F_P)=\widecheck{\mathrm{pr}}_P(\mathrm{relint}(F(P)))$
    and
    $\mathrm{span}_{\mathbb{R}}(\widecheck{P})=\widecheck{\mathrm{pr}}_P^{-1}(\{0\})$
    imply
    $\mathrm{span}_{\mathbb{R}}(\widecheck{P})\cap\mathrm{relint}(F(P))\neq\emptyset$,
    so we
    have \ref{lem_jhfacetequivalence2}. In the other
    direction, let $P$ have \ref{lem_jhfacetequivalence2}: We
    know that $\mathrm{span}_{\mathbb{R}}(\widecheck{P})\cap
    \mathrm{relint}(F(P))\neq\emptyset$, so there is a neighborhood $U$ of
    $\mathrm{span}_{\mathbb{R}}(\widecheck{P})\cap F(P)$ in $F(P)$, so
    $\widecheck{\mathrm{pr}}_P(U)$ is a neighborhood of $0$ in $F_P$, since
    $F_P=\widecheck{\mathrm{pr}}_P(F(P))$. In particular, since
    $V^*=\mathrm{span}_{\mathbb{R}}(\widecheck{P},F(P))$, $F_P$ is of maximal
    dimension in $P^\perp$, so $\widecheck{\mathrm{pr}}_P(U)$ is a
    neighborhood of
    $0$ in $P^\perp$, implying $0\in\mathrm{int}(F_P)$. By Lemma
    \ref{lem_stabilityofcomplementarypolyhedra}
    \ref{lem_stabilityofcomplementarypolyhedra2}, we have \ref{def_jhfacet2}.
\end{proof}

\begin{remark}[(Jordan--H{\"o}lder facets and
    stability)]\label{rem_jhfacetstability}
    By Lemma \ref{lem_jhfacetequivalence}, a semistable
    complementary polyhedron $d:\mathfrak{W}_\Phi\rightarrow V^*$ is stable
    if and only if $\{0\}$ is a Jordan--H{\"o}lder facet for $d$. This is
    because \ref{lem_jhfacetequivalence1} becomes $0\in F$,
    and \ref{lem_jhfacetequivalence2} becomes
    $0\in\mathrm{relint}(F)$ and $\mathrm{span}_{\mathbb{R}}(F)=V^*$,
    altogether equivalent
    to $0\in\mathrm{int}(F)$, so the claim follows by Lemma
    \ref{lem_stabilityofcomplementarypolyhedra}
    \ref{lem_stabilityofcomplementarypolyhedra2}.
\end{remark}

\begin{example}[(Jordan--H{\"o}lder facets of $A_2$)]\label{ex_jhfacetsA2}
    We adopt the same notation as in Example \ref{ex_complementarypolyhedra}.
    \begin{enumerate}[label=(\alph*)]
        \item\label{ex_jhfacetsA21} Let
            $d_b:\mathfrak{W}_{A_2}\rightarrow\mathbb{R}^{2*}$ be the semistable
            complementary polyhedron from Example
            \ref{ex_complementarypolyhedra}, with convex hull $F_b$.
            Using Lemma \ref{lem_jhfacetequivalence}, we
            claim that the facets $P_1$ and $P_2$, whose duals are labeled
            in Figure \ref{ex_jhfacetsA2} \ref{ex_jhfacetsA21} below,
            are all the
            Jordan--H{\"o}lder facets:
            \begin{equation*}
                \begin{tikzpicture}[rootsys]
                    %draw coroot vectors
                    \draw[root] (0,0) -- (0:1.0)
                    node[anchor=west,xshift=2pt]{$\widecheck{\alpha}$};
                    \draw[root] (0,0) -- (60:1.0)
                    node[pos=0.94,anchor=south
                    west,xshift=1pt,yshift=1pt]{$\widecheck{\alpha+\beta}$};
                    \draw[root] (0,0) -- (120:1.0)
                    node[pos=0.94,anchor=south
                    east,xshift=-1pt,yshift=1pt]{$\widecheck{\beta}$};
                    \foreach \angle in {180,240,300}
                    {
                        \draw[root] (0,0) -- (\angle:1.0);
                    }

                    %draw coroot walls
                    \foreach \angle in {30,150,210,330}
                    {
                        \draw[rootwall] (0,0) -- (\angle:1.35);
                    }

                    %draw Jordan--H{\"o}lder facets
                    \draw[facet] (0,0) -- (90:1.0)
                    node[anchor=south] {$\widecheck{P}_1$};
                    \draw[facet] (0,0) -- (270:1.0)
                    node[anchor=north] {$\widecheck{P}_2$};

                    %draw interval
                    \draw[thickpolygon] ($(0,0)+(0:1.0)$) --
                    ($(0,0)+(0:-1.0)$) -- cycle;
                    \fill[blue!80] ($(0,0)+(0:1.0)$) circle (1.3pt);
                    \fill[blue!80] ($(0,0)+(0:-1.0)$) circle (1.3pt);

                    %draw origin
                    \fill[red!80] (0,0) circle (1.3pt);

                    %label
                    \stepcounter{subfig}
                    \node at (0,-1.65) {\footnotesize \protect\thesubfig};
                \end{tikzpicture}
            \end{equation*}
            We first
            claim that the only facets fulfilling
            \ref{lem_jhfacetequivalence1} are $P_1$, $P_2$, and
            $\{0\}$: No Weyl chambers fulfill
            \ref{lem_jhfacetequivalence1}, since no Weyl chamber
            evaluates $d_b$ to $0$. Any $1$-dimensional facet $P$ not equal
            to $P_1$ or $P_2$ has $F_b(P)\subset\{\alpha,-\alpha\}$, so
            \ref{lem_jhfacetequivalence1} is not fulfilled for
            $P$. For the remaining facets, we have $0\in
            F_b(P_1)=F_b(P_2)=F_b(\{0\})=F_b$, so
            \ref{lem_jhfacetequivalence1} is fulfilled. From the
            list of $P_1$, $P_2$, and $\{0\}$, the facet $\{0\}$ is the
            only facet that does not have
            \ref{lem_jhfacetequivalence2}, since
            $0\in\mathrm{relint}(F_b)$, and $(\widecheck{P}_1,F_b(P_1))$ and
            $(\widecheck{P}_2,F_b(P_2))$ span $\mathbb{R}^{2*}$, but
            $(\{0\},F_b)$ does not span $\mathbb{R}^{2*}$.

        \item\label{ex_jhfacetsA22} Let
            $d:\mathfrak{W}_{A_2}\rightarrow\mathbb{R}^{2*}$
            be given by:
            \begin{align*}
                d(\mathfrak{c}_{-\beta})=d(\mathfrak{c}_{\alpha})=\widecheck{\beta}
                &  &
                d(\mathfrak{c}_{-\alpha-\beta})=\widecheck{\alpha+\beta} &  &
                d(\mathfrak{c}_{-\alpha})=d(\mathfrak{c}_{\beta})=\widecheck{\alpha}
                &  &
                d(\mathfrak{c}_{\alpha+\beta})=0
            \end{align*}
            which is a complementary polyhedron with convex hull $F$ given by:
            \begin{equation*}
                \begin{tikzpicture}[rootsys]
                    %draw and label Jordan--H{\"o}lder facet
                    \draw[greenweylchamber] (30:1.35) -- (0,0) --
                    (90:1.35) -- cycle;
                    \node[text=green!50!black, xshift=5pt] at
                    ($(0,0)!1!(60:1.35)$)
                    {$\widecheck{\mathfrak{c}_{\alpha+\beta}}$};

                    %draw rhombus
                    \filldraw[
                        fill=blue!15,
                        draw=blue!60!black,
                        line width=1.0pt
                    ]
                    (0,0) -- (0:1) -- (60:1) -- (120:1) -- cycle;

                    %draw coroots
                    \draw[root] (0,0) -- (0:1)
                    node[anchor=west,xshift=2pt] {$\widecheck{\alpha}$};
                    \draw[root] (0,0) -- (120:1)
                    node[anchor=south east,xshift=-1pt,yshift=2pt]
                    {$\widecheck{\beta}$};
                    \draw[root] (0,0) -- (60:1.0);
                    \foreach \angle in {180,240,300}
                    {
                        \draw[root] (0,0) -- (\angle:1.0);
                    }

                    %draw coroot walls
                    \foreach \angle in {30,90,150,330}
                    {
                        \draw[rootwall] (0,0) -- (\angle:1.35);
                    }
                    \draw[facet] (0,0) -- (210:1.35)
                    node[anchor=east] {$\widecheck{P}_1$};
                    \draw[facet] (0,0) -- (270:1)
                    node[anchor=north] {$\widecheck{P}_2$};

                    %draw origin
                    \fill[red!80] (0,0) circle (1.3pt);

                    %label
                    \stepcounter{subfig}
                    \node at (0,-1.65) {\footnotesize \protect\thesubfig};
                \end{tikzpicture}
            \end{equation*}
            The duals of the facets $\mathfrak{c}_{\alpha+\beta}$, $P_1$, and
            $P_2$, are highlighted in Figure \ref{ex_jhfacetsA2}
            \ref{ex_jhfacetsA22}. Using Lemma \ref{lem_jhfacetequivalence}, we
            claim that $\mathfrak{c}_{\alpha+\beta}$ is the only
            Jordan--H{\"o}lder
            facet. We
            first show that the only facets fulfilling
            \ref{lem_jhfacetequivalence2} are $P_1$, $P_2$, and
            all Weyl chambers: The case of Weyl chambers is clear, since
            their duals all span $\mathbb{R}^{2*}$. We calculate for $P_1$
            and $P_2$ that:
            \begin{align*}
                F(P_1)=\mathrm{conv}(\widecheck{\alpha},\widecheck{\alpha+\beta})
                &
                &
                F(P_2)=\mathrm{conv}(\widecheck{\beta},\widecheck{\alpha+\beta})
            \end{align*}
            whose relative interiors intersect with
            $\mathrm{span}_{\mathbb{R}}(\widecheck{P}_1)$ and
            $\mathrm{span}_{\mathbb{R}}(\widecheck{P}_2)$
            respectively, as seen in Figure \ref{ex_jhfacetsA2}
            \ref{ex_jhfacetsA22}.
            Furthermore, since $(\widecheck{P}_1,F(P_1))$ and
            $(\widecheck{P}_2,F(P_2))$ span $\mathbb{R}^{2*}$, we follow that
            $P_1$ and $P_2$ have \ref{lem_jhfacetequivalence2}.
            To exclude the remaining facets from having
            \ref{lem_jhfacetequivalence2}, we calculate that:
            \begin{align*}
                F(-P_1)=\mathrm{conv}(0,\widecheck{\beta}) &  &
                F(-P_2)=\mathrm{conv}(0,\widecheck{\alpha})
            \end{align*}
            whose relative interiors do not intersect with
            $\mathrm{span}_{\mathbb{R}}(-\widecheck{P}_1)$ and
            $\mathrm{span}_{\mathbb{R}}(-\widecheck{P}_2)$
            respectively. Furthermore, the spans of the remaining
            $1$-dimensional facets do not intersect with $\mathrm{relint}(F)$ at
            all. To test \ref{lem_jhfacetequivalence1} on $P_1$,
            $P_2$, and all Weyl chambers, we calculate that:
            \begin{align*}
                \mathrm{span}_{\mathbb{R}}(\widecheck{P}_1)\cap
                F(P_1)\neq\{0\} &  &
                \mathrm{span}_{\mathbb{R}}(\widecheck{P}_2)\cap F(P_2)\neq\{0\}
            \end{align*}
            so $P_1$ and $P_2$ do not fulfill
            \ref{lem_jhfacetequivalence1}. Since
            $\mathfrak{c}_{\alpha+\beta}$ is the only Weyl chamber such that
            $d(\mathfrak{c}_{\alpha+\beta})=0$,
            $\mathfrak{c}_{\alpha+\beta}$ is the
            only facet with \ref{lem_jhfacetequivalence1} and
            \ref{lem_jhfacetequivalence2}, i.e.,
            $\mathfrak{c}_{\alpha+\beta}$ is the only Jordan--H{\"o}lder facet.
    \end{enumerate}
\end{example}

Using how numerical invariants are preserved under projections in
\cite[Lemma
3.9]{behrend_semi-stabilityofreductivegroupschemesovercurves}, we can
prove the existence of Jordan--H{\"o}lder facets through induction.

\begin{theorem}[(Existence of Jordan--H{\"o}lder
    facets)]\label{thm_jhfacetexistence}
    For a semistable complementary polyhedron $d:\mathfrak{W}_\Phi\rightarrow
    V^*$, $P$ is a Jordan--H{\"o}lder facet $P$ if and only if it is a
    minimal facet of $\Phi$ with \ref{def_jhfacet1} (in terms of
    the order of facets $\geqslant$). In particular, a Jordan--H{\"o}lder
    facet exists.
\end{theorem}

\begin{proof}
    We first prove the case of $\Phi:=A_1=:\{\alpha,-\alpha\}$, the
    only rank $1$ root system up to isomorphism, where $\alpha>0$ is
    the positive root in $V=\mathbb{R}$. We label the Weyl chambers by
    $\mathbb{R}_{> 0}$ and $\mathbb{R}_{<0}$ with vertices
    ${\lambda}:=\mathrm{vert}(\mathbb{R}_{> 0})$ and
    ${-\lambda}=\mathrm{vert}(\mathbb{R}_{<0})$. Due to
    \ref{def_complementarypolyhedra2}, we have $d(\mathbb{R}_{>
    0})=\widecheck{\eta}$ and $d(\mathbb{R}_{<0})=\widecheck{\mu}$, for some
    $\eta,\mu\in\mathbb{R}$ such that $\eta\leqslant\mu$, and since $d$ is
    semistable, we know that $\eta\leqslant0\leqslant\mu$, using Lemma
    \ref{lem_stabilityofcomplementarypolyhedra}
    \ref{lem_stabilityofcomplementarypolyhedra1}. Since $n(\mathbb{R}_{>
    0},\lambda)=d(\mathbb{R}_{> 0})(\alpha)$ and $n(\mathbb{R}_{<
    0},-\lambda)=d(\mathbb{R}_{< 0})(-\alpha)$, $\mathbb{R}_{>0}$ is a
    Jordan--H{\"o}lder facet if and only if $n(\mathbb{R}_{> 0},\lambda)=0$,
    $\mathbb{R}_{<0}$ is a Jordan--H{\"o}lder facet if and only if
    $n(\mathbb{R}_{<0},-\lambda)=0$, and $\{0\}$ is a
    Jordan--H{\"o}lder facet if
    and only if $n(\mathbb{R}_{> 0},\lambda)\neq0$ and
    $n(\mathbb{R}_{<0},-\lambda)\neq0$.
    The claim of the theorem follows.

    In the general situation, if $P:=\{0\}$ is minimal with
    \ref{def_jhfacet1}, then equivalently, for every $1$-dimensional
    facet $Q$ of $\Phi$ with vertex $\lambda$, we have
    $n(Q,\lambda)\neq0$. By \cite[Proposition
    1.9]{behrend_semi-stabilityofreductivegroupschemesovercurves},
    and by using that $d$ is semistable, we
    know that $\sum_{\alpha\in R(Q)}\alpha\in Q$, so
    $n(Q,\lambda)\neq0$ is in equivalence with:
    \begin{equation*}
        \deg(Q)=\sum_{\alpha\in R(Q)}d(\mathfrak{c})(\alpha)<0
    \end{equation*}
    for any Weyl chamber $\mathfrak{c}$ such that
    $Q\geqslant\mathfrak{c}$. Using
    \cite[Remark 5.3.4]{behrend_phdthesis}, $d$
    is equivalently stable, so by Remark \ref{rem_jhfacetstability},
    $\{0\}$ is a Jordan--H{\"o}lder facet.

    Otherwise, $\{0\}$ is not minimal with \ref{def_jhfacet1}, and we
    perform induction on the rank $n$ of the root system $\Phi$, i.e.,
    the dimension $\dim_{\mathbb{R}}(V)$, so we assume the theorem is true
    for rank $n-1$ root systems. Let $P$ be a minimal facet of $\Phi$
    with \ref{def_jhfacet1}, which exists since $\{0\}$ always fulfills
    \ref{def_jhfacet1}, but now $P\neq\{0\}$ by assumption. To prove
    that the minimality condition for $P$ is equivalent to being
    Jordan--H{\"o}lder, we enumerate the $1$-dimensional facets
    $Q_1,\ldots,Q_l$, such that $Q_i\geqslant P$ for all $i=1,\ldots,l$. Let
    $\mathrm{pr}_{Q_i}:V\rightarrow Q^\perp_i$ denote the projection
    to $Q_i$ for all
    $i=1,\ldots,l$, then by using that $P$ is the unique facet with
    $Q_i\geqslant P$ for all $i=1,\ldots,l$, \cite[Lemma
    3.9]{behrend_semi-stabilityofreductivegroupschemesovercurves}
    implies that the minimality of $P$ with \ref{def_jhfacet1} is
    equivalent to the minimality of $\mathrm{pr}_{Q_i}(P)$ with
    \ref{def_jhfacet1}
    for $d_{Q_i}$, for all $i=1,\ldots,l$. Since $\Phi_{Q_i}$ are rank
    $n-1$ root systems for all $i=1,\ldots,l$, by the induction
    assumption, this is equivalent to $\mathrm{pr}_{Q_i}(P)$ being a
    Jordan--H{\"o}lder
    facet of $d_{Q_i}$, for all $i=1,\ldots,l$. By Definition \ref{def_jhfacet},
    this is equivalent to $\mathrm{pr}_{Q_i}(P)$ with \ref{def_jhfacet1}, for
    all  $i=1,\ldots,l$,
    and $\mathrm{pr}_{Q_i}(P)$ with \ref{def_jhfacet2}, for all
    $i=1,\ldots,l$, with \ref{def_jhfacet2} equivalent to
    \ref{lem_jhfacetequivalence2} by Lemma
    \ref{lem_jhfacetequivalence}:
    \begin{align*}
        \mathrm{span}_{\mathbb{R}}(\widecheck{\mathrm{pr}}_{Q_i}(\widecheck{P}))\cap\mathrm{relint}(\widecheck{\mathrm{pr}}_{Q_i}(F(P)))\neq\emptyset
        &  & \widecheck{Q}^\perp_i\text{ is spanned by
        }(\widecheck{\mathrm{pr}}_{Q_i}(\widecheck{P}),\widecheck{\mathrm{pr}}_{Q_i}(F(P)))
    \end{align*}
    \sloppy
    for all $i=1,\ldots,l$, where $F$ is the convex hull of $d$. Using
    \cite[Theorem 6.6]{rockafellar_convexanalysis}, implying that
    $\mathrm{relint}(\widecheck{\mathrm{pr}}_{Q_i}(F(P)))$ is equal to
    $\widecheck{\mathrm{pr}}_{Q_i}(\mathrm{relint}(F(P)))$, and using
    the fact that $V^*$ is
    spanned by $(\widecheck{Q}^\perp_i,F(Q_i),i=1,\ldots,l)$, the
    minimality of $P$ with \ref{def_jhfacet1} is equivalent to
    $\mathrm{span}_{\mathbb{R}}(\widecheck{P})\cap\mathrm{relint}(F(P))$
    being nonempty, and that
    $V^*$ is spanned by $(\widecheck{P},F(P))$. In short, the
    minimality of $P$ with \ref{def_jhfacet1} is equivalent to
    \ref{def_jhfacet1} and \ref{def_jhfacet2}, using Lemma
    \ref{lem_jhfacetequivalence}.
\end{proof}

Now we prove that there exist automorphisms of the root system that
identify Jordan--H{\"o}lder facets. This is the analog of the uniqueness
of Jordan--H{\"o}lder filtrations up to its graded, as described in
the introduction.

\begin{theorem}[(Uniqueness of Jordan--H{\"o}lder facets up to
    Levi-factor)]\label{thm_jhfacetuniqueness}
    For a semistable complementary polyhedron $d:\mathfrak{W}_\Phi\rightarrow
    V^*$, and Jordan--H{\"o}lder facets $P_1$ and $P_2$, then:
    \begin{enumerate}[label=(\roman*)]
        \item\label{thm_jhfacetuniqueness1} The perpendicular spaces of the
            facets are equal, i.e., $P_1^\perp=P_2^\perp$.

        \item\label{thm_jhfacetuniqueness2} There exists a linear orthogonal
            automorphism $\tau:V\rightarrow V$, such that $\tau(\Phi)=\Phi$
            and $\tau(P_1)=P_2$.

        \item\label{thm_jhfacetuniqueness3} The induced Levi-factors of the
            complementary polyhedron $d$ are equal, i.e., $d_{P_1}=d_{P_2}$.
    \end{enumerate}
\end{theorem}

\begin{proof}
    We first prove the case of $\Phi:=A_1$, the only rank $1$ root
    system up to isomorphism, using the notation of the proof of
    Theorem \ref{thm_jhfacetexistence}. The only situation where
    $P_1\neq P_2$ is when $d=0$ and
    $\{P_1,P_2\}=\{\mathbb{R}_{>0},\mathbb{R}_{<0}\}$, hence
    $P_1^\perp=P_2^\perp=\{0\}$ and $d_{P_1}=d_{P_2}=0$ are automatic.
    The linear orthogonal automorphism $\tau:\mathbb{R}\rightarrow
    \mathbb{R}$ is
    given by $-\mathrm{id}_{\mathbb{R}}$.

    We perform induction on the rank $n$ of the root system $\Phi$,
    i.e., the dimension $\dim_\mathbb{R}(V)$, so we assume the theorem is
    true for rank $n-1$ root systems. If $d$ is stable, the theorem
    is trivial, so we assume $d$ is non-stable. First, we claim that there exist
    $1$-dimensional facets $Q_1$ and $Q_2$ of $\Phi$, such that
    $Q_1\geqslant P_1$ and $Q_2\geqslant P_2$, with:
    \begin{center}
        \begin{enumerate*}[label=(C\arabic*)]
        \item\label{thm_jhfacetuniquenessclaim1}
            $Q_1^\perp = Q_2^\perp$
            \hspace{10em}
        \item\label{thm_jhfacetuniquenessclaim2}
            $d_{Q_1} = d_{Q_2}$
        \end{enumerate*}
    \end{center}
    \sloppy
    Let $F$ denote the convex hull of $d$. If no such $Q_1$ and $Q_2$ exist
    with \ref{thm_jhfacetuniquenessclaim1}, then
    $0\in\mathrm{int}(\mathrm{conv}(F(Q_1),F(Q_2)))$ for every $1$-dimensional
    facet $Q_1$ and $Q_2$ with $Q_1\geqslant P_1$ and $Q_2\geqslant P_2$,
    implying that there exists a facet $P$ with $P\geqslant P_1,P_2$, such
    that $0\in\mathrm{int}(F(P))$. Thus, $y(P)$ from Remark
    \ref{rem_facety} is $0$, and by \cite[Lemma
    3.11]{behrend_semi-stabilityofreductivegroupschemesovercurves},
    and the fact that relative interiors of convex bodies are
    preserved by linear
    transformations, from \cite[Theorem
    6.6]{rockafellar_convexanalysis}, $d_P$ is stable. Then, by
    Remark \ref{rem_jhfacetstability}, $P\neq\{0\}$, since $d$ is not
    stable. Using \cite[Lemma
    3.9]{behrend_semi-stabilityofreductivegroupschemesovercurves} and
    Theorem \ref{thm_jhfacetexistence} for the projection
    $\mathrm{pr}_P:V\rightarrow P^\perp$, $\mathrm{pr}_P(P_1)$ and
    $\mathrm{pr}_P(P_2)$ are
    Jordan--H{\"o}lder facets of $d_P$, so
    $\mathrm{pr}_P(P_1)=\mathrm{pr}_P(P_2)=\{0\}$.
    Thus, $P_1,P_2\subset P^\perp$, in contradiction to $P\neq\{0\}$
    and $P\geqslant P_1,P_2$.
    As a result, there exists $Q_1$ and $Q_2$ with
    \ref{thm_jhfacetuniquenessclaim1}. Since $P_1$ and $P_2$ are
    Jordan--H{\"o}lder facets, they both have \ref{def_jhfacet1}, and due to
    Remark \ref{rem_facety}, we have:
    \begin{equation*}
        \mathrm{span}_{\mathbb{R}}(\widecheck{Q}_1)\cap
        F(Q_1)=\mathrm{span}_{\mathbb{R}}(\widecheck{Q}_2)\cap
        F(Q_2)=\{0\}
    \end{equation*}
    Since $Q_1^\perp=Q_2^\perp$, we have
    $\mathrm{span}_{\mathbb{R}}(\widecheck{Q}_1)=\mathrm{span}_{\mathbb{R}}(\widecheck{Q}_2)$
    and
    thus $Q_1=Q_2$ or $Q_1=-Q_2$, which implies
    $\mathrm{span}_{\mathbb{R}}(\widecheck{Q}_1)\cap
    F=\mathrm{span}_{\mathbb{R}}(\widecheck{Q}_2)\cap F=\{0\}$ and
    $F\subset\widecheck{Q}_1^\perp=\widecheck{Q}_2^\perp$. If
    $Q_1=Q_2$, then \ref{thm_jhfacetuniquenessclaim2} is trivial, so we
    assume $Q_1=-Q_2$. Let $\rho:V\rightarrow V$ be the linear
    orthogonal automorphism preserving $\Phi$ that maps $Q_1$ to
    $Q_2$ and fixes $Q_1^\perp=Q_2^\perp$, then for
    \ref{thm_jhfacetuniquenessclaim2} it suffices to show that for every
    Weyl chamber $\mathfrak{c}\in\mathfrak{W}_\Phi$ with
    $Q_1\geqslant\mathfrak{c}$, we have
    $d(\mathfrak{c})=d(\rho(\mathfrak{c}))$. Let $\mathrm{pr}_{Q_i}:V\rightarrow
    Q_i^\perp$, $i=1,2$, denote the projections, which have
    $\mathrm{pr}_{Q_1}=\mathrm{pr}_{Q_2}$, and let
    $\mathfrak{W}^{Q_i,\mathfrak{c}}_\Phi$ denote
    the Weyl chambers $\mathfrak{d}$ of $\Phi$ with
    $\mathrm{pr}_{Q_i}(\mathfrak{d})=\mathrm{pr}_{Q_i}(\mathfrak{c})$.
    Because Weyl chambers are
    connected, the closure of the union of Weyl chambers in
    $\mathfrak{W}^{Q_i,\mathfrak{c}}_\Phi$ is connected, so there exists an
    ordered list
    $(\mathfrak{c},\mathfrak{d}_1,\ldots,\mathfrak{d}_l,\rho(\mathfrak{c}))$ of
    pairwise distinct Weyl chambers in $\mathfrak{W}^{Q_i,\mathfrak{c}}_\Phi$,
    such that each consecutive pair shares $n-1$ vertices. By
    applying \ref{def_complementarypolyhedra1} to all shared vertices
    of each consecutive pair, and by using that
    $F\subset\widecheck{Q}_1^\perp=\widecheck{Q}_2^\perp$, we see
    that the value of $d$ stays constant along each consecutive pair,
    so
    $d(\mathfrak{c})=d(\mathfrak{d}_1)=\ldots=d(\mathfrak{d}_l)=d(\rho(\mathfrak{c}))$.

    For $Q_1$ and $Q_2$ with \ref{thm_jhfacetuniquenessclaim1} and
    \ref{thm_jhfacetuniquenessclaim2}, we have that
    $\mathrm{pr}_{Q_1}(P_1)$ and $\mathrm{pr}_{Q_2}(P_2)$ are
    Jordan--H{\"o}lder facets of
    $d_{Q_1}$, using \cite[Lemma
    3.9]{behrend_semi-stabilityofreductivegroupschemesovercurves} and
    Theorem \ref{thm_jhfacetexistence}.
    Since $\Phi_{Q_1}$ is a root system of rank $n-1$, by the induction
    assumption, we have
    $\mathrm{pr}_{Q_1}(P_1)^\perp=\mathrm{pr}_{Q_2}(P_2)^\perp$, and there
    exists a linear orthogonal automorphism $\tau':Q_1^\perp\rightarrow
    Q_1^\perp$, such that $\tau'(\Phi_{Q_1})=\Phi_{Q_1}$ and
    $\tau'(\mathrm{pr}_{Q_1}(P_1))=\mathrm{pr}_{Q_2}(P_2)$, and we have
    $(d_{Q_1})_{\mathrm{pr}_{Q_1}(P_1)}=(d_{Q_2})_{\mathrm{pr}_{Q_2}(P_2)}$.
    Since by
    \ref{thm_jhfacetuniquenessclaim1}, $Q_1=Q_2$ or $Q_1=-Q_2$,
    \ref{thm_jhfacetuniqueness1} and \ref{thm_jhfacetuniqueness3} follow, and
    $\tau'$ extends to the desired $\tau$ from \ref{thm_jhfacetuniqueness2}.
\end{proof}

\begin{remark}[(Outer automorphisms)]\label{rem_outerautomorphism}
    Theorem \ref{thm_jhfacetuniqueness} does not claim that $\tau$ is an
    element of the Weyl group of $\Phi$, i.e., an \textit{inner automorphism}
    of $\Phi$. Indeed, in Example \ref{ex_jhfacetsA2}
    \ref{ex_jhfacetsA21}, $\tau$ is the
    reflection along the $\alpha$-axis, which preserves $A_2$ but is not
    in the Weyl group. In this case, we call $\tau$ an \textit{outer
    automorphism} of $A_2$.
\end{remark}

We are now able to prove that a variant of Lemma
\ref{lem_stabilityofcomplementarypolyhedra} works for polystability.

\begin{lemma}[(Equivalent characterization of
    polystability)]\label{lem_polystabilityequivalencepolyhedra}
    Let $d:\mathfrak{W}_\Phi\rightarrow V^*$ be a semistable complementary
    polyhedron with convex hull $F$, then $d$ is polystable if and only
    if $0\in\mathrm{relint}(F)$.
\end{lemma}

\begin{proof}
    In the forward direction, for a Jordan--H{\"o}lder facet
    $P$, which exists due to Theorem \ref{thm_jhfacetexistence}, we
    have $\deg(P)=0$, using that $P$ has \ref{def_jhfacet1} and
    \cite[Propositions 1.9 and
    3.2]{behrend_semi-stabilityofreductivegroupschemesovercurves}. From
    \ref{def_stabilityofcomplementarypolyhedra32}, the opposite facet
    $-P$ has $\deg(-P)=0$. For every vertex $\lambda\in\mathrm{vert}(P)$,
    there exists a $1$-dimensional facet $Q$ such that
    $\mathrm{vert}(Q)=\{\lambda\}$. Since $d$ is semistable,
    $\deg(Q),\deg(-Q)\leqslant0$, and by \cite[Propositions 1.9 and
    3.2]{behrend_semi-stabilityofreductivegroupschemesovercurves},
    $\deg(P),\deg(-P)=0$ implies that $\deg(Q),\deg(-Q)=0$ and
    $n(Q,\lambda),n(-Q,-\lambda)=0$. From this, and the
    proof of Lemma \ref{lem_jhfacetequivalence}, it follows that
    $F\subset P^\perp$, and since $P$ is a Jordan--H{\"o}lder facet, $0$
    lies in the interior of $F_P$, by Lemma
    \ref{lem_stabilityofcomplementarypolyhedra}
    \ref{lem_stabilityofcomplementarypolyhedra2}. Thus, we have
    $0\in\mathrm{relint}(F)$.

    In the reverse direction, from $0\in\mathrm{relint}(F)$, we have $0\in F$,
    which gives \ref{def_stabilityofcomplementarypolyhedra31} due to
    Lemma \ref{lem_stabilityofcomplementarypolyhedra}
    \ref{lem_stabilityofcomplementarypolyhedra1} and \cite[Lemma
    3.4]{behrend_semi-stabilityofreductivegroupschemesovercurves}. To
    prove \ref{def_stabilityofcomplementarypolyhedra32}, let $P$ be a
    facet such that $\deg(P)=0$, then we must show that $\deg(-P)=0$.
    Since $0\in\mathrm{relint}(F)$ and $\deg(P)=0$, we have $F\subset P^\perp$,
    which implies $\deg(-P)=0$.
\end{proof}

Combining Lemmas \ref{lem_stabilityofcomplementarypolyhedra} and
\ref{lem_polystabilityequivalencepolyhedra}, we have the
summary \eqref{eq_intro} from the introduction.

\begin{example}[(Equivalent characterization of
    polystability)]\label{ex_polystabilitypolyhedra}
    Returning to Example \ref{ex_complementarypolyhedra}, the
    semistable complementary polyhedron
    $d_b:\mathfrak{W}_{A_2}\rightarrow\mathbb{R}^{2*}$ is polystable by applying
    Lemma \ref{lem_polystabilityequivalencepolyhedra}, since
    $0\in\mathrm{relint}(F_b)$.
\end{example}

Due to Lemma \ref{lem_polystabilityequivalencepolyhedra},
we get the following equivalent characterization of polystability
through Jordan--H{\"o}lder facets, reducing the number of facets to check
for polystability.

\sloppy
\begin{lemma}[(Jordan--H{\"o}lder characterization of
    polystability)]\label{lem_polystabilityjhfacet}
    Let $d:\mathfrak{W}_\Phi\rightarrow V^*$ be a semistable complementary
    polyhedron, then $d$ is polystable if and only if for every
    Jordan--H{\"o}lder facet $P$, the opposite facet $-P$ is also a
    Jordan--H{\"o}lder facet.
\end{lemma}

\begin{proof}
    For the forward direction, since $0\in\mathrm{relint}(F)$ by Lemma
    \ref{lem_polystabilityequivalencepolyhedra}, we have
    $F\subset P^\perp$ using \ref{lem_jhfacetequivalence2}
    from Lemma \ref{lem_jhfacetequivalence}. Recall that
    \ref{def_jhfacet1} for $P$ is equivalent to $y(P)=0$, through the
    proof of Lemma \ref{lem_jhfacetequivalence}. Together
    with $F\subset P^\perp$, $y(P)=0$ implies $y(-P)=0$, so $-P$ has
    \ref{def_jhfacet1} by the proof of Lemma
    \ref{lem_jhfacetequivalence}. Since taking opposite
    facets preserves the order of facets $\geqslant$, $-P$ is minimal with
    \ref{def_jhfacet1}, so $-P$ is a Jordan--H{\"o}lder facet due to Theorem
    \ref{thm_jhfacetexistence}.

    For the reverse direction, since $d$ is semistable, we have
    \ref{def_stabilityofcomplementarypolyhedra31} due to \cite[Lemma
    3.4]{behrend_semi-stabilityofreductivegroupschemesovercurves}. For
    \ref{def_stabilityofcomplementarypolyhedra32}, due to
    \cite[Propositions 1.9 and
    3.2]{behrend_semi-stabilityofreductivegroupschemesovercurves}, it
    suffices to check $1$-dimensional facets $P$. Thus, $\deg(P)=0$
    implies that $P$ has
    \ref{def_jhfacet1}, so there exists a Jordan--H{\"o}lder facet $Q$ with
    $P\geqslant Q$, by Theorem \ref{thm_jhfacetexistence}. Since $-Q$ is a
    Jordan--H{\"o}lder facet by assumption, we get $\deg(-P)=0$, due to
    \cite[Propositions 1.9 and
    3.2]{behrend_semi-stabilityofreductivegroupschemesovercurves}.
\end{proof}

\begin{example}[(Jordan--H{\"o}lder characterization of polystability)]
    From Example \ref{ex_complementarypolyhedra}, the complementary
    polyhedron $d_b:\mathfrak{W}_{A_2}\rightarrow\mathbb{R}^{2*}$ is polystable
    precisely because $d_b$ is semistable and the only two
    Jordan--H{\"o}lder facets $P_1$ and $P_2$ of $d_b$, identified in
    Example \ref{ex_jhfacetsA2} \ref{ex_jhfacetsA21}, are opposite to
    each other.
\end{example}

\section{Jordan--H{\"o}lder parabolics for reductive group schemes over a
curve}\label{sec_jhtheoryreductivegroupschemes}

To connect bundle moduli problems to the moduli problems stated in
the introduction in \ref{intro_applications}, we need to induce
complementary polyhedra from bundles that store their stability
profiles. In the case of principal bundles, this was achieved by
Behrend in
\cite[\S\S6--8]{behrend_semi-stabilityofreductivegroupschemesovercurves}
by studying the (semi)-stability of \textit{reductive group schemes}
over curves, since as an example, automorphism group schemes of
bundles over curves form reductive group schemes over curves. For parahoric
torsors, similar results were achieved by Heinloth in
\cite[\S3.D]{heinloth_hilbertmumfordstability}.

Let $K$ be a field. We will study bundles and reductive group schemes
over a smooth projective
$K$-curve $X$.

\subsection{Definition and stability of reductive group
schemes}\label{subsec_stabilitygroupscheme}

We recall the notion of a reductive group scheme.

\begin{definition}[(Reductive group schemes {\cite[Definition
    3.1.1]{conrad_reductivegroupschemes}})]\label{def_reductivegroupscheme}
    A smooth affine group scheme $\mathbf{G}$ over a $K$-scheme $S$
    is a \textit{reductive
    group scheme} if for all $K$-points $s$ of $S$, and the induced
    geometric $\overline{K}$-point $\overline{s}$, where
    $\overline{K}$ is an algebraic closure of $K$, the fiber
    $\mathbf{G}_{\overline{s}}$ is a connected reductive group.
\end{definition}

We have notions of \textit{Cartan}, \textit{Borel}, and
\textit{parabolic subgroup schemes} of $\mathbf{G}$, from \cite[Exp.
XXVI]{sga3} and \cite[\S5]{conrad_reductivegroupschemes}, that
generalize the usual notions for reductive groups.
Behrend defines \textit{degree} and \textit{(semi)-stability} for
reductive group schemes over $X$.

\begin{definition}[(Degree and stability {\cite[Definitions 4.1 and
    4.4]{behrend_semi-stabilityofreductivegroupschemesovercurves}})]\label{def_groupschemestability}
    Let $\mathbf{G}$ be a reductive group scheme over $X$.
    \begin{enumerate}[label=(\alph*)]
        \item\label{def_groupschemestability1} For a smooth affine group
            scheme $\mathbf{P}$ over $X$, we define the \textit{degree}:
            \begin{equation*}
                \deg(\mathbf{P}):=\deg(\mathrm{Lie}(\mathbf{P}))
            \end{equation*}
            where $\mathrm{Lie}(\mathbf{P})$ is seen as a vector
            bundle over $X$.

        \item\label{def_groupschemestability2} $\mathbf{G}$ is
            \textit{(semi)-stable} if for every proper parabolic subgroup
            scheme $\mathbf{P}\subsetneq\mathbf{G}$, we have
            $\deg(\mathbf{P})\semileq 0$, using Notation
            \ref{rem_semistabilitynotation}.
    \end{enumerate}
\end{definition}

Let $\mathbf{G}$ be a \textit{split} reductive group scheme over a
$K$-scheme $S$, i.e.,
there exists a \textit{maximal torus} $\mathbf{T}$ of $\mathbf{G}$, so
$\mathbf{T}\cong\mathbb{G}_{m,S}^r$ for $r$ maximal. As seen in
\cite[\S5.1]{conrad_reductivegroupschemes}, we obtain a root system
$\Phi(\mathbf{G},\mathbf{T})$ of
$(V_{\mathbf{G},\mathbf{T}},\langle\_,\_\rangle_{\mathbf{G},\mathbf{T}})$, with
$V_{\mathbf{G},\mathbf{T}}:=\mathcal{X}^*(\mathbf{T})\otimes\mathbb{R}$, where:
\begin{equation*}\label{eq_charactergroup}
    \tag{3.1}
    \mathcal{X}^*(\mathbf{T}):=\{\chi\in\mathrm{Hom}_S(\mathbf{T},\mathbb{G}_{m,S})\ \vert\ \chi(\mathbf{Z}(\mathbf{G}))=\mathbf{e}\}
\end{equation*}
is the lattice of \textit{characters} trivial on the center, and
$\langle\_,\_\rangle_{\mathbf{G},\mathbf{T}}$ is an adjoint invariant
inner-product.

For any reductive group scheme $\mathbf{G}$ over $S$, there exists a
surjective {\'e}tale morphism $\pi:Y\rightarrow S$ such that the
reductive group scheme $\mathbf{G}_Y:=\pi^*\mathbf{G}$ over $Y$ is split, due to
\cite[Lemma 5.1.3]{conrad_reductivegroupschemes}. We can classify
parabolic subgroup schemes of $\mathbf{G}$ using these splittings:
Let $\mathrm{Par}(\mathbf{G})$ denote the scheme over
$S$, as defined
in \cite[Exp. XXVI, 3.2]{sga3}, whose global sections
$\Gamma(\mathrm{Par}(\mathbf{G}))$ correspond to parabolic subgroup schemes of
$\mathbf{G}$. Let $\mathrm{Par}_{\mathbf{T}_Y}(\mathbf{G})$ denote
the subscheme of
$\mathrm{Par}(\mathbf{G})$ whose global sections
$\Gamma(\mathrm{Par}_{\mathbf{T}_Y}(\mathbf{G}))$
correspond to parabolic subgroup schemes $\mathbf{P}$ of $\mathbf{G}$ whose
pullbacks $\mathbf{P}_Y:=\pi^*\mathbf{P}$ contain $\mathbf{T}_Y$. Recall that a
parabolic subgroup scheme $\mathbf{P}$ of $\mathbf{G}$ defines a \textit{type}
$t(\mathbf{P})$, as in \cite[Remark
5.1]{behrend_semi-stabilityofreductivegroupschemesovercurves}, which
is a global section of the \textit{power scheme}
$\mathfrak{P}(\mathrm{Dyn}(\mathbf{G}))$ of the \textit{Dynkin scheme}
$\mathrm{Dyn}(\mathbf{G})$.

\sloppy
When $\mathbf{G}$ is \textit{rationally split} or \textit{rationally
trivial} over $X$, i.e., split over the \textit{generic point}
$\eta:=\mathrm{Spec}(K(X))$, Behrend proves the following lemma.

\begin{lemma}[(Root systems of reductive group schemes
            {\cite[Proposition
    6.2]{behrend_semi-stabilityofreductivegroupschemesovercurves}})]\label{lem_rootsystemgroupscheme}
    Let $\mathbf{G}$ be a rationally split reductive group scheme
    over $X$, with generic point $\pi:\eta\rightarrow X$ and maximal
    generic torus $\mathbf{T}_\eta$ in $\mathbf{G}_\eta:=\pi^*\mathbf{G}$.
    \begin{enumerate}[label=(\roman*)]
        \item\label{lem_rootsystemgroupscheme1} There is a
            bijective map
            $R:\Gamma(\mathrm{Par}_{\mathbf{T}_\eta}(\mathbf{G}))\rightarrow\{\textit{Parabolic
            subsets of }\Phi(\mathbf{G}_\eta,\mathbf{T}_\eta)\}$ such that
            for any parabolic
            subgroup scheme $\mathbf{P}$ of $\mathbf{G}$, with
            $\mathbf{P}_\eta:=\pi^*\mathbf{P}$
            containing $\mathbf{T}_\eta$, we have:
            \begin{equation*}
                \mathrm{Lie}(\mathbf{P}_\eta)=\mathrm{Lie}(\mathbf{T}_\eta)\oplus\bigoplus_{\alpha\in
                R(\mathbf{P})}\mathrm{Lie}((\mathbf{G}_{\eta})_\alpha)
            \end{equation*}

        \item\label{lem_rootsystemgroupscheme2} There is an
            order-preserving bijective map
            $\underline{R}:\Gamma(\mathrm{Par}_{\mathbf{T}_\eta}(\mathbf{G}))\rightarrow\{\textit{Facets
            of }\Phi(\mathbf{G}_\eta,\mathbf{T}_\eta)\}$ (with
                respect to inclusion of
                subgroups and
            $\geqslant$) such that
            $R(\mathbf{P})=R(\underline{R}(\mathbf{P}))$, for the
            parabolic subset $R(\underline{R}(\mathbf{P}))$ of
            $\underline{R}(\mathbf{P})$ from Definition
            \ref{def_stabilityofcomplementarypolyhedra}
            \ref{def_stabilityofcomplementarypolyhedra1}.

        \item\label{lem_rootsystemgroupscheme3} For any
            parabolic subgroup scheme $\mathbf{P}$ of $\mathbf{G}$,
            such that $\mathbf{P}_\eta$
            contains $\mathbf{T}_\eta$, there is a bijection:
            \begin{align*}
                \lambda:\pi_0(t(\mathbf{P}))\rightarrow\mathrm{vert}(\underline{R}(P))
                &  & \mathfrak{v}\mapsto\lambda_\mathfrak{v}
            \end{align*}
    \end{enumerate}
\end{lemma}

Using Lemma \ref{lem_rootsystemgroupscheme}, we can extend Behrend's
notions of stability and semistability to polystability.

\begin{definition}[(Polystability)]\label{def_polystabilitygroupscheme}
    Let $\mathbf{G}$ be a rationally split reductive group scheme
    over $X$, with generic point $\pi:\eta\rightarrow X$, then
    $\mathbf{G}$ is \textit{polystable} if for all parabolic subgroup schemes
    $\mathbf{P}$ of $\mathbf{G}$ with $\deg(\mathbf{P})=0$:
    \begin{enumerate}[label=(G\arabic*)]
        \item\label{def_polystabilitygroupscheme1} The \textit{Levi-factor}
            $\mathbf{L}:=\mathbf{P}/\mathbf{R}_u(\mathbf{P})$
            (quotient by the
            \textit{unipotent radical}) is semistable.

        \item\label{def_polystabilitygroupscheme2}
            There exists a
            parabolic subgroup scheme $\mathbf{P}'$ of $\mathbf{G}$
            \textit{opposite} to $\mathbf{P}$, i.e.,
            $\mathbf{P}\cap\mathbf{P}'$ is
            isomorphic to $\mathbf{L}$ through the extension
            $\mathbf{P}\rightarrow\mathbf{L}=\mathbf{P}/\mathbf{R}_u(\mathbf{P})$
            to the Levi-factor, and
            $\underline{R}(\mathbf{P}')=-\underline{R}(\mathbf{P})$
            through Lemma \ref{lem_rootsystemgroupscheme}
            \ref{lem_rootsystemgroupscheme2}.
    \end{enumerate}
\end{definition}

When $\mathbf{G}$ is rationally split over $X$, Behrend constructs an
equivalent condition for the (semi)-stability of $\mathbf{G}$ using
complementary polyhedra. For a parabolic subgroup scheme $\mathbf{P}$ of
$\mathbf{G}$ with vertex $\mathfrak{v}\in\pi_0(t(\mathbf{P}))$, we
use the notion of
the \textit{numerical invariant} $n(\mathbf{P},\mathfrak{v})$ from
\cite[Definition
5.5]{behrend_semi-stabilityofreductivegroupschemesovercurves}. Recall
that for a parabolic subgroup scheme $\mathbf{P}$ of
$\mathbf{G}$, the unipotent radical $\mathbf{R}_u(\mathbf{P})$ has a
filtration of
subgroups, called the \textit{composition series}:
\begin{equation*}
    \mathbf{R}_u(\mathbf{P})_0\supset\mathbf{R}_u(\mathbf{P})_1\supset\ldots
\end{equation*}
as defined in
\cite[Exp. XXVI, 2.1]{sga3} and
\cite[Proposition
5.4]{behrend_semi-stabilityofreductivegroupschemesovercurves}. The
quotients form direct sums of vector bundles over $X$, as seen in
\cite[Theorem 5.4.3]{conrad_reductivegroupschemes}, for which the
direct sums of
$\mathbf{R}_u(\mathbf{P})_0/\mathbf{R}_u(\mathbf{P})_1$ are called
\textit{elementary vector bundles} in \cite[Definition
5.5]{behrend_semi-stabilityofreductivegroupschemesovercurves}.
Elementary vector bundles can be naturally indexed by vertices in
$\pi_0(t(\mathbf{P}))$, labeled $W(\mathbf{P},\mathfrak{v})$, whose degrees are
\textit{numerical invariants}
$n(\mathbf{P},\mathfrak{v}):=\deg(W(\mathbf{P},\mathfrak{v}))$.
\begin{lemma}[(Complementary polyhedra of reductive group schemes
            {\cite[Proposition
    6.6]{behrend_semi-stabilityofreductivegroupschemesovercurves}})]\label{lem_complementarypolyhedrongroupscheme}
    Let $\mathbf{G}$ be a rationally split reductive group scheme
    over $X$, with generic point $\pi:\eta\rightarrow X$ and maximal
    generic torus $\mathbf{T}_\eta$ in
    $\mathbf{G}_\eta:=\pi^*\mathbf{G}$. The following is a
    complementary polyhedron:
    \begin{align*}
        d_{\mathbf{G}_{\eta},\mathbf{T}_{\eta}}:\mathfrak{W}_{\Phi(\mathbf{G}_{\eta},\mathbf{T}_{\eta})}\rightarrow
        V^*_{\mathbf{G}_{\eta},\mathbf{T}_{\eta}} &  &
        \mathfrak{c}\mapsto\sum_{\mathfrak{v}\in\pi_0(t(\underline{R}^{-1}(\mathfrak{c})))}n(\underline{R}^{-1}(\mathfrak{c}),\mathfrak{v})\widecheck{\lambda}_\mathfrak{v}
    \end{align*}
\end{lemma}

In Lemma \ref{lem_complementarypolyhedrongroupscheme}, $\mathbf{G}$ must be
rationally split, since Behrend's proof requires that there exists a finite
surjective {\'e}tale pullback of $\mathbf{G}$ that is an \textit{inner form},
as defined in \cite[Exp. XXIV, 1]{sga3}.
Next, Behrend proves that the
(semi)-stability of $\mathbf{G}$ is equivalent to the (semi)-stability of
$d_{\mathbf{G}_{\eta},\mathbf{T}_{\eta}}$ for all
$\mathbf{T}_{\eta}$, which we extend to polystability.

\begin{lemma}[(Stability conditions of reductive group schemes
            {\cite[Propositions 6.8 and
    6.9]{behrend_semi-stabilityofreductivegroupschemesovercurves}})]\label{lem_stabilityequivalencegroupscheme}
    Let $\mathbf{G}$ be a rationally split reductive group scheme
    over $X$, with generic point $\pi:\eta\rightarrow X$ and
    $\mathbf{G}_\eta:=\pi^*\mathbf{G}$.
    \begin{enumerate}[label=(\roman*)]
        \item\label{lem_stabilityequivalencegroupscheme1}
            $\mathbf{G}$ is (semi)-stable if
            and only if $d_{\mathbf{G}_{\eta},\mathbf{T}_{\eta}}$ is
            (semi)-stable for all maximal tori $\mathbf{T}_{\eta}$ of
            $\mathbf{G}_{\eta}$.
        \item\label{lem_stabilityequivalencegroupscheme2}
            $\mathbf{G}$ is polystable if and
            only if $d_{\mathbf{G}_{\eta},\mathbf{T}_{\eta}}$ is
            polystable for all maximal tori $\mathbf{T}_{\eta}$ of
            $\mathbf{G}_{\eta}$.
    \end{enumerate}
\end{lemma}

\begin{proof}
    We prove \ref{lem_stabilityequivalencegroupscheme1}: For every
    parabolic subgroup scheme $\mathbf{P}$ of $\mathbf{G}$, there
    exists a maximal torus $\mathbf{T}_\eta$ inside
    $\mathbf{P}_{\eta}:=\pi^*\mathbf{P}$ by \cite[Exp. XXVI,
    4.1.1]{sga3}. Hence, by Lemma \ref{lem_rootsystemgroupscheme},
    there is a facet $\underline{R}(\mathbf{P})$ of $\mathbf{P}$, for
    which by \cite[Proposition
    6.8]{behrend_semi-stabilityofreductivegroupschemesovercurves},
    $\deg(\mathbf{P})=\deg(\underline{R}(\mathbf{P}))$. This equates
    the stability conditions.

    We now prove \ref{lem_stabilityequivalencegroupscheme2}: It is clear that
    conditions \ref{def_polystabilitygroupscheme1} and
    \ref{def_stabilityofcomplementarypolyhedra31} are equivalent, due
    to \ref{lem_stabilityequivalencegroupscheme1} and the fact from
    \cite[Proposition
    6.9]{behrend_semi-stabilityofreductivegroupschemesovercurves} that
    for any parabolic subgroup scheme $\mathbf{P}$ of $\mathbf{G}$ where
    $\mathbf{P}_\eta$ contains $\mathbf{T}_\eta$, the extension
    $\mathbf{P}\rightarrow\mathbf{L}=\mathbf{P}/\mathbf{R}_u(\mathbf{P})$
    projects the root system:
    \begin{equation*}
        \Phi(\mathbf{G}_\eta,\mathbf{T}_\eta)\text{ of
        }(V_{\mathbf{G}_\eta,\mathbf{T}_\eta},\langle\_,\_\rangle_{\mathbf{G}_\eta,\mathbf{T}_\eta})
    \end{equation*}
    down to the following root system, with a canonical isomorphism of
    root systems:
    \begin{align*}
        \Phi(\mathbf{G}_\eta,\mathbf{T}_\eta)_{\underline{R}(\mathbf{P})}\text{
            of
        }(\underline{R}(\mathbf{P})^\perp,\langle\_,\_\rangle_{\underline{R}(\mathbf{P})})
        &  & \cong &  & \Phi(\mathbf{L}_\eta,\mathbf{T}_\eta)\text{ of
        }(V_{\mathbf{L}_\eta,\mathbf{T}_\eta},\langle\_,\_\rangle_{\mathbf{L}_\eta,\mathbf{T}_\eta})
    \end{align*}
    To prove the equivalence between \ref{def_polystabilitygroupscheme2}
    and \ref{def_stabilityofcomplementarypolyhedra32}, for parabolic
    subgroup schemes $\mathbf{P}$ and $\mathbf{P}'$, where $\mathbf{P}_\eta$ and
    $\mathbf{P}'_\eta:=\pi^*$ contain $\mathbf{T}_\eta$, $\mathbf{P}'$ is
    opposite to $\mathbf{P}$ if
    and only if
    $\underline{R}(\mathbf{P}')=-\underline{R}(\mathbf{P})$. By
    \cite[Proposition
    6.8]{behrend_semi-stabilityofreductivegroupschemesovercurves}, we
    know that $\deg(\mathbf{P})=0$ if and only if
    $\deg(\underline{R}(\mathbf{P}))=0$, and that
    $\deg(\mathbf{P}')=0$ if and only
    if
    $\deg(\underline{R}(\mathbf{P}'))=\deg(-\underline{R}(\mathbf{P}))=0$. Then
    the claim follows.
\end{proof}

\begin{remark}[(Maximal parabolics suffice)]\label{rem_maximalparabolicssuffice}
    In the setup of Lemmas \ref{lem_complementarypolyhedrongroupscheme} and
    \ref{lem_stabilityequivalencegroupscheme}, for the stability conditions of
    $\mathbf{G}$, it suffices to only check the relevant conditions
    for maximal parabolic subgroup
    schemes $\mathbf{P}$ of $\mathbf{G}$. To prove this for
    (semi)-stability,
    in \cite[Propositions 1.9 and
    3.2]{behrend_semi-stabilityofreductivegroupschemesovercurves}, it
    is shown that $d_{\mathbf{G}_{\eta},\mathbf{T}_{\eta}}$ from Lemma
    \ref{lem_complementarypolyhedrongroupscheme} is (semi)-stable if and only if
    for all $1$-dimensional facets $P$ of
    $\Phi(\mathbf{G}_{\eta},\mathbf{T}_{\eta})$, we have $\deg(P)\semileq 0$,
    with Notation \ref{rem_semistabilitynotation}. Since
    Lemma \ref{lem_rootsystemgroupscheme}
    \ref{lem_rootsystemgroupscheme2} gives an order-preserving
    bijection $\underline{R}$, maximal parabolics $\mathbf{P}$ of $\mathbf{G}$,
    where $\mathbf{P}_{\eta}:=\pi^*\mathbf{P}$ contains
    $\mathbf{T}_{\eta}$, correspond
    to $1$-dimensional facets $\underline{R}(\mathbf{P})$ of
    $\Phi(\mathbf{G}_{\eta},\mathbf{T}_{\eta})$. With these facts, the claim
    follows from Lemma \ref{lem_stabilityequivalencegroupscheme}
    \ref{lem_stabilityequivalencegroupscheme1}. For polystability, the same
    arguments also are compatible with conditions
    \ref{def_polystabilitygroupscheme1}, \ref{def_polystabilitygroupscheme2},
    \ref{def_stabilityofcomplementarypolyhedra31}, and
    \ref{def_stabilityofcomplementarypolyhedra32}.
\end{remark}

\subsection{Jordan--H{\"o}lder parabolics for reductive group
schemes over a curve}\label{subsec_jhgroupscheme}
Similar to how Behrend constructs \textit{canonical parabolics} for
reductive group schemes in \cite[Theorem
7.3]{behrend_semi-stabilityofreductivegroupschemesovercurves}, we
apply our Jordan--H{\"o}lder theory for complementary polyhedra to
reductive group schemes to construct \textit{Jordan--H{\"o}lder
parabolics}. This allows us to apply our Jordan--H{\"o}lder theory for
complementary polyhedra to the automorphism groups of the moduli
problems stated in the introduction, in \ref{intro_applications},
inducing Jordan--H{\"o}lder filtrations.

Using numerical invariants from \cite[Definition
5.5]{behrend_semi-stabilityofreductivegroupschemesovercurves} (which
we recalled after Definition \ref{def_polystabilitygroupscheme}), we
define Jordan--H{\"o}lder parabolics of reductive group schemes $\mathbf{G}$,
with the same properties as Jordan--H{\"o}lder facets of
$d_{\mathbf{G}_{\eta},\mathbf{T}_{\eta}}$ from Lemma
\ref{lem_complementarypolyhedrongroupscheme} when $\mathbf{G}$ is
rationally split.

\begin{definition}[(Jordan--H{\"o}lder
    parabolics)]\label{def_jhparabolicgroupscheme}
    Let $\mathbf{G}$ be a semistable reductive group scheme over $X$,
    then a parabolic
    subgroup scheme $\mathbf{P}$ of $\mathbf{G}$ is
    \textit{Jordan--H{\"o}lder} if:
    \begin{enumerate}[label=(\roman*)]
        \item\label{def_jhparabolicgroupscheme1} The numerical invariants
            $n(\mathbf{P},\mathfrak{v})$ are $0$ for every vertex
            $\mathfrak{v}\in\pi_0(t(\mathbf{P}))$.

        \item\label{def_jhparabolicgroupscheme2} $\mathbf{L}$ is stable, in the
            sense of Definition \ref{def_groupschemestability}
            \ref{def_groupschemestability2}, where
            $\mathbf{L}$ is
            the Levi-factor of $\mathbf{P}$.
    \end{enumerate}
\end{definition}

In our moduli problems, we often use different stability conditions
for reductive group schemes than those of Definition
\ref{def_groupschemestability} \ref{def_groupschemestability2}.
Motivated by Lemma
\ref{lem_stabilityequivalencegroupscheme}, we define stability
conditions and Jordan--H{\"o}lder parabolics of $\mathbf{G}$ through
prescribed complementary
polyhedra, directly generalizing Definitions
\ref{def_groupschemestability}, \ref{def_polystabilitygroupscheme},
and \ref{def_jhparabolicgroupscheme}.

\begin{definition}[($d$-stability and Jordan--H{\"o}lder
    parabolics)]\label{def_dstabilitygroupscheme}
    Let $\mathbf{G}$ be a rationally split reductive group scheme
    over $X$, with generic point $\pi:\eta\rightarrow X$ and
    $\mathbf{G}_\eta:=\pi^*\mathbf{G}$.
    \begin{enumerate}[label=(\alph*)]
        \item\label{def_dstabilitygroupscheme1} A family of
            complementary polyhedra
            $d=(d_{\mathbf{T}_\eta}:\mathfrak{W}_{\Phi(\mathbf{G}_\eta,\mathbf{T}_\eta)}\rightarrow
            V^*_{\mathbf{G}_\eta,\mathbf{T}_\eta})_{\mathbf{T}_\eta}$
            over the maximal tori $\mathbf{T}_\eta$ of
            $\mathbf{G}_\eta$ is called \textit{compatible}, if for
            any parabolic subgroup scheme $\mathbf{P}$ of
            $\mathbf{G}$, the sign of the degrees of
            $\underline{R}(\mathbf{P})$ (and the sign of its
            numerical invariants when the degree is $\geq0$), using
            Lemma \ref{lem_rootsystemgroupscheme}
            \ref{lem_rootsystemgroupscheme2} with respect to
            $d_{\mathbf{T}_\eta}$, do not depend on the choice of
            $\mathbf{T}_\eta$.

        \item\label{def_dstabilitygroupscheme2} If $d$ from
            \ref{def_dstabilitygroupscheme1} is compatible,
            $\mathbf{G}$ has \textit{bounded $d$-stability} if the
            degrees of $d$ have an upper bound $B<\infty$.

            \sloppy
        \item\label{def_dstabilitygroupscheme3} If $d$ from
            \ref{def_dstabilitygroupscheme1} is compatible,
            $\mathbf{G}$ is \textit{$d$-(semi)-stable} (respectively
            \textit{$d$-polystable}), if $d_{\mathbf{T}_\eta}$ is
            \textit{(semi)-stable} (respectively \textit{polystable})
            for all maximal tori $\mathbf{T}_\eta$.

        \item\label{def_dstabilitygroupscheme4} If $\mathbf{G}$ is
            $d$-semistable as in \ref{def_dstabilitygroupscheme3}, a
            parabolic subgroup scheme $\mathbf{P}$ of
            $\mathbf{G}$ is \textit{Jordan--H{\"o}lder} with respect to $d$ if
            the numerical invariants of
            $\underline{R}(\mathbf{P}_\eta)$, using Lemma
            \ref{lem_rootsystemgroupscheme}
            \ref{lem_rootsystemgroupscheme2} with respect to
            $d_{\mathbf{T}_Y}$ for any $\mathbf{T}_\eta$ contained in
            $\mathbf{P}_\eta$, are all equal to $0$. Moreover,
            $(d_{\mathbf{T}_\eta})_{\underline{R}(\mathbf{P}_\eta)}$
            is stable for any $\mathbf{T}_\eta$ contained in $\mathbf{P}_\eta$.
    \end{enumerate}
\end{definition}

Bounded $d$-stability is important to verify, as done by Behrend in
\cite[Lemma
4.3]{behrend_semi-stabilityofreductivegroupschemesovercurves}, as it
allows us to translate properties of special facets to properties of
\textit{canonical parabolics} of $\mathbf{G}$, from \cite[Definition
5.6]{behrend_semi-stabilityofreductivegroupschemesovercurves}. In our
Jordan--H{\"o}lder setup, bounded $d$-stability follows directly from
$d$-(semi)-stability or $d$-polystability.
For a rationally split semistable reductive group scheme $\mathbf{G}$ over $X$,
due to Lemma \ref{lem_stabilityequivalencegroupscheme},
Jordan--H{\"o}lder parabolics $\mathbf{P}$ of $\mathbf{G}$ with respect to
the compatible
$d=(d_{\mathbf{G}_{\eta},\mathbf{T}_{\eta}})_{\mathbf{T}_{\eta}}$,
from Lemma \ref{lem_complementarypolyhedrongroupscheme}, are precisely
Jordan--H{\"o}lder parabolics
$\mathbf{P}$ of $\mathbf{G}$, in the sense of Definition
\ref{def_jhparabolicgroupscheme}.

Through our Jordan--H{\"o}lder theory for complementary polyhedra, we have
the following theorem proving the existence of Jordan--H{\"o}lder
parabolics, and their uniqueness up to isomorphism of Levi-factor.
\begin{theorem}[(Jordan--H{\"o}lder parabolics with respect to
    $d$)]\label{thm_jhgroupscheme}
    Let $\mathbf{G}$ be a rationally split reductive group scheme
    over $X$, with generic point $\pi:\eta\rightarrow X$ and
    $\mathbf{G}_\eta:=\pi^*\mathbf{G}$. Let $d$ be a compatible
    family of complementary polyhedra on $\mathbf{G}_\eta$, such that
    $\mathbf{G}$ is $d$-semistable.
    \begin{enumerate}[label=(\roman*)]
        \item\label{thm_jhgroupscheme1} There exists a
            Jordan--H{\"o}lder parabolic $\mathbf{P}$ of $\mathbf{G}$
            with respect to $d$.

        \item\label{thm_jhgroupscheme2} For Jordan--H{\"o}lder
            parabolics $\mathbf{P}_1$ and $\mathbf{P}_2$ of
            $\mathbf{G}$ with respect to $d$,
            there exists an automorphism
            $\tau_{\mathbf{L}}:\mathbf{L}_1\rightarrow\mathbf{L}_2$
            between their
            Levi-factors
            $\mathbf{L}_1=\mathbf{P}_1/\mathbf{R}_u(\mathbf{P}_1)$ and
            $\mathbf{L}_2=\mathbf{P}_2/\mathbf{R}_u(\mathbf{P}_2)$.
    \end{enumerate}
\end{theorem}

\begin{proof}
    For \ref{thm_jhgroupscheme1}, by applying Lemmas
    \ref{lem_rootsystemgroupscheme}
    \ref{lem_rootsystemgroupscheme2} and
    \ref{lem_stabilityequivalencegroupscheme}, a parabolic subgroup
    scheme $\mathbf{P}$ of $\mathbf{G}$ is a
    Jordan--H{\"o}lder parabolic with respect to $d$ if and only if
    $\underline{R}(\mathbf{P})$ is a Jordan--H{\"o}lder facet of
    $d_{\mathbf{T}_\eta}$ for every $\mathbf{T}_\eta$ contained
    inside $\mathbf{P}_\eta:=\pi^*\mathbf{P}$. From the existence of
    Jordan--H{\"o}lder facets of $d$ from
    Theorem \ref{thm_jhfacetexistence}, the Jordan--H{\"o}lder facets
    of maximal dimension induce the existence of Jordan--H{\"o}lder parabolics.

    For \ref{thm_jhgroupscheme2}, by \cite[Exp. XXVI, 4.1.1]{sga3}
    $(\mathbf{P}_1)_\eta:=\pi^*\mathbf{P}_1$ and
    $(\mathbf{P}_2)_\eta:=\pi^*\mathbf{P}_2$ share a maximal torus
    $\mathbf{T}_\eta$, so by Lemma
    \ref{lem_rootsystemgroupscheme}
    \ref{lem_rootsystemgroupscheme2}, they
    correspond to Jordan--H{\"o}lder facets $P_1$ and $P_2$ of
    $d_{\mathbf{T}_\eta}$. Since $P_1^\perp=P_2^\perp$ due to Theorem
    \ref{thm_jhfacetuniqueness} \ref{thm_jhfacetuniqueness1},
    $\mathbf{P}_1\cap\mathbf{P}_2$ is isomorphic to the Levi factors
    $\mathbf{L}_1$ and $\mathbf{L}_2$ respectively, using
    \cite[Proposition
    6.9]{behrend_semi-stabilityofreductivegroupschemesovercurves}.
\end{proof}

\begin{remark}[(Based root data)]\label{rem_basedrootdata}
    The automorphism in Theorem \ref{thm_jhgroupscheme}
    \ref{thm_jhgroupscheme2} can be realized as follows:
    By using \textit{based root data} from
    \cite[(1.5.1) and Theorem 6.1.17]{conrad_reductivegroupschemes},
    the automorphism $\tau:V_{\mathbf{G}_\eta,\mathbf{T}_\eta}\rightarrow
    V_{\mathbf{G}_\eta,\mathbf{T}_\eta}$ from Theorem
    \ref{thm_jhfacetuniqueness} induces an
    isomorphism $\tau_{\mathbf{G}_\eta}$ of $\mathbf{G}_\eta$, since
    $\tau$ preserves
    $\Phi(\mathbf{G}_\eta,\mathbf{T}_\eta)$. This restricts to
    $\tau_{\mathbf{L}_\eta}:(\mathbf{L}_1)_\eta\rightarrow(\mathbf{L}_2)_\eta$,
    which extends uniquely to $X$ through Theorem
    \ref{thm_jhgroupscheme} \ref{thm_jhgroupscheme2}.
\end{remark}

In the setup of Theorem \ref{thm_jhgroupscheme},
We obtain another characterization of the polystability of $d$,
inspired by Lemma \ref{lem_polystabilityjhfacet}.

\begin{lemma}[(Equivalent characterization of
    polystability)]\label{lem_polystabilityequivalencegroupscheme}
    Let $\mathbf{G}$ be a rationally split reductive group scheme
    over $X$, with generic point $\pi:\eta\rightarrow X$ and
    $\mathbf{G}_\eta:=\pi^*\mathbf{G}$. Let $d$ be a compatible
    family of complementary polyhedra on $\mathbf{G}_\eta$. Then
    $\mathbf{G}$ is $d$-polystable if and only if for every parabolic
    subgroup scheme $\mathbf{P}$ with
    $\mathbf{P}_\eta:=\pi^*\mathbf{P}$ and
    $\deg(\underline{R}(\mathbf{P}_\eta))=0$, we have:
    \begin{enumerate}[label=(H\arabic*)]
        \item\label{lem_polystabilityequivalencegroupscheme1}
            For all maximal tori $\mathbf{T}_\eta$ contained in
            $\mathbf{P}_\eta$,
            $(d_{\mathbf{T}_\eta})_{\underline{R}(\mathbf{P}_\eta)}$
            is semistable.
        \item\label{lem_polystabilityequivalencegroupscheme2} There exists a
            parabolic subgroup scheme $\mathbf{P}'$ of $\mathbf{G}$
            \textit{opposite}
            to $\mathbf{P}$, i.e.,
            $\mathbf{P}\cap\mathbf{P}'$ is
            isomorphic to $\mathbf{L}$ through the extension
            $\mathbf{P}\rightarrow\mathbf{L}=\mathbf{P}/\mathbf{R}_u(\mathbf{P})$
            to the Levi-factor, and
            $\underline{R}(\mathbf{P}')=-\underline{R}(\mathbf{P})$
            through Lemma \ref{lem_rootsystemgroupscheme}
            \ref{lem_rootsystemgroupscheme2}.
    \end{enumerate}
    Moreover, $\mathbf{G}$ is $d$-polystable if and only if it is
    $d$-semistable and:
    \begin{enumerate}[label=(H')]
        \item\label{lem_polystabilityequivalencegroupscheme3} For every
            Jordan--H{\"o}lder parabolic $\mathbf{P}$ of $\mathbf{G}$
            with respect to $d$,
            there exists a Jordan--H{\"o}lder parabolic $\mathbf{P}'$
            of $\mathbf{G}$ with
            respect to $d$ that is \textit{opposite} to $\mathbf{P}$.
    \end{enumerate}
\end{lemma}

\begin{proof}
    The equivalence between
    \ref{def_stabilityofcomplementarypolyhedra31} and
    \ref{lem_polystabilityequivalencegroupscheme1} is analogous to Lemma
    \ref{lem_stabilityequivalencegroupscheme}
    \ref{lem_stabilityequivalencegroupscheme2}, using
    \cite[Lemma
    3.4]{behrend_semi-stabilityofreductivegroupschemesovercurves}. For
    the equivalence between
    \ref{def_stabilityofcomplementarypolyhedra32} and
    \ref{lem_polystabilityequivalencegroupscheme2}, we use that $\mathbf{P}$ and
    $\mathbf{P}'$
    are opposite if and only if $\mathbf{P}'_\eta:=\pi^*\mathbf{P}'$ contains
    $\mathbf{T}_\eta$, and
    $\underline{R}(\mathbf{P}')=-\underline{R}(\mathbf{P})$ together
    with $\deg(\underline{R}(\mathbf{P}'))=0$. For
    \ref{lem_polystabilityequivalencegroupscheme3}, we apply Lemma
    \ref{lem_polystabilityjhfacet}.
\end{proof}

\section{Bundle moduli problems}\label{sec_modulistacks}

In order to apply our Jordan--H{\"o}lder theory for complementary
polyhedra to the moduli problems stated in the introduction in
\ref{intro_applications}, we introduce \textit{bundle moduli
problems}. Through the \textit{Rees construction}, bundle moduli
problems model moduli stacks of bundles through their complementary
polyhedra, and connect the Jordan--H{\"o}lder facets from Definition
\ref{def_jhfacet} to S-equivalence from Zhang in
\cite{zhang_Sequivalenceforalgebraicstacks}.

In this section, our stacks $\mathcal{M}$ are algebraic stacks over a field $K$.

\subsection{Filtrations and
\texorpdfstring{$\mathcal{L}$}{L}-stability}\label{subsec_filtrations}

By assigning families of complementary polyhedra to $K$-points of
$\mathcal{M}$, we can study Jordan--H{\"o}lder filtrations and
S-equivalence in $\mathcal{M}$. The induced stability
conditions are often equivalent to the stability conditions from a
line bundle $\mathcal{L}$ over $\mathcal{M}$, using
\textit{$\mathcal{L}$-(semi)-stability} from Heinloth in \cite[Definition
1.2]{heinloth_hilbertmumfordstability}, which we extend to
\textit{$\mathcal{L}$-polystability}, using \textit{opposite filtrations}.

\begin{definition}[(Filtrations
            {\cite[Definition
    1.1]{heinloth_hilbertmumfordstability}})]\label{def_filtration}
    Let $x$ be a $K$-point of $\mathcal{M}$, then a \textit{(stacky)
    filtration} of $x$ is a
    morphism $f:\Theta_K\rightarrow\mathcal{M}$, where
    $\Theta_K:=[\mathbb{A}^1_K/\mathbb{G}_{m,K}]$, and $f(1)=x$. We call
    $f(0)$ the \textit{Levi-factor} of $f$.
\end{definition}

For a filtration $f:\Theta_K\rightarrow\mathcal{M}$ of $x$ such that
$f(0)\neq x$, $f$ is called a \textit{very close degeneration} in
\cite[Definition 1.1]{heinloth_hilbertmumfordstability}, and we say
that $f(1)$ \textit{degenerates} or \textit{specializes} to $f(0)$.

For the moduli stack of vector bundles
$\mathcal{M}:=\mathcal{M}\textsc{Bun}_{r,d}$ over a
smooth projective $K$-curve $X$, filtrations of $\mathcal{M}\textsc{Bun}_{r,d}$
correspond to filtrations of vector bundles indexed by $\mathbb{Z}$, as
proven by Klyachko in \cite{klyachko_equivariantbundlesontoralvarieties} and
recalled in \cite[Lemma 1.10]{heinloth_hilbertmumfordstability}. This
correspondence is called the \textit{Rees construction}, since by
taking the spectrum of commutative rings, it comes from the
\textit{Rees algebra} in \cite{rees_onaproblemofzariski}.
In
\cite[Example 4.9]{zhang_Sequivalenceforalgebraicstacks}, Zhang
shows that opposite filtrations of $\mathcal{M}\textsc{Bun}_{r,d}$ correspond to
opposite filtrations of vector bundles, indexed by $\mathbb{Z}$, justifying
the naming in Definition \ref{def_filtration}.

Heinloth defines \textit{$\mathcal{L}$-stability} as follows, which we
extend to \textit{$\mathcal{L}$-polystability}.

\begin{definition}[($\mathcal{L}$-stability {\cite[Definition
    1.2]{heinloth_hilbertmumfordstability}})]\label{def_Lstability}
    Let $\mathcal{M}$ be locally of finite type and with affine diagonal, and
    let $x$ be a $K$-point of $\mathcal{M}$. For a line bundle
    $\mathcal{L}$ over
    $\mathcal{M}$, we define the following stability conditions:
    \begin{enumerate}[label=(\alph*)]
        \item\label{def_Lstability1} $x$ is \textit{$\mathcal{L}$-(semi)-stable}
            if for all filtrations $f$ of $x$, we have
            $\mathrm{wt}(f^*\mathcal{L})\semileq 0$, where $\mathrm{wt}$
            denotes the \textit{weight}
            of the $\mathbb{G}_{m,K}$-action on $\mathbb{A}^1_K$
            lifted to $f^*\mathcal{L}$, using Notation
            \ref{rem_semistabilitynotation}.

        \item\label{def_Lstability2} $x$ is \textit{$\mathcal{L}$-polystable}
            if for all filtrations $f$ of $x$ with
            $\mathrm{wt}(f^*\mathcal{L})=0$,
            $f(0)$ is $\mathcal{L}$-semistable and there exists an opposite
            filtration $f'$ to $f$.
    \end{enumerate}
\end{definition}

For appropriate choices of $\mathcal{L}$, we can recover classical
stability conditions, such as slope-stability for vector bundles in \cite[Lemma
1.11]{heinloth_hilbertmumfordstability} and Ramanathan-stability
(\cite[Definition 1.1]{ramanathan_stableprincipalbundles}) for
principal bundles in \cite[Corollary
1.16]{heinloth_hilbertmumfordstability}. Heinloth's definition
includes $\dim_K(\mathrm{Aut}_\mathcal{M}(x))=0$ for
$\mathcal{L}$-stability, which is
useful for obtaining coarse moduli spaces, but does not account for
\textit{central scalar automorphisms} in moduli problems of bundles,
so we remove this condition (which is not a significant alteration as
    remarked in \cite[Remark 1.3.
(4)]{heinloth_hilbertmumfordstability}).

In order to apply our Jordan--H{\"o}lder theory for complementary
polyhedra in this context, we will impose enough conditions on
$\mathcal{M}$ so that $\mathcal{L}$-stability can be modeled by complementary
polyhedra. We will relate facets to certain filtrations of $\mathcal{M}$,
requiring \textit{quasi-refinements} of filtrations.

\begin{definition}[(Quasi-refinements)]\label{def_quasirefinements}
    Let $x\in\mathcal{M}(K)$ be a $K$-point, and let
    $f:\Theta_K\rightarrow\mathcal{M}$
    be a filtration of $x$.
    \begin{enumerate}[label=(\alph*)]
        \item\label{def_quasirefinements1} A filtration
            $f':\Theta_K\rightarrow\mathcal{M}$ of $x$ is a
            \textit{quasi-refinement} of $f$ if there exists a filtration
            $\widetilde{f}:\Theta_K\rightarrow\mathcal{M}$ of $f(0)$ such that
            $f'(0)=\widetilde{f}(0)$. We call $f'$ \textit{proper} if
            $f'(0)\neq f(0)$, and we call $f'$ \textit{non-proper} if
            $f'(0)= f(0)$.

        \item\label{def_quasirefinements2} A filtration $
            g:\Theta_K\rightarrow\mathcal{M}$ of $x$ is \textit{full}
            if for every
            quasi-refinement $g'$ of $g$, we have $g'(0)=g(0)$.
    \end{enumerate}
\end{definition}

In the case of vector bundles, quasi-refinements are a weaker notion
than refinements of filtrations of vector bundles, i.e., all
refinements of filtrations of vector bundles form quasi-refinements of
the stack $\mathcal{M}:=\mathcal{M}\textsc{Bun}_{r,d}$, but the
reverse is not true in
general. For example, for two vector bundles $F$ and $L$ over $X$,
the following filtrations of vector bundles do not refine each other:
\begin{align*}
    0\subset L\subset L\oplus F& &0\subset F\subset L\oplus F
\end{align*}
however, through the Rees construction, their filtrations $f$ and
$f'$ do quasi-refine each other. The difference arises because
quasi-refinements only compare Levi-factors.
Note that full filtrations of the stack induce full
filtrations of vector bundles.

\begin{lemma}[(Full quasi-refinements)]\label{lem_fullquasirefinement}
    Let $\mathcal{M}$ be quasi-separated and locally of finite type, and let
    $\mathcal{M}$ have affine stabilizers, then every filtration
    $f:\Theta_K\rightarrow\mathcal{M}$ of $x$ has a full quasi-refinement
    $g:\Theta_K\rightarrow\mathcal{M}$.
\end{lemma}

\begin{proof}
    $\mathcal{M}$ fulfills the properties of \cite[Lemma
    4.4.6]{danielhalpernleistner_onthestructureofinstabilityinmodulitheory}
    and \cite[Lemma
    3.2]{alperDHLheinloth_existenceofmodulispacesforalgebraicstacks},
    so there exists an affine $K$-scheme $B$ with a
    $\mathbb{G}_{m,K}^n$-action, such that there exists a surjection
    $[B/\mathbb{G}_{m,K}^n]\rightarrow\mathcal{M}$. Furthermore, for
    any $K$-point $b$
    of $[B/\mathbb{G}_{m,K}^n]$, there exists a surjection of filtrations:
    \begin{equation*}
        \mathrm{Filt}_b([B/\mathbb{G}_{m,K}^n])\rightarrow\mathrm{Filt}_x(\mathcal{M})
    \end{equation*}
    This surjection is compatible with quasi-refinements, and proper
    quasi-refinements in $\mathrm{Filt}_x(\mathcal{M})$ arise from proper
    quasi-refinements in $\mathrm{Filt}_b([B/\mathbb{G}_{m,K}^n])$. Hence, it
    suffices to prove the claim for filtrations in
    $\mathrm{Filt}_b([B/\mathbb{G}_{m,K}^n])$.

    Let $f$ be a filtration in $\mathrm{Filt}_b([B/\mathbb{G}_{m,K}^n])$,
    then by \cite[Theorem
    1.4.8]{danielhalpernleistner_onthestructureofinstabilityinmodulitheory}
    and \cite[Proposition
    2.11]{alperDHLheinloth_existenceofmodulispacesforalgebraicstacks},
    $f$ determines a \textit{cocharacter}
    $\lambda:\mathbb{G}_{m,K}\rightarrow\mathbb{G}_{m,K}^n$ such that the
    \textit{attractor} $\lim_{a\rightarrow0}\lambda(a)\cdot b$ is
    well-defined in $B(K)$ and maps to $f(0)$. If there were no full
    quasi-refinement of $f$, then we could find a chain of successive
    proper quasi-refinements $(f'_i)_{i\in\mathbb{N}}$ of $f$ whose
    cocharacters $\lambda'_i:\mathbb{G}_{m,K}\rightarrow\mathbb{G}_{m,K}^n$,
    $i\in\mathbb{N}$, have the following property: The corresponding sequence
    through the isomorphism
    $\mathcal{X}_*(\mathbb{G}_{m,K}^n)\cong\mathbb{Z}^n$ would
    have strictly increasing support in $\mathbb{Z}^n$ with respect to the
    standard $\mathbb{Z}$-basis, leading to a contradiction.
\end{proof}

\subsection{Bundle moduli problems}\label{subsec_bundlemoduliproblem}

Recall that any cocharacter
$\lambda:\mathbb{G}_{m,K}\to\mathbf{GL}(n,K)$ induces a
\textit{parabolic subgroup} $\mathsf{P}(\lambda)$ of $\mathbf{GL}(n,K)$:
\begin{align*}
    \mathsf{P}(\lambda) &
    :=\{g\in\mathbf{GL}(n,K)\ \vert\ \lim_{a\rightarrow0}\lambda(a)g\lambda(a)^{-1}\text{
    exists}\} \\
    \intertext{with the \textit{Levi-factor}
        $\mathsf{L}:=\mathsf{P}(\lambda)/\mathbf{R}_u(\mathsf{P}(\lambda))$
    (quotient by the \textit{unipotent radical}), such that:}
    \mathsf{L}\cong\mathsf{L}(\lambda)&:=
    \{g\in\mathbf{GL}(n,K)\ \vert\ \lim_{a\rightarrow0}\lambda(a)g\lambda(a)^{-1}=g\}
\end{align*}
In the setup of Lemma \ref{lem_fullquasirefinement}, we have the relations
on filtrations $f$ and $f'$ of $x$, using the surjections
$\mathrm{Filt}_b([B/\mathbb{G}_{m,K}^n])\rightarrow\mathrm{Filt}_x(\mathcal{M})$
from \cite[Lemma
4.4.6]{danielhalpernleistner_onthestructureofinstabilityinmodulitheory}
and \cite[Lemma
3.2]{alperDHLheinloth_existenceofmodulispacesforalgebraicstacks}:
\begin{align*}
    f \sim_\textsc{P} f'
    \Leftrightarrow\; &
    \text{Through any choice of } [B/\mathbb{G}_{m,K}^n],\;
    f \text{ and } f' \text{ correspond to cocharacters } \\
    &
    \lambda,\mu:\mathbb{G}_{m,K}\to\mathbb{G}_{m,K}^n\subset\mathbf{GL}(n,K)
    \text{ such that }\mathsf{P}(\lambda)\cong\mathsf{P}(\mu) \\
    f \sim_\textsc{L} f'
    \Leftrightarrow\; &
    \text{Through any choice of } [B/\mathbb{G}_{m,K}^n],\;
    f \text{ and } f' \text{ correspond to cocharacters } \\
    &
    \lambda,\mu:\mathbb{G}_{m,K}\to\mathbb{G}_{m,K}^n\subset\mathbf{GL}(n,K)
    \text{ such that }\mathsf{L}(\lambda)\cong\mathsf{L}(\mu)
\end{align*}
These relations extend to equivalence relations on filtrations
$\mathrm{Filt}_x(\mathcal{M})$, which we use to define \textit{bundle moduli
problems}. We also have $f \sim_\textsc{L} f'$ if and only if $f(0)=f'(0)$.

\begin{definition}[(Bundle moduli problems)]\label{def_bundlemoduliproblem}
    Let $\mathcal{M}$ be quasi-separated and locally of finite type, and let
    $\mathcal{M}$ have affine stabilizers.
    \begin{enumerate}[label=(\alph*)]
        \item\label{def_bundlemoduliproblem1} The pair
            $(\mathcal{M},d_{\mathcal{M}})$ is a
            \textit{bundle moduli problem} if for all $x\in\mathcal{M}(K)$,
            all filtrations $f$ and $f'$ of $x$ quasi-refine to a
            full filtration $g:\Theta_K\rightarrow\mathcal{M}$ of $x$.
            Furthermore, for all equivalence classes $[g]_{\textsc{L}}$ of
            full filtrations $g:\Theta_K\rightarrow\mathcal{M}$ of $x$,
            there exists a root system
            $\Phi_{x,[g]_{\textsc{L}}}$ of a Euclidean space
            $(V_{x,[g]_{\textsc{L}}},\langle\_,\_\rangle_{x,[g]_{\textsc{L}}})$
            such that:
            \begin{enumerate}[label=(\roman*)]
                \item\label{def_bundlemoduliproblem11}
                    $\sim_{\textsc{P}}$-equivalence classes of filtrations of
                    $x$ inject into facets $P$, so we can give the
                    labeling $[P]_{\textsc{P}}$, such
                    that $f\in [P]_{\textsc{P}}$ quasi-refines to a
                    representative of $[g]_{\textsc{L}}$. Facets $P$ with an
                    equivalence class are called \textit{admissible}.

                \item\label{def_bundlemoduliproblem12} The order relation
                    $P\geqslant Q$ on admissible facets implies that
                    representatives
                    $f\in [P]_{\textsc{P}}$ quasi-refine to representatives of
                    $[Q]_{\textsc{P}}$. If furthermore $P\neq Q$, i.e. $P>Q$,
                    then representatives $f\in [P]_{\textsc{P}}$ have proper
                    quasi-refinements to representatives of $[Q]_{\textsc{P}}$.

                \item\label{def_bundlemoduliproblem13} For any $f\in
                    [P]_{\textsc{P}}$ for an admissible facet $P$,
                    there exists a
                    filtration $g'$ of $f(0)$ such that $g(0)=g'(0)$. For any
                    such $g'$, we have:
                    \begin{align*}
                        \Phi_{x,[g]_{\textsc{L}}}=\Phi_{f(0),[g']_{\textsc{L}}}
                        &
                        &
                        (V_{x,[g]_{\textsc{L}}},\langle\_,\_\rangle_{x,[g]_{\textsc{L}}})=(V_{f(0),[g']_{\textsc{L}}},\langle\_,\_\rangle_{f(0),[g']_{\textsc{L}}})
                    \end{align*}

                \item\label{def_bundlemoduliproblem14}
                    $d_{\mathcal{M}}:=(d_{x,[g]_{\textsc{L}}})_{x,[g]_{\textsc{L}}}$
                    is the
                    data of nonempty families $(d_{x,[g]_{\textsc{L}}}^C)_{C\in
                    I_{x,[g]_{\textsc{L}}}}$ of complementary polyhedra
                    $d_{x,[g]_{\textsc{L}}}^C:\mathfrak{W}_{\Phi_{x,[g]_{\textsc{L}}}}\rightarrow
                    V_{x,[g]_{\textsc{L}}}^*$ such that for any $f\in
                    [P]_{\textsc{P}}$ for an admissible facet $P$, and for any
                    full filtration $g'$ of $f(0)$ such that $g(0)=g'(0)$, we
                    have
                    $(d_{x,[g]_{\textsc{L}}}^C)_P=d_{f(0),[g']_{\textsc{L}}}^C$.
                    Furthermore, with respect to any choice
                    of $x\in\mathcal{M}(K)$, $[g]_{\textsc{L}}$, and $C\in
                    I_{x,[g]_{\textsc{L}}}$, we have for the
                    complementary polyhedron $d_{x,[g]_{\textsc{L}}}^C$:
                    \begin{itemize}
                        \item For any non-admissible facet $P$, we
                            have $\deg(P)<0$.

                        \item For any admissible facet $P$, and for
                            any $f\in[P]_{\textsc{P}}$ with another
                            full quasi-refinement $g'$, the sign of
                            the numerical invariants of $P$, with
                            respect to $d_{x,[g]_{\textsc{L}}}^C$ and
                            $d_{x,[g']_{\textsc{L}}}^{C'}$, for any
                            $C'\in I_{x,[g']_{\textsc{L}}}$, are the
                            same. If
                            $[g]_{\textsc{L}}=[g']_{\textsc{L}}$, the
                            numerical invariants are equal.
                    \end{itemize}

                    \sloppy
                \item\label{def_bundlemoduliproblem15} For any
                    $K$-point $x$ of $\mathcal{M}$, $x$ is
                    \textit{$d_{\mathcal{M}}$-(semi)-stable} if $d_x$ is
                    (semi)-stable, and $x$ is
                    \textit{$d_{\mathcal{M}}$-polystable} if $d_x$ is
                    polystable. The $d_{\mathcal{M}}$-(semi)-stable loci
                    form substacks
                    $\mathcal{M}^{d_{\mathcal{M}}-\mathrm{st}}$ and
                    $\mathcal{M}^{d_{\mathcal{M}}-\mathrm{sst}}$ of
                    $\mathcal{M}$.
            \end{enumerate}

        \item\label{def_bundlemoduliproblem2} Let $\mathcal{M}$ furthermore have
            affine diagonal, and let $\mathcal{L}$ be a line bundle
            over $\mathcal{M}$.
            Then $\mathcal{L}$ has \textit{compatible stability} to a bundle
            moduli problem $(\mathcal{M},d_{\mathcal{M}})$ if for all
            $x\in\mathcal{M}(K)$, for all
            equivalence classes $[g]_{\textsc{L}}$ of full filtrations of
            $x$, and for all admissible facets $P$ of
            $\Phi_{x,[g]_{\textsc{L}}}$, the sign of $\deg(P)$ is equal to
            the sign of $\mathrm{wt}(f^*\mathcal{L})$ for any
            representative $f\in
            [P]_{\textsc{P}}$.
    \end{enumerate}
\end{definition}

Definition \ref{def_bundlemoduliproblem}
\ref{def_bundlemoduliproblem2} is well-defined due to the properties
of $d_{\mathcal{M}}$ laid out in Definition
\ref{def_bundlemoduliproblem} \ref{def_bundlemoduliproblem1}
\ref{def_bundlemoduliproblem14}. If $\mathcal{L}$ is a line bundle over
$\mathcal{M}$ with compatible stability, then
$\mathcal{L}$-stability
conditions are equivalent to $d_{\mathcal{M}}$-stability conditions.
Moreover, if
$(\mathcal{M},d_{\mathcal{M}})$ has
Definitions \ref{def_bundlemoduliproblem}
\ref{def_bundlemoduliproblem1}
\ref{def_bundlemoduliproblem11}-\ref{def_bundlemoduliproblem14}, and
has a line bundle $\mathcal{L}$ over $\mathcal{M}$ with compatible stability, in
the sense of Definition \ref{def_bundlemoduliproblem}
\ref{def_bundlemoduliproblem2}, then the fact that
$\mathcal{M}^{d_{\mathcal{M}}-\mathrm{st}}$ and
$\mathcal{M}^{d_{\mathcal{M}}-\mathrm{sst}}$ form substacks
in Definition \ref{def_bundlemoduliproblem}
\ref{def_bundlemoduliproblem1} \ref{def_bundlemoduliproblem15} comes
directly from \cite[Remark 1.3.
(5)]{heinloth_hilbertmumfordstability}.

The stability
conditions of
$d_x:=(d_{x,[g]_{\textsc{L}}})_{[g]_{\textsc{L}}}$
only depend on $x$ and $[g]_{\textsc{L}}$,
not $C\in I_{x,[g]_{\textsc{L}}}$.
Due to Definition
\ref{def_bundlemoduliproblem} \ref{def_bundlemoduliproblem1}
\ref{def_bundlemoduliproblem13}, two equivalence classes
$[P]_{\textsc{P}}$ and $[Q]_{\textsc{P}}$ have $f \sim_\textsc{L} f'$,
for any $f\in [P]_{\textsc{P}}$ and $f'\in [Q]_{\textsc{P}}$, if and only
if their corresponding admissible facets $P$ and $Q$ fulfill
$P^\perp=Q^\perp$.

In Section \ref{sec_applicationstobundles}, we will see examples of
bundle moduli problems, including \textit{vector bundles} and \textit{parahoric
(Higgs) torsors}. We now define \textit{Jordan--H{\"o}lder filtrations}
and \textit{S-equivalence}.

\begin{definition}[(S-equivalence of bundle moduli
    problems)]\label{def_Sequivalence}
    For a bundle moduli problem $(\mathcal{M},d_{\mathcal{M}})$, let
    $x,y\in\mathcal{M}(K)$ be points
    that are $d_{\mathcal{M}}$-semistable.
    \begin{enumerate}[label=(\alph*)]
        \item\label{def_Sequivalence1} A filtration
            $f:\Theta_K\rightarrow\mathcal{M}$ of $x$ is
            \textit{Jordan--H{\"o}lder} if there exists a full filtration
            $g:\Theta_K\rightarrow\mathcal{M}$ of $x$ such that $f\in
            [P]_{\textsc{P}}$ for a Jordan--H{\"o}lder facet $P$ of
            $d_{x,[g]_{\textsc{L}}}$. Moreover, $P$ is maximal in
            dimension amongst all Jordan--H{\"o}lder facets and full
            filtrations $g$.

        \item\label{def_Sequivalence2} The points $x$ and $y$ are
            \textit{S-equivalent} if they
            have Jordan--H{\"o}lder filtrations $f$ and $f'$ respectively such
            that $f(0)=f'(0)$.
    \end{enumerate}
\end{definition}

Due to \cite[Propositions 1.9 and
3.2]{behrend_semi-stabilityofreductivegroupschemesovercurves},
Jordan--H{\"o}lder facets $P$ have $\deg(P)=0$, so they are
admissible in the sense of Definition \ref{def_bundlemoduliproblem}
\ref{def_bundlemoduliproblem1} \ref{def_bundlemoduliproblem11}. By
construction, Jordan--H{\"o}lder filtrations factorize through
$\mathcal{M}^{d_{\mathcal{M}}-\mathrm{sst}}$. We claim that
Jordan--H{\"o}lder filtrations of $d_{\mathcal{M}}$-semistable points
always exist, ensuring that S-equivalence defines an equivalence relation.

\begin{lemma}[(Jordan--H{\"o}lder filtrations of bundle moduli
    problems)]\label{lem_jhfiltrationsmodulistack}
    For a bundle moduli problem $(\mathcal{M},d_{\mathcal{M}})$ and a
    $K$-point $x$ that is
    $d_{\mathcal{M}}$-semistable, there exists a Jordan--H{\"o}lder
    filtration $f$ of
    $x$, unique up to Levi-factor $f(0)$, such that $f(0)$ is
    $d_{\mathcal{M}}$-polystable.
\end{lemma}

\begin{proof}
    By Lemma \ref{lem_fullquasirefinement}, full filtrations $g$ of $x$
    exist, and by Theorem \ref{thm_jhfacetexistence}, we know that
    Jordan--H{\"o}lder facets $P$ of $\Phi_{x,[g]_{\textsc{L}}}$ of
    maximal dimension exist, which are admissible due to
    \cite[Propositions 1.9 and
    3.2]{behrend_semi-stabilityofreductivegroupschemesovercurves}
    implying $\deg(P)=0$.
    Hence, $P$ is admissible and there exists a Jordan--H{\"o}lder
    filtration $f\in [P]_{\textsc{P}}$. Since two Jordan--H{\"o}lder filtrations
    $f$ and $f'$ quasi-refine to the same full filtration $g$, Theorem
    \ref{thm_jhfacetuniqueness} implies uniqueness up to Levi-factor
    $f(0)=f'(0)$, and by Lemma
    \ref{lem_polystabilityequivalencepolyhedra} and
    Definition \ref{def_bundlemoduliproblem} \ref{def_bundlemoduliproblem1}
    \ref{def_bundlemoduliproblem13}, we have the
    $d_{\mathcal{M}}$-polystability of $f(0)$.
\end{proof}

For the rest of this section, we assume that $K$ is algebraically
closed of characteristic $0$, as we will work with \textit{good
moduli spaces}, in the sense of Alper in
\cite{alper_goodmodulispacesforartinstacks}. However, we can relax
the characteristic $0$ condition by replacing mentions of good moduli
space with \textit{adequate moduli space}, from
\cite{alper_adequatemodulispacesandgeometricallyreductivegroupschemes}.
When the $d_{\mathcal{M}}$-semistable locus of a bundle moduli
problem $(\mathcal{M},d_{\mathcal{M}})$ exhibits a good moduli space
$\pi:\mathcal{M}^{d_{\mathcal{M}}-\mathrm{sst}}\rightarrow
M^{d_{\mathcal{M}}-\mathrm{sst}}$, the natural question to ask is how
S-equivalence from Definition \ref{def_Sequivalence}
\ref{def_Sequivalence2} compares to Zhang's notion of
\textit{quasi-S-equivalence}.

\begin{definition}[(Quasi-S-equivalence of good moduli spaces
            {\cite[Definitions 2.3 and
    2.11]{zhang_Sequivalenceforalgebraicstacks}})]\label{def_quasiSequivalence}
    Let $\mathcal{M}$ be an algebraic stack over $K$, and let
    $x,y\in\mathcal{M}(K)$ be points.
    \begin{enumerate}[label=(\alph*)]
        \item\label{def_quasiSequivalence1} A filtration
            $f:\Theta_K\rightarrow\mathcal{M}$ of $x$ is
            \textit{quasi-Jordan--H{\"o}lder} if the closure
            $\mathrm{cl}(x)$ of $x$ contains a unique closed point
            $x_0\in\mathcal{M}(K)$.

        \item\label{def_quasiSequivalence2} The points $x$ and $y$ are
            \textit{quasi-S-equivalent} if there are filtrations $f$
            of $x$ and $f'$ of $y$ such that $f(0)=f'(0)$.
    \end{enumerate}
\end{definition}

In \cite[Theorem 3.1]{zhang_Sequivalenceforalgebraicstacks}, Zhang
proves that when there exists a good moduli space
$\pi:\mathcal{M}\rightarrow M$, $\pi$ precisely identifies
quasi-S-equivalence classes through $\pi$-preimages of points.
Moreover, quasi-S-equivalence classes can be determined by
quasi-Jordan--H{\"o}lder filtrations. When
$(\mathcal{M},d_{\mathcal{M}})$ exhibits a good moduli space
$\pi:\mathcal{M}^{d_{\mathcal{M}}-\mathrm{sst}}\rightarrow
M^{d_{\mathcal{M}}-\mathrm{sst}}$, it is clear that S-equivalence
from Definition \ref{def_Sequivalence} implies quasi-S-equivalence in
$\mathcal{M}^{d-\mathrm{sst}}$. We are frequently in this situation,
as good moduli spaces have been constructed for vector bundles,
parahoric torsors by Heinloth in \cite[Theorem
3.19]{heinloth_hilbertmumfordstability}, and parahoric Higgs torsors
by R{\'e}ga in \cite[Theorem 5.26]{rega_phdthesis}.

\subsection{Jordan--H{\"o}lder stratifications}\label{subsec_jhstratifications}
\sloppy
It remains to study when quasi-S-equivalence implies S-equivalence,
as is the case for vector bundles by Zhang in \cite[Lemma
2.7]{zhang_Sequivalenceforalgebraicstacks}, for which
Jordan--H{\"o}lder stratifications are useful. A bundle moduli
problem $(\mathcal{M},d_{\mathcal{M}})$ induces a partition of
$K$-points, called the \textit{Jordan--H{\"o}lder stratification}:
\begin{equation*}
    \mathcal{M}^{d_{\mathcal{M}}-\mathrm{sst}}(K)=\bigsqcup_{\mu\in\mathrm{JH}_{(\mathcal{M},d_{\mathcal{M}})}}\mathcal{M}_\mu^{d_{\mathcal{M}}-\mathrm{sst}}(K)
\end{equation*}
by identifying $K$-points
$x,y\in\mathcal{M}^{d_{\mathcal{M}}-\mathrm{sst}}(K)$ in the same stratum
$\mathcal{M}_{\mu}^{d_{\mathcal{M}}-\mathrm{sst}}(K)$ if there exist
linear orthogonal isomorphisms
$(V_{x,[g]_{\textsc{L}}},\langle\_,\_\rangle_{x,[g]_{\textsc{L}}})_{[g]_{\textsc{L}}}\cong(V_{y,[g']_{\textsc{L}}},\langle\_,\_\rangle_{x,[g]_{\textsc{L}}})_{[g']_{\textsc{L}}}$
identifying the root systems and Jordan--H{\"o}lder facets of $d_x$ and
$d_y$ of maximal dimension. We label the index set of
Jordan--H{\"o}lder strata by
$\mathrm{JH}_{(\mathcal{M},d_{\mathcal{M}})}$, and the stratum
corresponding to the
Jordan--H{\"o}lder facet $\{0\}$ by
$\mathcal{M}_0^{d_{\mathcal{M}}-\mathrm{sst}}(K)$, which is the
$d_{\mathcal{M}}$-stable locus $\mathcal{M}^{d_{\mathcal{M}}-\mathrm{st}}(K)$.

By construction, the strata
$\mathcal{M}_\mu^{d_{\mathcal{M}}-\mathrm{sst}}(K)$ are
compatible with S-equivalence, i.e.,
$x$ is in $\mathcal{M}_{\mu}^{d_{\mathcal{M}}-\mathrm{sst}}(K)$ if
and only if for any $K$-point
$y\in\mathcal{M}^{d_{\mathcal{M}}-\mathrm{sst}}(K)$, S-equivalent to
$x$, in the sense of
Definition \ref{def_Sequivalence} \ref{def_Sequivalence2}, is in
$\mathcal{M}_{\mu}^{d_{\mathcal{M}}-\mathrm{sst}}(K)$. Hence, given a
good moduli
space $\pi:\mathcal{M}^{d_{\mathcal{M}}-\mathrm{sst}}\rightarrow
M^{d_{\mathcal{M}}-\mathrm{sst}}$,
$\mathcal{M}_{\mu}^{d_{\mathcal{M}}-\mathrm{sst}}(K)$ are the
$K$-points inside the $\pi$-preimage of a
subset of $M^{d_{\mathcal{M}}-\mathrm{sst}}_{\mu}(K)$ of
$M^{d_{\mathcal{M}}-\mathrm{sst}}(K)$.

In some respects, this stratification is similar to
$\Theta$-stratifications of algebraic stacks from
\cite[\S2]{danielhalpernleistner_onthestructureofinstabilityinmodulitheory},
which stratifies $\mathcal{M}$:
\begin{equation*}
    \mathcal{M}=\bigsqcup_{c\in\Gamma}\mathcal{M}_{\leqslant c}
\end{equation*}
into \textit{locally closed} substacks, whose open stratum is
$\mathcal{M}_{\leqslant0}$ called the \textit{semistable locus}. For a bundle
moduli problem $(\mathcal{M},d_{\mathcal{M}})$ with a compatible
$\Theta$-stratification, we have
$\mathcal{M}_{\leqslant0}=\mathcal{M}^{d_{\mathcal{M}}-\mathrm{sst}}$.
However, Jordan--H{\"o}lder
stratifications of bundle moduli problems have less geometry, as it
is unclear whether the strata
$\mathcal{M}_\mu^{d_{\mathcal{M}}-\mathrm{sst}}(K)$ form substacks
$\mathcal{M}_\mu^{d_{\mathcal{M}}-\mathrm{sst}}$ of
$\mathcal{M}^{d_{\mathcal{M}}-\mathrm{sst}}$,
and if so, whether the
substacks $\mathcal{M}_\mu^{d_{\mathcal{M}}-\mathrm{sst}}$ are locally closed.

\begin{definition}[(Jordan--H{\"o}lder
    stratifications)]\label{def_jhstratification}
    Let $(\mathcal{M},d)$ be a bundle moduli problem.
    \begin{enumerate}[label=(\alph*)]
        \item\label{def_jhstratification1}
            $(\mathcal{M},d_{\mathcal{M}})$ has a \textit{locally
            closed Jordan--H{\"o}lder stratification} if its
            Jordan--H{\"o}lder strata form locally closed substacks
            $\mathcal{M}^{d_{\mathcal{M}}-\mathrm{sst}}_{\mu}$ of
            $\mathcal{M}^{d_{\mathcal{M}}-\mathrm{sst}}$.
        \item\label{def_jhstratification2}
            $(\mathcal{M},d_{\mathcal{M}})$ exhibits a
            \textit{compatible good moduli space}
            $\pi:\mathcal{M}^{d_{\mathcal{M}}-\mathrm{sst}}\rightarrow
            M^{d_{\mathcal{M}}-\mathrm{sst}}$ if quasi-equivalence
            implies S-equivalence.
    \end{enumerate}
\end{definition}

Our strategy to prove that a good moduli space is compatible is to
prove that $(\mathcal{M},d)$ has a locally closed Jordan--H{\"o}lder
stratification, then that the $d_{\mathcal{M}}$-polystable points of
every Jordan--H{\"o}lder stratum
$\mathcal{M}^{d_{\mathcal{M}}-\mathrm{sst}}_{\mu}$ biject to
$M^{d_{\mathcal{M}}-\mathrm{sst}}_{\mu}(K)$. Hence, since
$M^{d_{\mathcal{M}}-\mathrm{sst}}_{\mu}$ is a locally closed subspace
of $M^{d_{\mathcal{M}}-\mathrm{sst}}$, this assigns every
quasi-equivalence class a unique S-equivalence class, and every
$d_{\mathcal{M}}$-polystable point is a closed point. We will see
this approach applied in Section \ref{sec_applicationstobundles}.

\section{Applications}\label{sec_applicationstobundles}

Using bundle moduli problems from Section \ref{sec_modulistacks}, we
apply our Jordan--H{\"o}lder theory for complementary polyhedra to
\textit{principal bundles} and \textit{parahoric (Higgs) torsors},
covering some moduli problems mentioned in the introduction
in \ref{intro_applications}. This produces new results for the
S-equivalence of parahoric (Higgs)
torsors, and recovers known results for S-equivalence for vector
bundles in \cite{seshadri_spaceofunitaryvectorbundles}, for principal
bundles in
\cite{ramanathan_moduliforprincipalbundlesI,ramanathan_moduliforprincipalbundlesII},
for
\textit{Higgs bundles} in
\cite{granaotero_jhreductionsforprincipalhiggsbundlesoncurves}, and
\textit{parabolic bundles} in \cite[\S7.2]{henke_phdthesis}.

\subsection{Principal bundles}\label{subsec_principal}

From the introduction in \ref{intro_applications}
\ref{intro_applications1}, we apply our Jordan--H{\"o}lder theory to
principal $\mathsf{G}$-bundles
$\xi$ over $X$, where $\mathsf{G}$ is a split reductive $K$-group.

\subsubsection{Definition and stability of principal
bundles}\label{subsubsec_stabilityprincipal}

We first show that the stability conditions of the
\textit{automorphism group scheme} $\mathrm{Ad}(\xi)$, from Definition
\ref{def_groupschemestability}, are equivalent to the
\textit{Ramanathan-stability} conditions of $\xi$ from
\cite[Definition 1.1]{ramanathan_stableprincipalbundles}, and we
introduce \textit{Ramanathan-polystability}.
\begin{definition}[(Ramanathan-polystability)]\label{def_ramanathanpolystability}
    A principal $\mathsf{G}$-bundle $\xi$ over $X$ is
    \textit{Ramanathan-polystable} if for all reductions
    $s:X\rightarrow\xi/\mathsf{P}$ to maximal parabolic subgroups
    $\mathsf{P}$ with $\deg(s^*V_{\xi/\mathsf{P}})=0$, where
    $V_{\xi/\mathsf{P}}$ is the \textit{vertical tangent bundle} of
    $\xi/\mathsf{P}$, we have:
    \begin{enumerate}[label=(\roman*)]
        \item\label{def_ramanathanpolystability1} The extension
            $(s^*\xi)(\mathsf{L})$
            to the \textit{Levi-factor}
            $\mathsf{L}:=\mathsf{P}/\mathbf{R}_u(\mathsf{P})$,
            through the quotient
            $\mathsf{P}\rightarrow\mathsf{L}=\mathsf{P}/\mathbf{R}_u(\mathsf{P})$
            by the
            \textit{unipotent radical}, is Ramanathan-semistable.

        \item\label{def_ramanathanpolystability2} There exists an
            \textit{opposite} reduction $s':X\rightarrow\xi/\mathsf{P}'$ to
            $s$, where $\mathsf{P}'$ is \textit{opposite} to $\mathsf{P}$,
            i.e., $\mathsf{P}\cap\mathsf{P}'$ is isomorphic to $\mathsf{L}$
            through the projection
            $\mathsf{P}\rightarrow\mathsf{L}=\mathsf{P}/\mathbf{R}_u(\mathsf{P})$,
            and there exists an automorphism
            $\tau:\mathsf{G}\rightarrow\mathsf{G}$ sending $\mathsf{P}$ to
            $\mathsf{P}'$, inducing an isomorphism
            $f:\xi/\mathsf{P}\rightarrow\xi/\mathsf{P}'$, such that
            $s'=f\circ s$.
    \end{enumerate}
\end{definition}

We can identify the Ramanathan-stability conditions of $\xi$ with the
stability conditions of the automorphism group scheme $\mathrm{Ad}(\xi)$.

\begin{lemma}[(Stability conditions for principal bundles
    {\cite[\protect\S8]{behrend_semi-stabilityofreductivegroupschemesovercurves}})]\label{lem_stabilityequivalenceprincipal}
    Let $\xi$ be a principal $\mathsf{G}$-bundle over $X$.
    \begin{enumerate}[label=(\roman*)]
        \item\label{lem_stabilityequivalenceprincipal1} $\xi$
            is Ramanathan-(semi)-stable if and only if $\mathrm{Ad}(\xi)$ is
            (semi)-stable.

        \item\label{lem_stabilityequivalenceprincipal2} $\xi$
            is Ramanathan-polystable if and only if
            $\mathrm{Ad}(\xi)$ is polystable.
    \end{enumerate}
\end{lemma}

\begin{proof}
    For
    \ref{lem_stabilityequivalenceprincipal1}, by Remark
    \ref{rem_maximalparabolicssuffice}, it suffices to check
    inequalities for maximal parabolic subgroup schemes of $\mathrm{Ad}(\xi)$.
    Let $\mathsf{P}$ be a maximal parabolic subgroup of $\mathsf{G}$,
    and let $s:X\rightarrow\xi/\mathsf{P}$ be a reduction to
    $\mathsf{P}$. By \cite[Exp. XXVI, 3.20]{sga3},
    $s:X\rightarrow\xi/\mathsf{P}$ corresponds to a maximal parabolic
    subgroup scheme $\mathrm{Ad}(s^*\xi)$ of $\mathrm{Ad}(\xi)$. Since:
    \begin{equation*}
        \deg(\mathrm{ad}(s^*\xi))=\deg(\mathrm{Lie}(\mathrm{Ad}(s^*\xi)))=\deg(\mathrm{Ad}(s^*\xi))
    \end{equation*}
    we have $\deg(\mathrm{ad}(s^*\xi))\semileq 0$ if and only if
    $\deg(\mathrm{Ad}(s^*\xi))\semileq 0$, with Notation
    \ref{rem_semistabilitynotation}.

    For \ref{lem_stabilityequivalenceprincipal2}, due to \cite[Exp.
    XXIV, 1.5]{sga3}, we know that $\mathrm{Ad}(\xi)$ is
    rationally split.
    By Remark
    \ref{rem_maximalparabolicssuffice}, it suffices to check the
    required conditions for maximal parabolics. Let $\mathsf{P}$ be a
    maximal parabolic subgroup of $\mathsf{G}$, and let
    $s:X\rightarrow\xi/\mathsf{P}$ be a reduction such that
    $\deg(s^*V_{\xi/\mathsf{P}})=0$. By \cite[Exp. XXVI, 3.20]{sga3},
    $s:X\rightarrow\xi/\mathsf{P}$ corresponds to a parabolic subgroup
    scheme $\mathbf{P}$ of $\mathrm{Ad}(\xi)$, then using
    \ref{lem_stabilityequivalenceprincipal1}, the
    Ramanathan-polystability conditions of $s^*\xi$ translate to
    \ref{def_polystabilitygroupscheme1} and
    \ref{def_polystabilitygroupscheme2} for $\mathbf{P}$.
\end{proof}

\subsubsection{Jordan--H{\"o}lder theory for principal
bundles}\label{subsubsec_jhprincipal}

Through our results, we can reprove the Jordan--H{\"o}lder theorem of
Ramanathan for principal bundles.

\begin{theorem}[(Jordan--H{\"o}lder theorem for principal bundles
            {\cite[3.12
    Proposition]{ramanathan_moduliforprincipalbundlesI}})]\label{thm_jhprincipal}
    For a Ramanathan-semistable principal $\mathsf{G}$-bundle $\xi$,
    there exists a
    reduction $s:X\rightarrow\xi/\mathsf{P}$ to a parabolic subgroup
    $\mathsf{P}$ of $\mathsf{G}$ such that:
    \begin{enumerate}[label=(\roman*)]
        \item\label{thm_jhprincipal1} The reduction $s$ is
            \textit{admissible}, i.e., for any character
            $\chi:\mathsf{P}\rightarrow\mathbb{C}^\times$ that is
            admissible, i.e., trivial on the
            center $\mathbf{Z}(\mathsf{P})$, we have $\deg(\chi(s^*\xi))=0$.

        \item\label{thm_jhprincipal2} The extension to the Levi-factor
            $(s^*\xi)(\mathsf{L})$ is Ramanathan-stable.
    \end{enumerate}
    These reductions have the same extension $(s^*\xi)(\mathsf{L})$ to
    the Levi-factor, up to an isomorphism.
\end{theorem}

This generalizes the Jordan--H{\"o}lder theorem of vector bundles from
Seshadri in Theorem \ref{thm_jhvectorbundles} and \cite[Proposition
3.1]{seshadri_spaceofunitaryvectorbundles}. We call such reductions
\textit{Jordan--H{\"o}lder reductions} of $\xi$.

\begin{proof}
    We
    know that $\mathrm{Ad}(\xi)$ is rationally split due to \cite[Exp. XXIV,
    1.5]{sga3}. Thus,
    $d:=(d_{\mathrm{Ad}(\xi)_{\eta},\mathbf{T}_{\eta}})_{\mathbf{T}_\eta}$,
    for all maximal tori $\mathbf{T}_{\eta}$ of
    $\mathrm{Ad}(\xi)_{\eta}$, forms a compatible family of
    complementary polyhedra for $\mathrm{Ad}(\xi)_{\eta}$, using Lemma
    \ref{lem_complementarypolyhedrongroupscheme}. Using Lemmas
    \ref{lem_stabilityequivalencegroupscheme}
    \ref{lem_stabilityequivalencegroupscheme1} and
    \ref{lem_stabilityequivalenceprincipal}
    \ref{lem_stabilityequivalenceprincipal1}, we have that
    $\mathrm{Ad}(\xi)_\eta$ is $d$-semistable, so Theorem
    \ref{thm_jhgroupscheme} \ref{thm_jhgroupscheme1} is applicable,
    which gives us
    Jordan--H{\"o}lder parabolics $\mathbf{P}$ of $\mathrm{Ad}(\xi)$. We now use
    \cite[Exp. XXVI, 3.20]{sga3} to associate parabolic reductions
    $s:X\rightarrow\xi/\mathsf{P}$ to $\mathbf{P}$, such that
    $\mathbf{P}=\mathrm{Ad}(s^*\xi)$. Since $\mathbf{P}$ is
    Jordan--H{\"o}lder, Lemmas
    \ref{lem_stabilityequivalencegroupscheme}
    \ref{lem_stabilityequivalencegroupscheme1} and
    \ref{lem_stabilityequivalenceprincipal}
    \ref{lem_stabilityequivalenceprincipal1} translate the
    stability of the Levi-factor $\mathbf{L}$ to the Ramanathan-stability of
    the Levi-factor $(s^*\xi)(\mathsf{L})$, applying the correspondence
    of \cite[Exp. XXVI, 3.20]{sga3} to the Levi-factors. The numerical
    invariants of $\mathbf{P}$ are $0$, so we get that $\deg(\chi(s^*\xi))=0$
    for all admissible characters $\chi:\mathsf{P}\rightarrow
    K^\times$, using \cite[Proposition
    6.8]{behrend_semi-stabilityofreductivegroupschemesovercurves}.
    For the uniqueness of the Levi-factor, for two Jordan--H{\"o}lder
    reductions $s_1:X\rightarrow\xi/\mathsf{P}_1$ and
    $s_1:X\rightarrow\xi/\mathsf{P}_2$, there exists a maximal torus
    $\mathbf{T}_\eta$ inside the corresponding Jordan--H{\"o}lder parabolics
    $(\mathbf{P}_{1})_\eta$ and $(\mathbf{P}_{2})_\eta$, by \cite[Exp. XXVI,
    4.1.1]{sga3}. Then Theorem \ref{thm_jhgroupscheme}
    \ref{thm_jhgroupscheme2} implies that there exists an
    isomorphism between $\mathbf{P}_{1}$ and $\mathbf{P}_{2}$, thus the
    Levi-factors $(s_1^*\xi)(\mathsf{L}_1)$ and
    $(s_2^*\xi)(\mathsf{L}_2)$ are isomorphic.
\end{proof}

Theorem \ref{thm_jhprincipal} gives us an equivalent characterization of
Ramanathan-polystability.

\begin{lemma}[(Equivalent characterization of
    Ramanathan-polystability)]\label{lem_polystabilityequivalenceprincipal}
    A Ramanathan-semistable principal $\mathsf{G}$-bundle $\xi$ over
    $X$ is Ramanathan-polystable if and only if:
    \begin{enumerate}[label=(\roman*)]
        \item\label{lem_polystabilityequivalenceprincipal1} $\xi$ is
            isomorphic to
            the extension of a Ramanathan-stable principal
            $\mathsf{L}$-bundle $\xi^{\mathsf{L}}$, where $\mathsf{L}$ is a
            Levi-factor of a parabolic of $\mathsf{G}$.

        \item\label{lem_polystabilityequivalenceprincipal2} For all
            \textit{admissible} characters $\chi:\mathsf{L}\rightarrow
            K^\times$, i.e., characters trivial on the center
            $\mathbf{Z}(\mathsf{L})$, we have $\deg(\chi(\xi^{\mathsf{L}}))=0$.
    \end{enumerate}
\end{lemma}

\begin{proof}
    Using Lemmas \ref{lem_stabilityequivalencegroupscheme}
    \ref{lem_stabilityequivalencegroupscheme2} and
    \ref{lem_stabilityequivalenceprincipal}
    \ref{lem_stabilityequivalenceprincipal2}, $\xi$ is
    Ramanathan-polystable if and only if the complementary polyhedra
    $(d_{\mathrm{Ad}(\xi)_\eta,\mathbf{T}_{\eta}})_{\mathbf{T}_{\eta}}$
    from Lemma
    \ref{lem_complementarypolyhedrongroupscheme} are polystable. By
    \ref{lem_polystabilityequivalencegroupscheme3} from Lemma
    \ref{lem_polystabilityequivalencegroupscheme}, and by the existence of
    Jordan--H{\"o}lder reductions in Theorem \ref{thm_jhprincipal}, the
    equivalence follows.
\end{proof}

\begin{remark}[(Bundle moduli
    problem)]\label{rem_bundlemoduliproblemprincipal}
    \sloppy
    For the moduli stack $\mathcal{M}\textsc{Bun}_{\mathsf{G}}$ of principal
    $\mathsf{G}$-bundles over $X$, it is known that
    $\mathcal{M}\textsc{Bun}_{\mathsf{G}}$ is algebraic and locally
    of finite type, as
    seen in \cite[Theorem
    7.11]{casalainamartin-wise_anintroductiontomodulistacksviewtowardshiggsbundles},
    and $\mathcal{M}\textsc{Bun}_{\mathsf{G}}$ also has affine
    diagonal, due to the proof of
    \cite[Proposition 1]{heinloth_uniformizationofG-bundles}. We claim that
    $(\mathcal{M}\textsc{Bun}_{\mathsf{G}},d_{\mathcal{M}\textsc{Bun}_{\mathsf{G}}})$
    is a bundle moduli
    problem, from
    Definition \ref{def_bundlemoduliproblem}
    \ref{def_bundlemoduliproblem1}, with
    $d_{\mathcal{M}\textsc{Bun}_{\mathsf{G}}}$ induced by
    $d_{\mathbf{G}_\eta,\mathbf{T}_\eta}$ from
    Lemma \ref{lem_complementarypolyhedrongroupscheme} on automorphism group
    schemes $\mathbf{G}:=\mathrm{Ad}(\xi)$ of the principal
    $\mathsf{G}$-bundles $\xi$.
    The properties of Definition \ref{def_bundlemoduliproblem}
    \ref{def_bundlemoduliproblem1}
    \ref{def_bundlemoduliproblem11}-\ref{def_bundlemoduliproblem14}
    arise from Lemmas \ref{lem_stabilityequivalencegroupscheme} and
    \ref{lem_stabilityequivalenceprincipal}. In this case,
    all facets $P$ of $\Phi(\mathrm{Ad}(\xi)_\eta,\mathbf{T}_\eta)$
    are admissible,
    in the sense of Definition \ref{def_bundlemoduliproblem}
    \ref{def_bundlemoduliproblem1} \ref{def_bundlemoduliproblem11}.
    For Definition \ref{def_bundlemoduliproblem}
    \ref{def_bundlemoduliproblem1} \ref{def_bundlemoduliproblem15},
    we follow \cite[\S1.E.a]{heinloth_hilbertmumfordstability} to
    define the line bundle $\mathcal{L}_{\det}$ on
    $\mathcal{M}\textsc{Bun}_{\mathsf{G}}$
    by the following fibers: For every principal $\mathsf{G}$-bundle
    $\xi$ (up to isomorphism), we have the graded vector space of
    cohomology $H^*(X,\mathrm{ad}(\xi))$, which induces the dual-determinant
    $\mathcal{L}_{\det,\xi}:=\det(H^*(X,\mathrm{ad}(\xi)))^*$. As shown in
    \cite[Corollary 1.16]{heinloth_hilbertmumfordstability},
    $\mathcal{L}_{\det}$-stability conditions are equivalent to the usual
    Ramanathan-stability conditions, and by Definitions
    \ref{def_Lstability} \ref{def_Lstability2} and
    \ref{def_ramanathanpolystability}, this equivalence extends to
    polystability. Thus, $\mathcal{L}_{\det}$ has compatible stability in the
    sense of Definition \ref{def_bundlemoduliproblem}
    \ref{def_bundlemoduliproblem2}.
\end{remark}

To study the Jordan--H{\"o}lder stratification and S-equivalence of
$(\mathcal{M}\textsc{Bun}_{\mathsf{G}},d_{\mathcal{M}\textsc{Bun}_{\mathsf{G}}})$,
let $\mathrm{JH}_{\mathsf{G}}$ be the
set of isomorphism classes of Levi-factors of $\mathsf{G}$. For
$\mu\in\mathrm{JH}_{\mathsf{G}}$, we define
$(\mathcal{M}\textsc{Bun}_{\mathsf{G}})^{d_{\mathcal{M}\textsc{Bun}_{\mathsf{G}}}-\mathrm{sst}}_\mu(K)$
as the set of
$K$-points of Ramanathan-semistable principal $\mathsf{G}$-bundles $\xi$
with a Jordan--H{\"o}lder reduction $s:X\rightarrow\xi/\mathsf{P}$, such that
$\mathsf{L}:=\mathsf{P}/\mathbf{R}_u(\mathsf{P})$ is in the
isomorphism class $\mu$. By
construction, it is clear that this induces the same Jordan--H{\"o}lder
stratification as in Subsection \ref{subsec_jhstratifications}.

\begin{theorem}[(Jordan--H{\"o}lder
    stratification)]\label{thm_jhstratificationprincipal}
    Let $K$ be algebraically closed of characteristic $0$. The
    Jordan--H{\"o}lder stratification of
    $(\mathcal{M}\textsc{Bun}_{\mathsf{G}},d_{\mathcal{M}\textsc{Bun}_{\mathsf{G}}})$,
    from Subsection
    \ref{subsec_jhstratifications}, consists of locally closed
    substacks. Moreover,
    $(\mathcal{M}\textsc{Bun}_{\mathsf{G}},d_{\mathcal{M}\textsc{Bun}_{\mathsf{G}}})$
    exhibits a compatible good moduli space
    $\pi:(\mathcal{M}\textsc{Bun}_{\mathsf{G}})^{d_{\mathcal{M}\textsc{Bun}_{\mathsf{G}}}-\mathrm{sst}}\rightarrow
    (M\textsc{Bun}_{\mathsf{G}})^{d_{\mathcal{M}\textsc{Bun}_{\mathsf{G}}}-\mathrm{sst}}$,
    from Definition \ref{def_jhstratification} \ref{def_jhstratification2}.
\end{theorem}

\begin{proof}
    We have a good moduli space
    $\pi:(\mathcal{M}\textsc{Bun}_{\mathsf{G}})^{d_{\mathcal{M}\textsc{Bun}_{\mathsf{G}}}-\mathrm{sst}}\rightarrow
    (M\textsc{Bun}_{\mathsf{G}})^{d_{\mathcal{M}\textsc{Bun}_{\mathsf{G}}}-\mathrm{sst}}$
    from \cite[5.9 Theorem]{ramanathan_moduliforprincipalbundlesII}. For
    $\mu\in\mathrm{JH}_{\mathsf{G}}$,
    $(\mathcal{M}\textsc{Bun}_{\mathsf{G}})^{d_{\mathcal{M}\textsc{Bun}_{\mathsf{G}}}-\mathrm{sst}}_\mu$
    is the
    $\pi$-preimage of a subspace
    $(M\textsc{Bun}_{\mathsf{G}})^{d_{\mathcal{M}\textsc{Bun}_{\mathsf{G}}}-\mathrm{sst}}_\mu$
    of
    $(M\textsc{Bun}_{\mathsf{G}})^{d_{\mathcal{M}\textsc{Bun}_{\mathsf{G}}}-\mathrm{sst}}$
    defined as follows:
    Due to Lemma
    \ref{lem_polystabilityequivalenceprincipal},
    $(M\textsc{Bun}_{\mathsf{G}})^{d_{\mathcal{M}\textsc{Bun}_{\mathsf{G}}}-\mathrm{sst}}_\mu$
    is the image of
    a morphism
    $\iota:(M\textsc{Bun}_{\mathsf{L}})^{d_{M\textsc{Bun}_{\mathsf{L}}}-\mathrm{st}}\rightarrow
    (M\textsc{Bun}_{\mathsf{G}})^{d_{\mathcal{M}\textsc{Bun}_{\mathsf{G}}}-\mathrm{sst}}$
    extending the
    structure group through the canonical inclusion
    $\mathsf{L}\hookrightarrow\mathsf{G}$. Thus,
    $(\mathcal{M}\textsc{Bun}_{\mathsf{G}})^{d_{\mathcal{M}\textsc{Bun}_{\mathsf{G}}}-\mathrm{sst}}_\mu$
    forms a
    substack of
    $(\mathcal{M}\textsc{Bun}_{\mathsf{G}})^{d_{\mathcal{M}\textsc{Bun}_{\mathsf{G}}}-\mathrm{sst}}$.
    To show that
    $(\mathcal{M}\textsc{Bun}_{\mathsf{G}})^{d_{\mathcal{M}\textsc{Bun}_{\mathsf{G}}}-\mathrm{sst}}_\mu$
    is a locally closed substack, we claim that $\iota$ is an
    immersion of algebraic spaces. Since
    $d_{\mathcal{M}\textsc{Bun}_{\mathsf{G}}}$-polystable points in
    $(\mathcal{M}\textsc{Bun}_{\mathsf{G}})^{d_{\mathcal{M}\textsc{Bun}_{\mathsf{G}}}-\mathrm{sst}}_\mu(K)$
    correspond uniquely to $K$-points in
    $M\textsc{Bun}_{\mathsf{L}}^{d_{M\textsc{Bun}_{\mathsf{L}}}-\mathrm{st}}(K)$,
    $\iota$ is a monomorphism. For a $K$-scheme $T$ and
    a family $\xi_T\in M\textsc{Bun}_{\mathsf{G}}(T)$ of principal
    $\mathsf{G}$-bundles, the
    locus $T'$ in $T$, for which $\xi_{T'}/\mathsf{P}\rightarrow T'\times X$
    admits a section, parametrizes the family $\xi_{T'}$ over $T'$
    precisely in the image of $\iota$. It suffices to show that $T'$
    is locally closed in $T$: By trivializing locally inside
    $T'\times X$, sections of $\xi_{T'}/\mathsf{P}\rightarrow
    T'\times X$ map into
    Bruhat cells of $\mathsf{G}/\mathsf{P}$, which are locally closed, hence,
    $T'$ is locally closed in $T$. Thus, $\iota$ is an immersion of
    algebraic spaces locally of finite type, so
    $(\mathcal{M}\textsc{Bun}_{\mathsf{G}})^{d_{\mathcal{M}\textsc{Bun}_{\mathsf{G}}}-\mathrm{sst}}_\mu$
    is a locally
    closed substack of
    $(\mathcal{M}\textsc{Bun}_{\mathsf{G}})^{d_{\mathcal{M}\textsc{Bun}_{\mathsf{G}}}-\mathrm{sst}}$.

    The fact that $\pi$ is a compatible good moduli space comes from
    the correspondence between
    $d_{\mathcal{M}\textsc{Bun}_{\mathsf{G}}}$-polystable points in
    $(\mathcal{M}\textsc{Bun}_{\mathsf{G}})^{d_{\mathcal{M}\textsc{Bun}_{\mathsf{G}}}-\mathrm{sst}}_\mu(K)$
    and $K$-points in
    $M\textsc{Bun}_{\mathsf{L}}^{d_{M\textsc{Bun}_{\mathsf{L}}}-\mathrm{st}}(K)$,
    which implies that quasi-S-equivalence implies S-equivalence.
\end{proof}

To apply Theorem \ref{thm_jhprincipal} in practice for
vector bundles $E$ (via their frame bundles $\mathrm{Fr}(E)$), we must
calculate the complementary polyhedron
$d_{\mathrm{Ad}(\mathrm{Fr}(E))_{\eta},\mathbf{T}_{\eta}}$ to
determine a Jordan--H{\"o}lder
parabolic $\mathbf{P}$. For this, we recall how vector bundles split
into subbundles.

\begin{remark}[(Serre vanishing)]\label{rem_serrevanishing}
    For a short exact sequence of vector bundles over $X$:
    \begin{equation*}
        0\rightarrow F\rightarrow E\rightarrow F'\rightarrow 0
    \end{equation*}
    let $\beta\in \mathrm{Ext}^1(F',F)$ be the \textit{extension
    class}. For an \textit{effective divisor} $D\neq0$ on $X$,
    \textit{Serre's vanishing theorem} implies there exists $N\in\mathbb{N}$,
    so that for all $n\geqslant N$, we have:
    \begin{equation*}
        0=H^1(X,(F'^*\otimes F)(nD))=H^1(X,(F'(-nD))^*\otimes
        F)=\mathrm{Ext}^1(F'(-nD),F)
    \end{equation*}
    so every short exact sequence with $F$ in front and $F'(-nD)$ at
    the end splits. Using
    \cite[{\href{https://stacks.math.columbia.edu/tag/010I}{010I}}]{stacks-project},
    through the inclusion of coherent sheaves $\iota:F'^*\otimes
    F\hookrightarrow (F'^*\otimes F)(nD)$, we induce the pullback
    extension $\iota^*\beta\in\mathrm{Ext}^1(F'(-nD),F)=0$ which splits,
    implying that the coherent sheaf $E\times_{F'}F'(-nD)$ is
    isomorphic to $F\oplus F'(-nD)$. Thus, $F'(-nD)$ naturally embeds
    as a subsheaf of $E$, such that its \textit{saturation}
    $F'(-nD)^{\mathrm{sat}}$ in $E$ is a subbundle such that
    $E_\eta\cong F_\eta\oplus (F'(-nD)^{\mathrm{sat}})_\eta$ splits
    over $\eta$. The saturation $F'(-nD)^{\mathrm{sat}}$ is unique, and
    does not depend on the choice of $n\geqslant N$. Furthermore, $N=0$
    occurs if and only if the original sequence splits, i.e., $\beta=0$.
\end{remark}

\begin{example}[(Jordan--H{\"o}lder filtrations of a rank 3
    bundle)]\label{ex_jhfiltrationsvectorbundle}
    By modifying the example of Huybrechts--Lehn in \cite[Example
    1.2.10]{huybrechts-lehn_geometryofmodulispaces},
    we wish to construct a rank $3$ vector bundle $E$ over $X$, with
    two different Jordan--H{\"o}lder filtrations $0\subset F\subset E$, and
    $0\subset L\subset E$, where $L$ is a line subbundle, and $F$ is a
    rank $2$ slope-stable subbundle.
    We construct $F$ as follows: Let $L_i$ denote a line bundle over
    $X$ of degree $i$, then there exists a non-split extension:
    \begin{equation*}
        0\rightarrow L_0\rightarrow F\rightarrow L_2\rightarrow0
    \end{equation*}
    Following Remark \ref{rem_serrevanishing}, for any effective
    divisor $D\neq0$ on $X$, there exists $N\in\mathbb{N}$, such that for all
    $n\geqslant N$, we have $\mathrm{Ext}^1(L_2(-nD),L_0)=0$. Then let
    $L':=L_2(-nD)^{\mathrm{sat}}$ be the saturation inside $F$, with
    degree $l':=\deg(L')$, then $F_\eta\cong (L_0)_\eta\oplus L'_\eta$.
    We follow that $l'\leqslant\deg(L_2)=2$, as $L_2$ is slope-stable and
    $L'$ embeds into $L_2$ as a subsheaf through $F\rightarrow L_2$.
    Since $F_\eta\cong (L_0)_\eta\oplus L'_\eta$, $F$ is slope-stable
    if and only if $l'\leqslant0$, since $\mu(F)=1/2$, so we assume
    that $l'\leqslant0$.
    We choose $L:=L_1$ to have degree $1$ and define $E:=F\oplus L$.
    By construction, $E$ splits such that $E_{\eta}\cong F_{\eta}\oplus
    L_{\eta}\cong (L_0)_{\eta}\oplus(L')_{\eta}\oplus  L_{\eta}$.

    We follow Wi{\ss}dorf's method in \cite[3.3.11
    Remark]{wissdorf_phdthesis} to
    calculate the complementary polyhedron of
    $\mathrm{Ad}(E):=\mathrm{Ad}(\mathrm{Fr}(E))$
    with respect to the maximal torus $\mathbf{T}_{\eta}$ induced from the
    splitting of $E_{\eta}$ above. Just as in \cite[3.3.12 Example
    and 3.A]{wissdorf_phdthesis}, we associate to Borel subgroup
    schemes $\mathbf{B}$ of $\mathrm{Ad}(E)$, with
    $\mathbf{B}_{\eta}$ containing
    $\mathbf{T}_{\eta}$, vectors in $\mathbb{R}^3$ that record the
    degrees of the
    quotient bundles in the induced vector bundle filtrations of $E$.
    We adopt the notation of Example \ref{ex_complementarypolyhedra},
    and fix a choice of simple roots
    $\Delta:=\{\alpha,\beta\}\subset\Phi(\mathrm{Ad}(E)_{\eta},\mathbf{T}_{\eta})$,
    inducing a Borel subgroup scheme
    $\mathbf{B}:=\underline{R}^{-1}(\mathfrak{c}_{\alpha+\beta})$ of
    $\mathrm{Ad}(E)$.
    Wi{\ss}dorf requires that the degrees of the filtration corresponding
    to $\mathbf{B}'$ appear permuted, as vectors in $\mathbb{R}^3$, by the
    permutation in $S_3$ corresponding to the element of the Weyl group
    mapping $\mathbf{B}'$ to $\mathbf{B}$. We assign to the simple roots:
    \begin{equation*}
        \alpha,\beta:\mathrm{Ad}(\mathrm{Fr}(L_0))_\eta\times\mathrm{Ad}(\mathrm{Fr}(L'))_{\eta}\times\mathrm{Ad}(\mathrm{Fr}(L))_{\eta}\rightarrow
        \mathbb{G}_{m,\eta}
    \end{equation*}
    the vectors $(1,-1,0)$ and $(0,1,-1)$, and by assigning the
    filtration $0\subset L_0\subset F\subset E$ to $\mathbf{B}$, we have:
    \begin{equation*}
        \begin{tabular}{|c|c|c|c|}
            \hline
            \textit{Borel subgroup}                      &
            \textit{Filtration of $E$}                 &
            \textit{Permutation in $S_3$} & \textit{Vector in $\mathbb{R}^3$} \\
            \hline
            $\underline{R}^{-1}(\mathfrak{c}_{\alpha+\beta})$  & $0\subset
            L_0\subset F\subset E$           & $\mathrm{id}$
            & $(0,2,1)$                   \\
            \hline
            $\underline{R}^{-1}(\mathfrak{c}_{\beta})$         & $0\subset
            L'\subset F\subset E$            & (12)
            & $(2-l',l',1)$               \\
            \hline
            $\underline{R}^{-1}(\mathfrak{c}_{\alpha})$        & $0\subset
            L_0\subset L_0\oplus L\subset E$ & (23)
            & $(0,2,1)$                   \\
            \hline
            $\underline{R}^{-1}(\mathfrak{c}_{-\alpha-\beta})$ & $0\subset
            L\subset L'\oplus L\subset E$    & (13)
            & $(2-l',l',1)$               \\
            \hline
            $\underline{R}^{-1}(\mathfrak{c}_{-\beta})$        & $0\subset
            L\subset L_0\oplus L\subset E$   & (123)
            & $(0,2,1)$                   \\
            \hline
            $\underline{R}^{-1}(\mathfrak{c}_{-\alpha})$       & $0\subset
            L'\subset L'\oplus L\subset E$   & (132)
            & $(2-l',l',1)$               \\
            \cline{1-4}
            \multicolumn{4}{c}{%
                \raisebox{-0.4em}{\footnotesize\textbf{Table~5.1.8.}}%
            }\\
        \end{tabular}
    \end{equation*}
    Following \cite[3.3.11 Remark]{wissdorf_phdthesis}, we project
    these vectors in $\mathbb{R}^3$ down to the hyperplane
    $H:=\{(x,y,z)\in\mathbb{R}^3\ |\ x+y+z=0\}$ which maps
    $(0,2,1)\mapsto(-1,1,0)$ and $(2-l',l',1)\mapsto(1-l',l'-1,0)$.
    Through the isomorphism
    $V_{\mathrm{Ad}(E)_{\eta},\mathbf{T}_{\eta}}^*\cong H$
    given by $\widecheck{\alpha}\mapsto(1,-1,0)$ and
    $\widecheck{\beta}\mapsto(0,1,-1)$, the convex hull of
    $d_{\mathrm{Ad}(E)_{\eta},\mathbf{T}_{\eta}}$ contains $0$ for
    every possible
    $l'\leqslant0$, confirming that $E$ is slope-semistable by Lemmas
    \ref{lem_stabilityofcomplementarypolyhedra}
    \ref{lem_stabilityofcomplementarypolyhedra1} and
    \ref{lem_stabilityequivalenceprincipal}
    \ref{lem_stabilityequivalenceprincipal1}. For $l'=0$,
    $d_{\mathrm{Ad}(E)_{\eta},\mathbf{T}_{\eta}}$ is given by the complementary
    polyhedron $d_b$ from Example \ref{ex_complementarypolyhedra}. It
    is clear that the Jordan--H{\"o}lder filtrations $0\subset F\subset E$
    and $0\subset L\subset E$ of $E$, with graded bundle $L\oplus
    F$, correspond to the Jordan--H{\"o}lder facets $P_1$ and $P_2$ from
    Example \ref{ex_jhfacetsA2} \ref{ex_jhfacetsA21}.
\end{example}

\subsection{Parahoric torsors}\label{subsec_parahoric}

From the introduction in \ref{intro_applications} \ref{intro_applications5},
we apply our Jordan--H{\"o}lder theory to \textit{parahoric torsors}.
Studied in
\cite{heinloth_uniformizationofG-bundles,pappas-rapoport_someqnsaboutG-bundles,balaji-sesh_parahorictorsors},
they unite the constructions of \textit{principal parabolic
bundles} and \textit{Hecke transforms} very naturally. These are
torsors of \textit{parahoric group
schemes} $\mathcal{G}$ over $X$ (Definition \ref{def_parahoricgroup}),
which are generically reductive group schemes, but over finitely many
\textit{ramification} or \textit{bad points} $\mathcal{R}\subset X$, they
look like the \textit{parahoric subgroups} defined by Bruhat--Tits in
\cite{bruhat-tits_groupesreductifsI,bruhat-tits_groupesreductifsII}.

Let $\mathsf{G}$ be a split reductive $K$-group.
With the setup of Balaji--Seshadri in \cite[Theorem
5.3.1]{balaji-sesh_parahorictorsors}, imposing that the genus $g$ of
$X$ is greater than $1$, and that $\mathcal{G}$ is generically split (with
respect to $\mathsf{G}$), semisimple, and simply connected, we have that
parahoric $\mathcal{G}$-torsors are naturally isomorphic to
\textit{$(\Gamma,\mathsf{G})$-bundles}, as defined in \cite[2.2.3.
Definition]{balaji-sesh_parahorictorsors}. Thus, parahoric torsors
are the right notion to study how Galois covers $Y\rightarrow X$,
with ramification at $\mathcal{R}$, impose conjugacy classes on
$\mathsf{G}$-representations $\rho$ of $\pi_1(X)$, namely conjugacy classes
$(\rho_{x^i})_{x^i\in\mathcal{R}}$ at ramification points, as seen in
\cite[1.0.1. Definition]{balaji-sesh_parahorictorsors}.

\subsubsection{Definition and stability of parahoric
torsors}\label{subsubsec_stabilityparahoric}

We first look at the local situation around a $K$-point $x\in X$,
with a coordinate $t$ that vanishes at $x$. For a smooth affine
$K$-group $\mathsf{P}$, the \textit{loop group} is
$L\mathsf{P}:=\mathsf{P}(K\llparenthesis t\rrparenthesis)$, and the
\textit{positive loop group} is
$L^+\mathsf{P}:=\mathsf{P}(K\llbracket t\rrbracket)$.

\begin{definition}[(Parahoric group schemes)]\label{def_parahoricgroup}
    \leavevmode
    \begin{enumerate}[label=(\alph*)]
        \item\label{def_parahoricgroup1} Let $\mathsf{T}$ be a maximal
            torus of $\mathsf{G}$, and for all roots
            $\alpha\in\Phi(\mathsf{G},\mathsf{T})$, let $\mathsf{U}_\alpha$
            denote the corresponding \textit{root group}. Let $\theta\in
            V_{\mathsf{G},\mathsf{T}}:=\mathcal{X}^*(\mathsf{T})\otimes\mathbb{R}$,
            from \eqref{eq_charactergroup}, and for all
            $\alpha\in\Phi(\mathsf{G},\mathsf{T})$, let
            $\theta_\alpha:=\widecheck{\alpha}(\theta)\in\mathbb{R}$,
            with $\lfloor\theta_{\alpha}\rfloor$ being its \textit{floor}. The
            \textit{parahoric subgroup} $\mathcal{P}_\theta$ is the subgroup of
            $L\mathsf{G}$ generated as follows:
            \begin{equation*}
                \mathcal{P}_\theta:=\left\langle
                L^+\mathsf{T},\mathsf{U}_\alpha\left(t^{-\lfloor\theta_\alpha\rfloor}K\llbracket
                t\rrbracket\right)\ \middle\vert\ \alpha\in\Phi(\mathsf{G},\mathsf{T})\right\rangle
            \end{equation*}

        \item\label{def_parahoricgroup2} A subgroup $\mathcal{P}$ of
            $L\mathsf{G}$, which is conjugate by an element in
            $L\mathsf{G}$ to $\mathcal{P}_\theta$, for some $\theta$, from
            \ref{def_parahoricgroup1}, is called a \textit{parahoric
            subgroup} of $L\mathsf{G}$.

        \item\label{def_parahoricgroup3} A smooth affine group scheme
            $\mathcal{G}$ over $\mathrm{Spec}(K\llbracket
            t\rrbracket)$ is called a
            \textit{parahoric group scheme} if:
            \begin{enumerate}[label=(\roman*)]
                \item\label{def_parahoricgroup31}
                    $\mathcal{G}(\mathrm{Spec}(K\llbracket
                    t\rrbracket))$ is isomorphic to a parahoric subgroup
                    $\mathcal{P}$ of $L\mathsf{G}$, from
                    \ref{def_parahoricgroup2}.

                \item\label{def_parahoricgroup32}
                    $\mathcal{G}_{\mathrm{Spec}(K\llparenthesis
                    t\rrparenthesis)}$ is isomorphic to
                    $\mathrm{Spec}(K\llparenthesis
                    t\rrparenthesis)\times\mathsf{G}$.
            \end{enumerate}

        \item\label{def_parahoricgroup4} Let $\mathcal{R}$ be a finite
            collection of $K$-points in $X$. A smooth affine group scheme
            $\mathcal{G}$ over $X$ is called a \textit{parahoric
            group scheme} if:
            \begin{enumerate}[label=(\roman*)]
                \item\label{def_parahoricgroup41}
                    $\mathcal{G}_{X\setminus\mathcal{R}}$ is a
                    reductive group scheme over $X\setminus\mathcal{R}$.

                \item\label{def_parahoricgroup42} For all
                    $x^i\in\mathcal{R}$, and
                    a local coordinate $t$ of $X$ vanishing at $x^i$,
                    $\mathcal{G}_{\mathrm{Spec}(K\llbracket t\rrbracket)}$ is a
                    \textit{parahoric group scheme} over
                    $\mathrm{Spec}(K\llbracket
                    t\rrbracket)$, from \ref{def_parahoricgroup3}.
            \end{enumerate}
    \end{enumerate}
\end{definition}

Definition \ref{def_parahoricgroup} \ref{def_parahoricgroup3} is the
\textit{schematization} of the parahoric subgroup in
\ref{def_parahoricgroup2}, which exists due to
\cite[\S4.6]{bruhat-tits_groupesreductifsII}. Gluing schematizations
together in the sense of \ref{def_parahoricgroup4} is possible due to
\cite[5.2.2 Lemma]{balaji-sesh_parahorictorsors}. We now define
\textit{parahoric torsors}. For a finite collection of $K$-points
$\mathcal{R}$ in $X$, let $\mathcal{G}$ be a parahoric group scheme over $X$.

\begin{definition}[(Parahoric torsors)]\label{def_parahorictorsor}
    A \textit{parahoric $\mathcal{G}$-torsor} $\xi$ over $X$ is a flat morphism
    $\xi\rightarrow X$, locally of finite presentation, with a
    $\mathcal{G}$-action $\sigma:\xi\times\mathcal{G}\rightarrow\xi$ such that:
    \begin{equation*}
        (\mathrm{pr}_1,\sigma):\xi\times\mathcal{G}\rightarrow\xi\times\xi
    \end{equation*}
    is an isomorphism, where $\mathrm{pr}_1$ is the projection onto the first
    factor. We denote the moduli stack of parahoric $\mathcal{G}$-torsors by
    $\mathcal{M}\textsc{Bun}_{\mathcal{G}}$, a stack over $\mathrm{Spec}(K)$.
\end{definition}

We wish to understand reductions of parahoric torsors to parabolics.
For a parabolic subgroup $\mathsf{P}$ of $\mathsf{G}$, the preimage
of $\mathcal{P}_P:=\mathrm{ev}_0^{-1}(\mathsf{P})$ under
$\mathrm{ev}_0:L^+\mathsf{G}\rightarrow\mathsf{G}$, sending $t$ to $0$, is a
parahoric subgroup of $L\mathsf{G}$. We now define \textit{parabolic
subgroups} of parahoric groups, using that after fixing simple roots
$\Delta\subset\Phi(\mathsf{G},\mathsf{T})$, every parahoric
subgroup $\mathcal{P}$ of $L\mathsf{G}$ is conjugate to $\mathcal{P}_\theta$,
where $\theta$ lies in the closure of the \textit{standard Weyl
alcove}, i.e., $\theta_\alpha\geqslant0$ for all $\alpha\in\Delta$, and
$\theta_\beta\leqslant1$ for \textit{highest roots} $\beta$ with respect
to $\Delta$.

\begin{definition}[(Parabolic subgroups)]\label{def_parabolicsubgroupsparahoric}
    Let $\Delta\subset\Phi(\mathsf{G},\mathsf{T})$ be simple roots,
    and let $\mathsf{P}$ be a parabolic subgroup of $\mathsf{G}$
    containing $\mathsf{T}$, such that $\Delta$ is in the parabolic
    subset $R(\mathsf{P})$ from Lemma
    \ref{lem_rootsystemgroupscheme}
    \ref{lem_rootsystemgroupscheme1}.
    \begin{enumerate}[label=(\alph*)]
        \item\label{def_parabolicsubgroupsparahoric1} In the setup of
            Definition \ref{def_parahoricgroup} \ref{def_parahoricgroup1},
            let $\theta$ lie in the standard Weyl alcove. The
            \textit{parabolic subgroup} of $\mathcal{P}_\theta$ corresponding to
            $\mathsf{P}$ is
            $\mathcal{P}_{\mathsf{P},\theta}:=\mathrm{ev}_0^{-1}(\mathsf{P})\cap\mathcal{P}_\theta$.

        \item\label{def_parabolicsubgroupsparahoric2} Let $\mathcal{P}$ be a
            parahoric subgroup of $L\mathsf{G}$ conjugate to
            $\mathcal{P}_\theta$, where $\theta$ lies in the standard Weyl
            alcove, as in \ref{def_parabolicsubgroupsparahoric1}. The
            \textit{parabolic subgroup} of $\mathcal{P}$ corresponding to
            $\mathsf{P}$ is $\mathcal{P}_{\mathsf{P}}$, which is conjugate to
            $\mathcal{P}_{\mathsf{P},\theta}$, from
            \ref{def_parabolicsubgroupsparahoric1}.

        \item\label{def_parabolicsubgroupsparahoric3} For a parahoric
            group scheme $\mathcal{G}$ over
            $\mathrm{Spec}(K\llbracket t\rrbracket)$, the
            \textit{parabolic subgroup} $\mathcal{G}_{\mathsf{P}}$ of
            $\mathcal{G}$
            corresponding to $\mathsf{P}$ is the subgroup scheme of
            $\mathcal{G}$ such that:
            \begin{enumerate}[label=(\roman*)]
                \item
                    $\mathcal{G}_{\mathsf{P}}(\mathrm{Spec}(K\llbracket
                    t\rrbracket))$ is
                    isomorphic to the parabolic subgroup
                    $\mathcal{P}_{\mathsf{P}}$
                    of $\mathcal{G}(\mathrm{Spec}(K\llbracket
                    t\rrbracket))$, from
                    \ref{def_parabolicsubgroupsparahoric2}.

                \item $(\mathcal{G}_{\mathsf{P}})_{\mathrm{Spec}(K\llparenthesis
                    t\rrparenthesis)}$ is isomorphic to
                    $\mathrm{Spec}(K\llparenthesis
                    t\rrparenthesis)\times\mathsf{P}$.
            \end{enumerate}

        \item\label{def_parabolicsubgroupsparahoric4} Let $\mathcal{R}$ be a
            finite collection of $K$-points in $X$, and let $\mathcal{G}$ be a
            parahoric group scheme over $X$. The \textit{parabolic
            subgroup} $\mathcal{G}_{\mathsf{P}}$ of $\mathcal{G}$
            corresponding to
            $\mathsf{P}$ is the subgroup scheme of $\mathcal{G}$ such that:
            \begin{enumerate}[label=(\roman*)]
                \item
                    $(\mathcal{G}_{\mathsf{P}})_{X\setminus\mathcal{R}}$
                    is a parabolic subgroup
                    scheme of $\mathcal{G}_{X\setminus\mathcal{R}}$.
                \item For all $x^i\in\mathcal{R}$, and a local coordinate $t$ of
                    $X$ vanishing at $x^i$,
                    $(\mathcal{G}_{\mathsf{P}})_{\mathrm{Spec}(K\llbracket
                    t\rrbracket)}$ is the
                    parabolic subgroup of
                    $\mathcal{G}_{\mathrm{Spec}(K\llbracket
                    t\rrbracket)}$, from
                    \ref{def_parabolicsubgroupsparahoric3}.
            \end{enumerate}
    \end{enumerate}
\end{definition}
Due to \cite[5.2.2 Lemma]{balaji-sesh_parahorictorsors}, the
schemes in Definitions \ref{def_parabolicsubgroupsparahoric}
\ref{def_parabolicsubgroupsparahoric3} and
\ref{def_parabolicsubgroupsparahoric4} exist, and for a parahoric
$\mathcal{G}$-torsor $\xi$, a \textit{reduction to a parabolic} is a section
$s:X\rightarrow\xi/\mathcal{G}_{\mathsf{P}}$ for a parabolic subgroup
$\mathsf{P}$ of $\mathsf{G}$.

$\mathrm{Ad}(\xi)$ is a rationally split reductive group scheme due to
\cite[Lemmas 1.13 and 1.14]{heinloth_hilbertmumfordstability}.
Thus, since the scheme $\mathrm{Par}(\mathrm{Ad}(\xi))$ over $X$, whose global
sections are parabolic subgroup schemes of $\mathrm{Ad}(\xi)$, is projective,
due to \cite[Exp. XXVI, 4.1.1]{sga3}, Heinloth uses that
$\mathcal{G}_\eta:=\pi^*\mathcal{G}$ is a reductive group scheme over
$\eta:=\mathrm{Spec}(K(X))$ to prove the following.

\begin{lemma}[(Rees construction {\cite[Lemmas 1.14 and
    3.8]{heinloth_hilbertmumfordstability}})]\label{lem_reesconstructionparahoric}
    Let $\pi:\eta\rightarrow X$ be the generic point, and let
    $\pi_Y:Y\rightarrow \eta$ be a finite {\'e}tale cover over which
    $\mathcal{G}_Y:=\pi_Y^*\mathcal{G}_\eta$ splits with maximal
    torus $\mathcal{T}_Y$, which
    exists due to \cite[Lemma 5.1.3]{conrad_reductivegroupschemes}.
    Let $\mathsf{T}$ be a maximal torus of $\mathsf{G}$ induced by
    $\mathcal{T}_Y$, then we have the following:
    \begin{enumerate}[label=(\alph*)]
        \item\label{lem_reesconstructionparahoric1} Parabolic subgroups
            $\mathsf{P}$ of $\mathsf{G}$, containing $\mathsf{T}$,
            correspond to rational cocharacters
            $\boldsymbol{\lambda}_\eta:\mathbb{G}_{m,\eta}\rightarrow\mathcal{G}_\eta$,
            whose pullback to $Y$ maps into $\mathcal{T}_Y$, such
            that for any
            $K$-scheme $T$, we have:
            \begin{equation*}
                \mathcal{P}(\boldsymbol{\lambda}_\eta)_{\eta}(T):=\{\mathbf{g}\in\mathcal{G}_\eta(T)\ \vert\ \lim_{a\rightarrow0}\boldsymbol{\lambda}_\eta(a)\mathbf{g}\boldsymbol{\lambda}_\eta(a)^{-1}\text{
                exists}\}=(\mathcal{G}_{\mathsf{P}})_\eta(T)
            \end{equation*}

        \item\label{lem_reesconstructionparahoric2} For a parahoric
            $\mathcal{G}$-torsor $\xi$, the following are in bijection:
            \begin{enumerate}[label=(\roman*)]
                \item\label{lem_reesconstructionparahoric21} Filtrations
                    $f:\Theta_K\rightarrow\mathcal{M}\textsc{Bun}_{\mathcal{G}}$
                    of $\xi$.
                \item\label{lem_reesconstructionparahoric22} Reductions
                    $s:X\rightarrow\xi/\mathcal{G}_{\mathsf{P}}$ of
                    $\xi$ to parabolic
                    subgroups $\mathcal{G}_{\mathsf{P}}$ of $\mathcal{G}$.
            \end{enumerate}
    \end{enumerate}
\end{lemma}

\begin{proof}
    For \ref{lem_reesconstructionparahoric1},
    due to the splitting by $\pi_Y$, there exists a canonical
    isomorphism between the root systems
    $\Phi(\mathcal{G}_Y,\mathcal{T}_Y)\cong\Phi(\mathsf{G},\mathsf{T})$.
    Then, for any
    parabolic subgroup
    $\mathsf{P}$ of $\mathsf{G}$ containing $\mathsf{T}$, the parabolic subgroup
    $\mathcal{G}_\mathsf{P}$ of $\mathcal{G}$ is precisely the
    parabolic subgroup of $\mathcal{G}$
    whose pullback $\pi_Y^*\pi^*\mathcal{G}_\mathsf{P}$ is a parabolic
    subgroup scheme of $\mathcal{G}_Y$, containing $\mathcal{T}_Y$,
    whose type is
    isomorphic to $t(\mathsf{P})$ through
    $\Phi(\mathcal{G}_Y,\mathcal{T}_Y)\cong\Phi(\mathsf{G},\mathsf{T})$.
    By applying
    \cite[Proposition 5.2.3]{conrad_reductivegroupschemes} and the
    fact that $\mathcal{G}_\eta$ is a reductive group scheme, the claim follows.

    For \ref{lem_reesconstructionparahoric2}, \cite[Lemmas 1.14 and
    3.8]  {heinloth_hilbertmumfordstability} give the reverse
    direction. To prove the forward direction, a filtration
    $f:\Theta_K\rightarrow\mathcal{M}\textsc{Bun}_{\mathcal{G}}$
    induces a cocharacter
    $\boldsymbol{\lambda}:\mathbb{G}_{m,X}\rightarrow\mathrm{Ad}(\xi)$ on the
    automorphism groups, then \cite[Proposition
    5.2.3]{conrad_reductivegroupschemes} and \cite[Lemma
    1.14]{heinloth_hilbertmumfordstability} imply the claim.
\end{proof}

To define stability conditions for parahoric torsors, Heinloth in
\cite[\S3.B]{heinloth_hilbertmumfordstability} uses a collection
$\underline{\chi}:=(\chi^i:\mathcal{G}_{x^i}\rightarrow\mathbb{G}_{m,x^i})_{x^i\in\mathcal{R}}$
of characters, called \textit{stability parameters}, to define a line
bundle $\mathcal{L}_{\det,\underline{\chi}}$ over
$\mathcal{M}\textsc{Bun}_{\mathcal{G}}$ such that
$\mathcal{L}_{\det,\underline{\chi}}$-stability conditions recover the
usual parabolic stability conditions from Mehta--Seshadri in \cite[Definition
1.13]{mehta-sesh_parabolic}. This uses Definition
\ref{def_Lstability}, which is applicable due to the following remark.

\begin{remark}[(Geometry of $\mathcal{M}\textsc{Bun}_{\mathcal{G}}$
            {\cite[Proposition
    1]{heinloth_uniformizationofG-bundles}})]\label{rem_geometryparahoric}
    The stack $\mathcal{M}\textsc{Bun}_{\mathcal{G}}$ is smooth,
    algebraic, and locally of finite
    type, due to \cite[Proposition
    1]{heinloth_uniformizationofG-bundles}. In Heinloth's proof, he
    proves that for
    $\mathcal{P},\mathcal{Q}\in\mathcal{M}\textsc{Bun}_{\mathcal{G}}(T)$,
    where $T$ is a
    $K$-scheme, the sheaf
    $\mathrm{Iso}_{\mathcal{G}}(\mathcal{P},\mathcal{Q})$ whose
    sections are $\mathcal{G}$-equivariant isomorphisms $\mathcal{P}$
    to $\mathcal{Q}$ is
    representable by an affine scheme
    $(\mathcal{P}\times\mathcal{Q})/\mathcal{G}$ over
    $X\times T$. Because $X$ is projective, the restriction of
    $(\mathcal{P}\times\mathcal{Q})/\mathcal{G}$ to $T$ is also
    affine, so $\mathcal{M}\textsc{Bun}_{\mathcal{G}}$ has
    affine diagonal.
\end{remark}

For a parahoric $\mathcal{G}$-torsor $\xi$, let
$s:X\rightarrow\xi/\mathcal{G}_{\mathsf{P}}$ be a reduction to a parabolic, and
let $\mathsf{T}$ be a maximal torus inside
$\mathsf{P}$. We first define the
\textit{$\underline{\chi}$-numerical invariant}
$n_{\underline{\chi}}(\mathbf{P},\mathfrak{v})$ for
$\mathbf{P}:=\mathrm{Ad}(s^*\xi)$ with a
vertex $\mathfrak{v}\in\pi_0(t(\mathbf{P}_\eta))$: By construction,
$\mathbf{P}_{\eta}$
contains the maximal torus $\mathbf{T}_{\eta}$ induced by $\mathsf{T}$, so
for any vertex $\mathfrak{v}\in\pi_0(t(\mathbf{P}_{\eta}))$, Lemma
\ref{lem_reesconstructionparahoric} induces a corresponding vertex
$\mathsf{v}\in\pi_0(t((\mathcal{G}_{\mathsf{P}})_{\eta}))$. For all
$x^i\in\mathcal{R}$, this uniquely induces a root
$\chi_{\mathsf{v}}^i:\mathcal{G}_{x^i}\rightarrow\mathbb{G}_{m,x^i}$ in
$\Phi(\mathcal{G}_{x^i},\mathcal{T}_{x^i})$, so we define for all characters
$\chi^i:\mathcal{G}_{x^i}\rightarrow\mathbb{G}_{m,x^i}$:
\begin{equation*}
    \chi^i(\mathfrak{v}):=\chi^i(\chi^i_{\mathsf{v}})
\end{equation*}
with the standard dual-pairing that generalizes the pairing in
\cite[Remark
4.1.2]{heinloth-schmitt_cohomologyringsofmodulistacksprincbundles}.
This pairing is independent of the choice of $\mathsf{T}$.
Using the elementary vector bundle $W(\mathbf{P},\mathfrak{v})$ as discussed in
Subsection \ref{subsec_jhgroupscheme}, we
define the \textit{$\underline{\chi}$-numerical invariant}:
\begin{equation*}\label{eq_numericalinvariant}
    \tag{5.2}
    n_{\underline{\chi}}(\mathbf{P},\mathfrak{v}):=\deg(W(\mathbf{P},\mathfrak{v}))+\sum_{x^i\in\mathcal{R}}\chi^i(\mathfrak{v})
\end{equation*}
Using this, Heinloth constructs complementary polyhedra for parahoric
torsors with $\mathcal{L}_{\det,\underline{\chi}}$-stability conditions.
\begin{lemma}[(Complementary polyhedra of parahoric torsors
            {\cite[Lemma
    3.15]{heinloth_hilbertmumfordstability}})]\label{lem_complementarypolyhedronparahoric}
    Let $\underline{\chi}$ be an admissible stability parameter in the
    sense of \cite[\S3.F]{heinloth_hilbertmumfordstability}, and let
    $\xi$ be a parahoric $\mathcal{G}$-torsor. Let
    $\mathbf{T}_{\eta}$ be a maximal
    torus of $\mathrm{Ad}(\xi)_{\eta}$, then the following is a complementary
    polyhedron:
    \begin{align*}
        d^{\underline{\chi}}_{\mathrm{Ad}(\xi)_{\eta},\mathbf{T}_{\eta}}:\mathfrak{W}_{\Phi(\mathrm{Ad}(\xi)_{\eta},\mathbf{T}_{\eta})}
        & \rightarrow V^*_{\mathrm{Ad}(\xi)_{\eta},\mathbf{T}_{\eta}}
        \\
        \mathfrak{c}
        & \mapsto
        \sum_{\mathfrak{v}\in\pi_0(t(\underline{R}^{-1}(\mathfrak{c})))}n_{\underline{\chi}}(\underline{R}^{-1}(\mathfrak{c}),\mathfrak{v}
        )\widecheck{\lambda}_\mathfrak{v}
    \end{align*}
\end{lemma}

Similar to Lemma \ref{lem_stabilityequivalenceprincipal}
for principal bundles,
$d^{\underline{\chi}}_{\mathrm{Ad}(\xi)_{\eta},\mathbf{T}_{\eta}}$ detects
$\mathcal{L}_{\det,\underline{\chi}}$-stability conditions.

\begin{lemma}[(Stability conditions for parahoric torsors {\cite[Lemma
    3.15]{heinloth_hilbertmumfordstability}})]\label{lem_stabilityequivalenceparahoric}
    Let $\underline{\chi}$ be an admissible stability parameter, and
    let $\xi$ be a parahoric $\mathcal{G}$-torsor.
    \begin{enumerate}[label=(\roman*)]
        \item\label{lem_stabilityequivalenceparahoric1} $\xi$ is
            $\mathcal{L}_{\det,\underline{\chi}}$-(semi)-stable if and only if
            $\xi$ is
            $d^{\underline{\chi}}_{\mathrm{Ad}(\xi)_{\eta},\mathbf{T}_{\eta}}$-(semi)-stable
            for all maximal tori $\mathbf{T}_{\eta}$.

        \item\label{lem_stabilityequivalenceparahoric2} $\xi$ is
            $\mathcal{L}_{\det,\underline{\chi}}$-polystable if and only if
            $\xi$ is
            $d^{\underline{\chi}}_{\mathrm{Ad}(\xi)_{\eta},\mathbf{T}_{\eta}}$-polystable
            for all maximal tori $\mathbf{T}_{\eta}$.
    \end{enumerate}
\end{lemma}

\begin{proof}
    The equivalences follow directly by the definition of
    $\underline{\chi}$-numerical invariants, since by construction,
    degrees induced by
    $d^{\underline{\chi}}_{\mathrm{Ad}(\xi)_{\eta},\mathbf{T}_{\eta}}$
    have the same
    sign as the weights
    $\mathrm{wt}(f^*\mathcal{L}_{\det,\underline{\chi}})$ tested
    in Definition \ref{def_Lstability}.
\end{proof}

\subsubsection{Jordan--H{\"o}lder theory for parahoric
torsors}\label{subsubsec_jhparahoric}

We can now state the theorem we aim to prove, constructing
Jordan--H{\"o}lder reductions for parahoric torsors. First, we recall the
following: For a filtration
$f:\Theta_K\rightarrow\mathcal{M}\textsc{Bun}_{\mathcal{G}}$ of
$\xi$, and the induced rational cocharacter
$\boldsymbol{\lambda}_\eta:\mathbb{G}_{m,\eta}\rightarrow\mathcal{G}_\eta$ from
Lemma \ref{lem_reesconstructionparahoric}
\ref{lem_reesconstructionparahoric1}, we have the parahoric group
scheme $\mathcal{L}(\boldsymbol{\lambda}_\eta)$ over $X$ given by the
following for all $K$-schemes $T$:
\begin{equation*}
    \mathcal{L}(\boldsymbol{\lambda}_\eta)_\eta(T)=\{\mathbf{g}\in\mathcal{G}_\eta(T)\ \vert\ \lim_{a\rightarrow0}\boldsymbol{\lambda}_\eta(a)\mathbf{g}\boldsymbol{\lambda}_{\eta}(a)^{-1}=\mathbf{g}\}
\end{equation*}
Let $s:X\rightarrow\xi/\mathcal{G}_{\mathsf{P}}$ be the reduction
corresponding to the filtration $f$, from Lemma
\ref{lem_reesconstructionparahoric}
\ref{lem_reesconstructionparahoric2}. From
\cite[\S3.D]{heinloth_hilbertmumfordstability}, we have that $f(0)$ is
the extension of $s^*\xi$, first through the \textit{Levi-factor}
$\mathcal{G}_{\mathsf{P}}\rightarrow\mathcal{L}_{\mathsf{P}}$,
explained in \cite[4.6.10.
Proposition, 4.6.12. Corollaire]{bruhat-tits_groupesreductifsII} and
\cite[Theorem 3.7]{rega_phdthesis},
to obtain $(s^*\xi)(\mathcal{L}_{\mathsf{P}})$, then second through the
isomorphism and inclusion
$\mathcal{L}_{\mathsf{P}}\cong\mathcal{L}(\boldsymbol{\lambda}_{\eta})\hookrightarrow\mathcal{G}$.
In short, $f(0)$ is the extension of Levi-factor of the reduction induced by
Lemma \ref{lem_reesconstructionparahoric}.

\begin{theorem}[(Jordan--H{\"o}lder theory for parahoric
    torsors)]\label{thm_jhparahoric}
    Let $\underline{\chi}$ be an admissible stability parameter in the
    sense of \cite[\S3.F]{heinloth_hilbertmumfordstability}, and let
    $\xi$ be an $\mathcal{L}_{\det,\underline{\chi}}$-semistable parahoric
    $\mathcal{G}$-torsor. There exists a reduction
    $s:X\rightarrow\xi/\mathcal{G}_{\mathsf{P}}$ to a parabolic subgroup
    $\mathcal{G}_{\mathsf{P}}$ of $\mathcal{G}$ such that:
    \begin{enumerate}[label=(\roman*)]
            \sloppy
        \item\label{thm_jhparahoric1} The reduction $s$ is
            \textit{admissible}, i.e., for every vertex
            $\mathfrak{v}\in\pi_0(t(\mathrm{Ad}(s^*\xi)))$, we have
            $n_{\underline{\chi}}(\mathrm{Ad}(s^*\xi),\mathfrak{v})=0$.

        \item\label{thm_jhparahoric2} The extension to the
            Levi-factor $(s^*\xi)(\mathcal{L}_{\mathsf{P}})$, through the
            quotient
            $\mathcal{G}_{\mathsf{P}}\rightarrow\mathcal{L}_{\mathsf{P}}$, is
            $\mathcal{L}_{\det,\underline{\chi}^{\mathcal{L}_\mathsf{P}}}$-stable,
            with
            $\underline{\chi}^{\mathcal{L}_\mathsf{P}}$ induced by the
            isomorphism
            $\mathcal{L}_\mathsf{P}\cong\mathcal{L}(\boldsymbol{\lambda}_\eta)$
            as the
            restriction of $\underline{\chi}$.
    \end{enumerate}
    These reductions have the same extension
    $(s^*\xi)(\mathcal{L}_{\mathsf{P}})$ to the Levi-factor, up to an
    isomorphism.
\end{theorem}

We call such reductions \textit{Jordan--H{\"o}lder reductions} of $\xi$.

\begin{proof}
    We know that $\mathrm{Ad}(\xi)$ is rationally split by \cite[Lemma
    3.8]{heinloth_hilbertmumfordstability} and Lemma
    \ref{lem_reesconstructionparahoric}. Thus,
    $d^{\underline{\chi}}:=(d_{\mathrm{Ad}(\xi)_{\eta},\mathbf{T}_{\eta}}^{\underline{\chi}})_{\mathbf{T}_\eta}$,
    for all maximal tori $\mathbf{T}_{\eta}$ of
    $\mathrm{Ad}(\xi)_{\eta}$, forms a compatible family of
    complementary polyhedra for $\mathrm{Ad}(\xi)_{\eta}$, using Lemma
    \ref{lem_complementarypolyhedronparahoric}. Using Lemma
    \ref{lem_stabilityequivalenceparahoric}
    \ref{lem_stabilityequivalenceparahoric1}, we have that
    $\mathrm{Ad}(\xi)_\eta$ is $d^{\underline{\chi}}$-semistable, so
    Theorem \ref{thm_jhgroupscheme} \ref{thm_jhgroupscheme1} is
    applicable, which gives us
    Jordan--H{\"o}lder parabolics $\mathbf{P}$ of $\mathrm{Ad}(\xi)$,
    with respect to $d^{\underline{\chi}}$. We now use
    Lemma \ref{lem_reesconstructionparahoric} to associate parabolic
    reductions $s:X\rightarrow\xi/\mathcal{G}_{\mathsf{P}}$ to
    Jordan--H{\"o}lder
    parabolics $\mathbf{P}$ of $\mathrm{Ad}(\xi)$, with respect to
    $d^{\underline{\chi}}$.
    Since $\mathbf{P}$
    is Jordan--H{\"o}lder, Lemma
    \ref{lem_stabilityequivalenceparahoric}
    \ref{lem_stabilityequivalenceparahoric1} translates the
    $(d^{\underline{\chi}})_{\underline{R}(\mathbf{P})}$-stability
    of the Levi-factor $(s^*\xi)(\mathcal{L}_{\mathsf{P}})$ into
    $\mathcal{L}_{\det,\underline{\chi}^{\mathcal{L}_\mathsf{P}}}$-stability.
    Furthermore, the $\underline{\chi}$-numerical invariants of $\mathbf{P}$
    are $0$. For the uniqueness of the Levi-factor, for two Jordan--H{\"o}lder
    reductions $s_1:X\rightarrow\xi/\mathcal{G}_{\mathsf{P}_1}$ and
    $s_2:X\rightarrow\xi/\mathcal{G}_{\mathsf{P}_2}$, there exists a maximal
    torus $\mathbf{T}_\eta$ inside the corresponding
    Jordan--H{\"o}lder parabolics
    $(\mathbf{P}_{1})_\eta$ and $(\mathbf{P}_{2})_\eta$, by \cite[Exp. XXVI,
    4.1.1]{sga3}. Then Theorem \ref{thm_jhgroupscheme}
    \ref{thm_jhgroupscheme2} implies that there exists an
    isomorphism between $\mathbf{P}_{1}$ and $\mathbf{P}_{2}$, thus the
    Levi-factors $(s_1^*\xi)(\mathcal{L}_{\mathsf{P}_1})$ and
    $(s_2^*\xi)(\mathcal{L}_{\mathsf{P}_2})$ are isomorphic.
\end{proof}

Theorem \ref{thm_jhparahoric} gives us an equivalent
characterization of $\mathcal{L}_{\det,\underline{\chi}}$-polystability,
generalizing Lemma \ref{lem_polystabilityequivalenceprincipal}.

\begin{lemma}[(Equivalent characterization of
    $\mathcal{L}_{\det,\underline{\chi}}$-polystability)]\label{lem_polystabilityequivalenceparahoric}
    Let $\underline{\chi}$ be an admissible stability parameter, and
    let $\xi$ be an $\mathcal{L}_{\det,\underline{\chi}}$-semistable parahoric
    $\mathcal{G}$-torsor. Then $\xi$ is
    $\mathcal{L}_{\det,\underline{\chi}}$-polystable if and only if:
    \begin{enumerate}[label=(\roman*)]
        \item\label{lem_polystabilityequivalenceparahoric1} $\xi$ is
            isomorphic to
            the extension of an
            $\mathcal{L}_{\det,\underline{\chi}^{\mathcal{L}_\mathsf{P}}}$-stable
            parahoric $\mathcal{L}_{\mathsf{P}}$-torsor
            $\xi^{\mathcal{L}_{\mathsf{P}}}$, where
            $\mathcal{L}_{\mathsf{P}}$ is a
            Levi-factor of a parabolic of $\mathcal{G}$.

        \item\label{lem_polystabilityequivalenceparahoric2} For all reductions
            $s:X\rightarrow\xi/\mathsf{P}'$ for which
            $\xi^{\mathcal{L}_{\mathsf{P}}}$ is isomorphic to the Levi-factor
            $(s^*\xi)(\mathcal{L}^{\mathsf{P}'})$, the reduction $s$ is
            admissible in the sense of Theorem
            \ref{thm_jhparahoric} \ref{thm_jhparahoric1}.
    \end{enumerate}
\end{lemma}

\begin{proof}
    Using Lemma \ref{lem_stabilityequivalenceparahoric}
    \ref{lem_stabilityequivalenceparahoric2}, $\xi$ is
    $\mathcal{L}_{\det,\underline{\chi}}$-polystable if and only if the
    complementary polyhedra
    $(d_{\mathrm{Ad}(\xi)_\eta,\mathbf{T}_{\eta}}^{\underline{\chi}})_{\mathbf{T}_{\eta}}$
    from Lemma
    \ref{lem_complementarypolyhedronparahoric} are polystable. By
    \ref{lem_polystabilityequivalencegroupscheme3} from Lemma
    \ref{lem_polystabilityequivalencegroupscheme}, and by the existence of
    Jordan--H{\"o}lder reductions in Theorem \ref{thm_jhparahoric},
    the equivalence follows.
\end{proof}

\begin{remark}[(Bundle moduli problem
    {\cite[\protect\S3.B]{heinloth_hilbertmumfordstability}})]\label{rem_bundlemoduliproblemparahoric}
    Using Remark \ref{rem_geometryparahoric}, for any admissible
    stability parameter $\underline{\chi}$, we claim that
    $(\mathcal{M}\textsc{Bun}_{\mathcal{G}},d^{\underline{\chi}})$
    forms a bundle moduli
    problem, in the sense of Definition
    \ref{def_bundlemoduliproblem}, with $d^{\underline{\chi}}$ induced by
    $d^{\underline{\chi}}_{\mathrm{Ad}(\xi)_{\eta},\mathbf{T}_{\eta}}$
    from Lemma
    \ref{lem_complementarypolyhedronparahoric}. The properties of
    Definition \ref{def_bundlemoduliproblem}
    \ref{def_bundlemoduliproblem1}
    \ref{def_bundlemoduliproblem11}-\ref{def_bundlemoduliproblem14}
    follow from Lemmas \ref{lem_complementarypolyhedronparahoric} and
    \ref{lem_stabilityequivalenceparahoric}, where all
    facets $P$ of $\Phi(\mathrm{Ad}(\xi)_\eta,\mathbf{T}_\eta)$ are
    admissible, in
    the sense of Definition \ref{def_bundlemoduliproblem}
    \ref{def_bundlemoduliproblem1} \ref{def_bundlemoduliproblem11}.
    For Definition \ref{def_bundlemoduliproblem}
    \ref{def_bundlemoduliproblem1} \ref{def_bundlemoduliproblem15},
    we use Lemma \ref{lem_stabilityequivalenceparahoric},
    which implies that
    $\mathcal{L}_{\det,\underline{\chi}}$ has compatible stability to
    $(\mathcal{M}\textsc{Bun}_{\mathcal{G}},d^{\underline{\chi}})$,
    in the sense of Definition
    \ref{def_bundlemoduliproblem} \ref{def_bundlemoduliproblem2}.
\end{remark}

To study the Jordan--H{\"o}lder stratification of
$(\mathcal{M}\textsc{Bun}_{\mathcal{G}},d^{\underline{\chi}})$, let
$\mathrm{JH}_{\mathcal{G}}$ be the
set of isomorphism classes of Levi-factors of $\mathcal{G}$. For
$\mu\in\mathrm{JH}_{\mathcal{G}}$, we define
$(\mathcal{M}\textsc{Bun}_{\mathcal{G}})^{d^{\underline{\chi}}-\mathrm{sst}}_\mu(K)$
as the set
of $K$-points of $d^{\underline{\chi}}$-semistable parahoric
$\mathcal{G}$-torsors $\xi$ with a Jordan--H{\"o}lder reduction
$s:X\rightarrow\xi/\mathcal{G}_{\mathsf{P}}$, such that
$\mathcal{L}_{\mathsf{P}}$ is in the
isomorphism class $\mu$. By construction, it is clear that this
induces the same Jordan--H{\"o}lder stratification as in Subsection
\ref{subsec_jhstratifications}. The following generalizes Theorem
\ref{thm_jhstratificationprincipal} for principal $\mathsf{G}$-bundles.

\begin{theorem}[(Jordan--H{\"o}lder
    stratification)]\label{thm_jhstratificationparahoric}
    Let $K$ be algebraically closed of characteristic $0$, and let
    $\underline{\chi}$ be an admissible stability parameter. The
    Jordan--H{\"o}lder stratification of
    $(\mathcal{M}\textsc{Bun}_{\mathcal{G}},d^{\underline{\chi}})$,
    from Subsection
    \ref{subsec_jhstratifications}, consists of locally closed
    substacks. Moreover,
    $(\mathcal{M}\textsc{Bun}_{\mathcal{G}},d^{\underline{\chi}})$
    exhibits a compatible good moduli space
    $\pi:(\mathcal{M}\textsc{Bun}_{\mathcal{G}})^{d^{\underline{\chi}}-\mathrm{sst}}\rightarrow
    (M\textsc{Bun}_{\mathcal{G}})^{d^{\underline{\chi}}-\mathrm{sst}}$,
    from Definition \ref{def_jhstratification} \ref{def_jhstratification2}.
\end{theorem}

\begin{proof}
    We have a good moduli space
    $\pi:(\mathcal{M}\textsc{Bun}_{\mathcal{G}})^{d^{\underline{\chi}}-\mathrm{sst}}\rightarrow
    (M\textsc{Bun}_{\mathcal{G}})^{d^{\underline{\chi}}-\mathrm{sst}}$
    from \cite[Theorem
    8.1]{alperDHLheinloth_existenceofmodulispacesforalgebraicstacks}. For
    $\mu\in\mathrm{JH}_{\mathcal{G}}$,
    $(\mathcal{M}\textsc{Bun}_{\mathcal{G}})^{d^{\underline{\chi}}-\mathrm{sst}}_\mu$
    is the
    $\pi$-preimage of a subspace
    $(M\textsc{Bun}_{\mathcal{G}})^{d^{\underline{\chi}}-\mathrm{sst}}_\mu$ of
    $(M\textsc{Bun}_{\mathcal{G}})^{d^{\underline{\chi}}-\mathrm{sst}}$
    defined as follows: Due to Lemma
    \ref{lem_polystabilityequivalenceparahoric},
    $(M\textsc{Bun}_{\mathcal{G}})^{d^{\underline{\chi}}-\mathrm{sst}}_\mu$
    is the image of
    a morphism
    $\iota:(M\textsc{Bun}_{\mathcal{L}_{\mathsf{P}}})^{d_{M\textsc{Bun}_{\mathcal{L}_{\mathsf{P}}}}-\mathrm{st}}\rightarrow
    (M\textsc{Bun}_{\mathcal{G}})^{d^{\underline{\chi}}-\mathrm{sst}}$
    extending the
    structure group through the canonical inclusion
    $\mathcal{L}_{\mathsf{P}}\hookrightarrow\mathcal{G}$. Thus,
    $(\mathcal{M}\textsc{Bun}_{\mathcal{G}})^{d^{\underline{\chi}}-\mathrm{sst}}_\mu$
    forms a
    substack of
    $(\mathcal{M}\textsc{Bun}_{\mathcal{G}})^{d^{\underline{\chi}}-\mathrm{sst}}$.
    To show that
    $(\mathcal{M}\textsc{Bun}_{\mathcal{G}})^{d^{\underline{\chi}}-\mathrm{sst}}_\mu$
    is a
    locally closed substack, we claim that $\iota$ is an immersion of
    algebraic spaces. Since $d^{\underline{\chi}}$-polystable points
    in
    $(\mathcal{M}\textsc{Bun}_{\mathcal{G}})^{d^{\underline{\chi}}-\mathrm{sst}}_\mu(K)$
    correspond uniquely to $K$-points in
    $(M\textsc{Bun}_{\mathcal{L}_{\mathsf{P}}})^{d_{M\textsc{Bun}_{\mathcal{L}_{\mathsf{P}}}}-\mathrm{st}}(K)$,
    $\iota$ is a monomorphism. For a $K$-scheme $T$ and a family
    $\xi_T\in M\textsc{Bun}_{\mathcal{G}}(T)$ of parahoric
    $\mathcal{G}$-torsors, the locus
    $T'$ in $T$, for which
    $\xi_{T'}/\mathcal{G}_{\mathsf{P}}\rightarrow T'\times X$
    admits a section, parametrizes the family $\xi_{T'}$ over $T'$
    precisely in the image of $\iota$. It suffices to show that $T'$
    is locally closed in $T$: By trivializing locally inside
    $T'\times X$, sections of
    $\xi_{T'}/\mathcal{G}_{\mathsf{P}}\rightarrow T'\times X$
    map into Bruhat cells of $\mathcal{G}/\mathcal{G}_{\mathsf{P}}$,
    which are locally
    closed, hence, $T'$ is locally closed in $T$. Thus, $\iota$ is an
    immersion of algebraic spaces locally of finite type, so
    $(\mathcal{M}\textsc{Bun}_{\mathcal{G}})^{d^{\underline{\chi}}-\mathrm{sst}}_\mu$
    is a
    locally closed substack of
    $(\mathcal{M}\textsc{Bun}_{\mathcal{G}})^{d^{\underline{\chi}}-\mathrm{sst}}$.

    The fact that $\pi$ is a compatible good moduli space comes from
    the correspondence between $d^{\underline{\chi}}$-polystable
    points in
    $(\mathcal{M}\textsc{Bun}_{\mathcal{G}})^{d^{\underline{\chi}}-\mathrm{sst}}_\mu(K)$
    and $K$-points in
    $(M\textsc{Bun}_{\mathcal{L}_{\mathsf{P}}})^{d_{M\textsc{Bun}_{\mathcal{L}_{\mathsf{P}}}}-\mathrm{st}}(K)$,
    which implies that quasi-S-equivalence implies S-equivalence.
\end{proof}

\begin{example}[(Jordan--H{\"o}lder filtrations of a parahoric
    torsor)]\label{ex_jhfiltrationsparahoric}
    Let $F$ be a slope-semistable rank $2$ vector bundle over $X$,
    given as an extension of two degree $0$ line bundles $L_0$ and
    $L_0'$ over $X$, i.e., we have the short exact sequence:
    \begin{equation*}
        0\rightarrow L_0\rightarrow F\rightarrow L'_0\rightarrow0
    \end{equation*}
    As in Remark \ref{rem_serrevanishing} and Example
    \ref{ex_jhfiltrationsvectorbundle}, let
    $L':=L'_0(-nD)^{\mathrm{sat}}$ be the saturation of a twist of
    $L'_0$, where $D\neq0$ is an effective divisor, such that $L'_0$
    embeds as a subbundle of $F$, with the splitting $F_\eta\cong
    (L_0)_\eta\oplus(L')_\eta$.  Fixing a point $x\in X$, we
    introduce a twist in $F$ through a \textit{Hecke transform} given
    by a short exact sequence of coherent $\mathcal{O}_X$-modules:
    \begin{equation*}
        0\rightarrow \widetilde{F}\rightarrow
        F\rightarrow\mathcal{O}_{X,x}\rightarrow 0
    \end{equation*}
    where the cokernel of $\widetilde{F}\rightarrow F$ is given by
    $(L_0)_x$. By passing to frame bundles, $F$ can be encoded by a
    parahoric $\mathcal{G}$-torsor $\mathrm{Fr}(F)$, where
    $\mathcal{G}:=X\times\mathbf{SL}(2,K)$, and $\widetilde{F}$ can
    be encoded by
    a parahoric $\widetilde{\mathcal{G}}$-torsor $\widetilde{\mathrm{Fr}(F)}$,
    where $\widetilde{\mathcal{G}}$ is the parahoric group scheme over $X$
    defined as follows:
    \begin{enumerate}[label=(\roman*)]
        \item
            $\widetilde{\mathcal{G}}_{X\setminus\{x\}}\cong(X\setminus\{x\})\times
            \mathbf{SL}(2,K)$.

        \item For a local coordinate $t$ vanishing at $x\in\mathcal{R}$,
            $\widetilde{\mathcal{G}}_{\mathrm{Spec}(K\llbracket
            t\rrbracket)}$ is a
            parahoric group scheme such that
            $\widetilde{\mathcal{G}}_{\mathrm{Spec}(K\llbracket
            t\rrbracket)}(\mathrm{Spec}(K\llbracket t\rrbracket))$ is
            isomorphic to:
            \begin{equation*}
                \left\{
                    \begin{pmatrix}
                        *  & t^{-1}* \\
                        t* & *       \\
                    \end{pmatrix} \ \middle\vert\ *\in K\llbracket
                t\rrbracket\right\}\subset L\mathbf{SL}(2,K)
            \end{equation*}
            through the Hecke transform.
    \end{enumerate}
    We set the stability parameters
    $\underline{\chi}:=(\chi_x:\mathcal{G}_x\rightarrow\mathbb{G}_{m,x})_x$ and
    $\underline{\widetilde{\chi}}:=(\widetilde{\chi}_x:\widetilde{\mathcal{G}}_x\rightarrow\mathbb{G}_{m,x})_x$
    to be the trivial characters, which are admissible.
    Following Wi{\ss}dorf's method in
    \cite[3.3.11 Remark]{wissdorf_phdthesis}, for the maximal tori
    $\mathbf{T}_\eta$ of $\mathrm{Ad}(\mathrm{Fr}(F))$ and
    $\widetilde{\mathbf{T}}_\eta$ of
    $\mathrm{Ad}(\widetilde{\mathrm{Fr}(F)})$ induced by the splitting of $F$ by
    $L_0$ and $L'$, the convex hulls of
    $d^{\underline{\chi}}_{\mathrm{Ad}(\mathrm{Fr}(F))_\eta,\mathbf{T}_\eta}$
    and
    $d^{\underline{\widetilde{\chi}}}_{\mathrm{Ad}(\widetilde{\mathrm{Fr}(F)})_\eta,\widetilde{\mathbf{T}}_\eta}$
    can be identified with the respective intervals in $\mathbb{R}$:
    \begin{align*}
        [0,-l']& &[-1/2,-l'-1/2]
    \end{align*}
    where $l':=\deg(L')\leqslant0$, and the root system $\Phi$ is $A_1$.
    Depending on the value of $l'$, if the intervals contain $0$, the
    Jordan--H{\"o}lder facets can be indentified out of the facets
    $\mathbb{R}_{>0}$, $\mathbb{R}_{<0}$, and $\{0\}$, of $A_1$. Then through
    Theorem \ref{thm_jhparahoric}, the corresponding
    Jordan--H{\"o}lder reductions of $\mathrm{Fr}(F)$ and
    $\widetilde{\mathrm{Fr}(F)}$ can
    be identified, which correspond to the Jordan--H{\"o}lder filtrations
    of $F$ and $\widetilde{F}$ induced by Theorem \ref{thm_jhprincipal}.

    The Hecke transform corresponds to a reduction and extension of
    structure group $\mathcal{G}$ to $\widetilde{\mathcal{G}}$, that restricts
    and extends $\mathrm{Fr}(F)$ to $\widetilde{\mathrm{Fr}(F)}$,
    giving an example
    of the morphisms of moduli stacks in
    \cite[\S8.2.1]{balaji-sesh_parahorictorsors}. Notably, the
    numerical invariants of the complementary polyhedron of the Hecke
    transform are shifted, so stability conditions are not always preserved.
\end{example}

\subsection{Parahoric Higgs torsors}\label{subsec_parahorichiggs}

Higgs bundles with poles were studied by Markman and Bottacin in
\cite{markman_spectralcurvesandintegrablesystems,bottacin_symplecticgeometryonmodulispacesofstablepairs}.
This behavior locally around poles can be modeled by a parahoric
torsor equipped with a Higgs field, from the introduction in
\ref{intro_applications}
\ref{intro_applications6}. We now study its
Jordan--H{\"o}lder theory.

We focus on \textit{weak} parahoric Higgs torsors, also called
\textit{logahoric Higgs torsors} in
\cite{kydonakis-sun-zhao_logahorichiggstorsors}, although our
results also apply to \textit{strong} parahoric Higgs
torsors, as studied in
\cite{baraglia-kamgarpour-varma_completeintegrabilityoftheparahorichitchinsystem,rega_phdthesis}.

\subsubsection{Definition and stability of parahoric Higgs
torsors}\label{subsubsec_stabilityparahorichiggs}

Let $D=\sum_{x^i\in\mathcal{R}}x^i$ be the effective divisor corresponding
to $\mathcal{R}$, and let
$\mathcal{K}(D):=\mathcal{K}\otimes\mathcal{O}_X(D)$ be the twisted
canonical bundle over $X$. Let $\mathcal{G}$ be a parahoric group
scheme over $X$.

\begin{definition}[(Parahoric Higgs torsors)]\label{def_parahorichiggs}
    \leavevmode
    \begin{enumerate}[label=(\alph*)]
        \item\label{def_parahorichiggs1} A \textit{(weak) parahoric
            Higgs $\mathcal{G}$-torsor} $(\xi,\varphi)$ consists of a parahoric
            $\mathcal{G}$-torsor $\xi$ and a \textit{(weak) parahoric
            Higgs field}
            $\varphi$ of $\xi$, i.e., $\varphi\in
            H^0(X,\mathrm{ad}(\xi)\otimes\mathcal{K}(D))$. We denote
            the moduli stack of
            parahoric Higgs $\mathcal{G}$-torsors by
            $\mathcal{M}\textsc{Bun}^{H}_{\mathcal{G}}$, a stack over
            $\mathrm{Spec}(K)$.

        \item\label{def_parahorichiggs2} For a parabolic subgroup
            $\mathsf{P}$ of $\mathsf{G}$, a \textit{Higgs reduction} of a
            parahoric Higgs $\mathcal{G}$-torsor $(\xi,\varphi)$ is a reduction
            $s:X\rightarrow\xi/\mathcal{G}_{\mathsf{P}}$ such that
            $\varphi$ factors
            through the inclusion
            $\iota:\mathrm{ad}(s^*\xi)\hookrightarrow\mathrm{ad}(\xi)$,
            i.e., there exists
            a global section
            $\varphi_\mathsf{P}:X\rightarrow\mathrm{ad}(s^*\xi)\otimes\mathcal{K}(D)$,
            such that
            $\varphi=(\iota\otimes\mathrm{id}_{\mathcal{K}(D)})\circ\varphi_\mathsf{P}$.
    \end{enumerate}
\end{definition}

Let
$\textsc{For}:\mathcal{M}\textsc{Bun}^{H}_{\mathcal{G}}\rightarrow\mathcal{M}\textsc{Bun}_{\mathcal{G}}$
be the
morphism forgetting the Higgs field. For a stability parameter
$\underline{\chi}$, we define stability conditions on
$\mathcal{M}\textsc{Bun}^{H}_{\mathcal{G}}$ by pulling back the line bundle
$\mathcal{L}_{\det,\underline{\chi}}$ from
\cite[\S3.B]{heinloth_hilbertmumfordstability} to the line bundle
$\mathcal{L}_{\det,\underline{\chi}}^H:=\textsc{For}^*\mathcal{L}_{\det,\underline{\chi}}$
over $\mathcal{M}\textsc{Bun}^{H}_{\mathcal{G}}$. Definition
\ref{def_Lstability} is applicable
due to the following remark.

\begin{remark}[(Geometry of
    $\mathcal{M}\textsc{Bun}_{\mathcal{G}}^H$)]\label{rem_geometryparahorichiggs}
    The forgetful morphism
    $\textsc{For}:\mathcal{M}\textsc{Bun}^{H}_{\mathcal{G}}\rightarrow\mathcal{M}\textsc{Bun}_{\mathcal{G}}$
    is representable by schemes, and
    since $\mathcal{M}\textsc{Bun}_{\mathcal{G}}$ is algebraic by Remark
    \ref{rem_geometryparahoric},
    $\mathcal{M}\textsc{Bun}^{H}_{\mathcal{G}}$ is an
    algebraic, similar to the proof of \cite[Theorem
    7.18]{casalainamartin-wise_anintroductiontomodulistacksviewtowardshiggsbundles}.
    Analogously,
    $\textsc{For}:\mathcal{M}\textsc{Bun}^{H}_{\mathcal{G}}\rightarrow\mathcal{M}\textsc{Bun}_{\mathcal{G}}$
    is locally of finite type, hence,
    $\mathcal{M}\textsc{Bun}^{H}_{\mathcal{G}}$ is locally of
    finite type. In an argument similar to \cite[Proposition
    4.1]{rega_phdthesis} and Remark \ref{rem_geometryparahoric},
    $\mathcal{M}\textsc{Bun}^{H}_{\mathcal{G}}$ has affine diagonal.
\end{remark}

Similar to Lemma \ref{lem_reesconstructionparahoric}
\ref{lem_reesconstructionparahoric2}, we have the following Rees construction.

\begin{lemma}[(Rees construction)]\label{lem_reesconstructionparahorichiggs}
    For a parahoric Higgs $\mathcal{G}$-torsor $(\xi,\varphi)$, the following
    are in bijection:
    \begin{enumerate}[label=(\roman*)]
        \item\label{lem_reesconstructionparahorichiggs1} Filtrations
            $f:\Theta_K\rightarrow\mathcal{M}\textsc{Bun}_{\mathcal{G}}^H$
            of $(\xi,\varphi)$.
        \item\label{lem_reesconstructionparahorichiggs2} Higgs
            reductions $s:X\rightarrow\xi/\mathcal{G}_{\mathsf{P}}$ of $\xi$ to
            parabolic subgroups $\mathcal{G}_{\mathsf{P}}$ of $\mathcal{G}$.
    \end{enumerate}
\end{lemma}

\begin{proof}
    The filtration $\textsc{For}\circ
    f:\Theta_K\rightarrow\mathcal{M}\textsc{Bun}_{\mathcal{G}}$
    induces a reduction $s:X\rightarrow\xi/\mathcal{G}_{\mathsf{P}}$ due to
    Lemma \ref{lem_reesconstructionparahoric}
    \ref{lem_reesconstructionparahoric2}. Since the automorphism group
    schemes of $\mathcal{M}\textsc{Bun}_{\mathcal{G}}^H$ preserve the
    Higgs field, the bijection
    in Lemma \ref{lem_reesconstructionparahoric}
    \ref{lem_reesconstructionparahoric2} restricts to the claimed
    bijection in this lemma.
\end{proof}

Similarly to the case of Higgs bundles studied by Wi{\ss}dorf in
\cite[\S3.5]{wissdorf_phdthesis}, for a parahoric Higgs $\mathcal{G}$-torsor
$(\xi,\varphi)$,
we can modify the complementary polyhedron
$d^{\underline{\chi}}_{\mathrm{Ad}(\xi)_\eta,\mathbf{T}_\eta}$ from Lemma
\ref{lem_complementarypolyhedronparahoric} in order to detect
$\mathcal{L}_{\det,\underline{\chi}}^H$-stability conditions. For
$\eta:=\mathrm{Spec}(K(X))$ and a maximal torus $\mathbf{T}_{\eta}$ of
$\mathrm{ad}(\xi)_{\eta}$, we have the root space decomposition:
\begin{equation*}
    \mathrm{ad}(\xi)_{\eta}=\mathrm{Lie}(\mathbf{T}_{\eta})\oplus\bigoplus_{\alpha\in\Phi(\mathrm{Ad}(\xi)_{\eta},\mathbf{T}_{\eta})}(\mathrm{ad}(\xi)_{\eta})_\alpha
\end{equation*}
By
\cite[{\href{https://stacks.math.columbia.edu/tag/0AY7}{0AY7}}]{stacks-project},
the bundles
$(\mathrm{ad}(\xi)_{\eta})_\alpha\otimes\mathcal{K}(D)_{\eta}$ extend
uniquely to $\mathrm{ad}(\xi)_\alpha\otimes\mathcal{K}(D)$ and
$\varphi:X\rightarrow\mathrm{ad}(\xi)\otimes\mathcal{K}(D)$ can be
projected down
to $\varphi_\alpha:X\rightarrow\mathrm{ad}(\xi)_\alpha\otimes\mathcal{K}(D)$.
Following \cite[\S3.5.1]{wissdorf_phdthesis}, by looking at
decompositions of $\alpha$ into indecomposable roots, Wi{\ss}dorf defines for
all $\alpha\in\Phi(\mathrm{Ad}(\xi)_{\eta},\mathbf{T}_{\eta})$:
\begin{equation*}
    \epsilon(\varphi,\alpha):=
    \begin{cases}
        0 \text{ if for all decompositions $\alpha=\sum_{i=1}^s\alpha_i$,
        there is an $i$ with }\varphi_{\alpha_i}=0                 \\
        1 \text{ if there exists a decomposition
            $\alpha=\sum_{i=1}^s\alpha_i$ such that
        }\varphi_{\alpha_i}\neq0\text{ for all $i$} \\
    \end{cases}
\end{equation*}
that depends on the Higgs field. R{\'e}ga then proves the following.

\begin{lemma}[(Complementary polyhedra for parahoric Higgs torsors
    {\cite[\protect\S5.2.2]{rega_phdthesis}})]\label{lem_complementarypolyhedronparahorichiggs}
    Let $\underline{\chi}$ be an admissible stability parameter, and
    let $(\xi,\varphi)$ be a parahoric Higgs $\mathcal{G}$-torsor. Let
    $\mathbf{T}_{\eta}$ be a maximal torus of $\mathrm{Ad}(\xi)_{\eta}$, and let
    $d_{\mathrm{Ad}(\xi)_{\eta},\mathbf{T}_{\eta}}^{\underline{\chi}}$ be the
    complementary polyhedron constructed in Lemma
    \ref{lem_complementarypolyhedronparahoric}. For all $C\geqslant0$, the
    following is a complementary polyhedron:
    \begin{align*}
        d_{\mathrm{Ad}(\xi)_{\eta},\mathbf{T}_{\eta}}^{\underline{\chi},\varphi,C}:\mathfrak{W}_{\Phi(\mathrm{Ad}(\xi)_{\eta},\mathbf{T}_{\eta})}\rightarrow
        V^*_{\mathrm{Ad}(\xi)_{\eta},\mathbf{T}_{\eta}} &  & \mathfrak{c}\mapsto
        d_{\mathrm{Ad}(\xi)_{\eta},\mathbf{T}_{\eta}}^{\chi}(\mathfrak{c})+C\sum_{\substack{\alpha\in\Phi(\mathrm{Ad}(\xi)_{\eta},\mathbf{T}_{\eta})
                \\ \alpha\notin
        R(\mathfrak{c})}}\epsilon(\varphi,\alpha)\widecheck{\alpha}
    \end{align*}
\end{lemma}

For $C\geqslant0$ large enough, we identify
$d_{\mathrm{Ad}(\xi)_{\eta},\mathbf{T}_{\eta}}^{\underline{\chi}}$-stability
conditions with $\mathcal{L}_{\det,\underline{\chi}}^H$-stability
conditions, generalizing Lemmas
\ref{lem_stabilityequivalenceprincipal} and
\ref{lem_stabilityequivalenceparahoric}.

\begin{lemma}[(Stability conditions for parahoric Higgs
    torsors)]\label{lem_stabilityequivalenceparahorichiggs}
    Let $\underline{\chi}$ be an admissible stability parameter, and
    let $(\xi,\varphi)$ be a parahoric Higgs $\mathcal{G}$-torsor. There exists
    $C\geqslant0$ such that we have the following:
    \begin{enumerate}[label=(\roman*)]
        \item\label{lem_stabilityequivalenceparahorichiggs1}
            $(\xi,\varphi)$ is
            $\mathcal{L}_{\det,\underline{\chi}}^H$-(semi)-stable if and only if
            $(\xi,\varphi)$ is
            $d_{\mathrm{Ad}(\xi)_{\eta},\mathbf{T}_{\eta}}^{\underline{\chi},\varphi,C}$-(semi)-stable
            for all maximal tori $\mathbf{T}_{\eta}$.

        \item\label{lem_stabilityequivalenceparahorichiggs2}
            $(\xi,\varphi)$ is
            $\mathcal{L}_{\det,\underline{\chi}}^H$-polystable if and
            only if $(\xi,\varphi)$ is
            $d_{\mathrm{Ad}(\xi)_{\eta},\mathbf{T}_{\eta}}^{\underline{\chi},\varphi,C}$-polystable
            for all maximal tori $\mathbf{T}_{\eta}$.
    \end{enumerate}
\end{lemma}

\begin{proof}
    Due to the Rees construction in Lemma
    \ref{lem_reesconstructionparahorichiggs},
    $\mathcal{L}_{\det,\underline{\chi}}^H$-stability conditions are
    equivalent to testing $\mathcal{L}_{\det,\underline{\chi}}$-stability
    conditions over Higgs reductions, in the sense of Definition
    \ref{def_parahorichiggs} \ref{def_parahorichiggs2}.

    For any maximal torus $\mathbf{T}_{\eta}$ of
    $\mathrm{Ad}(\xi)_{\eta}$, let $\deg(P)$ denote the degree of $P$
    with respect to
    $d_{\mathrm{Ad}(\xi)_{\eta},\mathbf{T}_{\eta}}^{\underline{\chi}}$
    from Lemma
    \ref{lem_complementarypolyhedronparahoric}, and let
    $\deg^{\varphi,C}(P)$ denote the degree of $P$ with respect to
    $d_{\mathrm{Ad}(\xi)_{\eta},\mathbf{T}_{\eta}}^{\underline{\chi},\varphi,C}$.
    If
    $P$ admits a Higgs reduction, through Lemmas
    \ref{lem_rootsystemgroupscheme} and
    \ref{lem_reesconstructionparahorichiggs}, then
    $\deg(P)=\deg^{\varphi,C}(P)$ by \cite[3.5.13
    Lemma]{wissdorf_phdthesis}. Hence, by Lemma
    \ref{lem_stabilityequivalenceparahoric}, it suffices to
    show that for some $C\geqslant0$ and for any maximal torus
    $\mathbf{T}_\eta$, if a facet $P$ of
    $\Phi(\mathrm{Ad}(\xi)_{\eta},\mathbf{T}_{\eta})$ does not admit a Higgs
    reduction, then $\deg(P)<0$.  By \cite[3.5.14
    Remark]{wissdorf_phdthesis}, we have:
    \begin{equation*}
        \deg^{\varphi,C}(P)\leqslant\deg(P)-C
    \end{equation*}
    Due to \cite[Proposition 3.16]{heinloth_hilbertmumfordstability},
    $\mathrm{Ad}(\xi)$ has bounded $d^{\underline{\chi}}$-stability,
    where
    $d^{\underline{\chi}}:=(d^{\underline{\chi}}_{\mathrm{Ad}(\xi)_\eta,\mathbf{T}_\eta})_{\mathbf{T}_\eta}$
    is a compatible family of complementary polyhedra from Lemma
    \ref{lem_complementarypolyhedronparahoric}. Thus, $\deg(P)$ is
    bounded above for all maximal tori $\mathbf{T}_\eta$ and for all
    possible facets $P$. Hence, we can pick $C$ large enough such
    that this always
    ensures $\deg^{\varphi,C}(P)<0$.
\end{proof}

Essentially, Lemma \ref{lem_stabilityequivalenceparahorichiggs} is
shifting the degrees of facets not admitting
Higgs reductions far enough so that we ignore stability conditions on
such reductions.

\subsubsection{Jordan--H{\"o}lder theory for parahoric Higgs
torsors}\label{subsubsec_jhparahorichiggs}

We can now state and prove the Jordan--H{\"o}lder theorem for parahoric
Higgs torsors.

\begin{theorem}[(Jordan--H{\"o}lder theory for parahoric Higgs
    torsors)]\label{thm_jhparahorichiggs}
    Let $\underline{\chi}$ be an admissible stability parameter, in the
    sense of \cite[\S3.F]{heinloth_hilbertmumfordstability}, and let
    $(\xi,\varphi)$ be an $\mathcal{L}_{\det,\underline{\chi}}^H$-semistable
    parahoric Higgs $\mathcal{G}$-torsor. There exists a Higgs reduction
    $s:X\rightarrow\xi/\mathcal{G}_{\mathsf{P}}$ to a parabolic subgroup
    $\mathcal{G}_{\mathsf{P}}$ of $\mathcal{G}$ such that:
    \begin{enumerate}[label=(\roman*)]
            \sloppy
        \item\label{thm_jhparahorichiggs1} The reduction $s$ is
            \textit{admissible}, i.e., for every vertex
            $\mathfrak{v}\in\pi_0(t(\mathrm{Ad}(s^*\xi)))$, we have
            $n_{\underline{\chi}}(\mathrm{Ad}(s^*\xi),\mathfrak{v})=0$.

        \item\label{thm_jhparahorichiggs2} The extension to the
            Levi-factor
            $((s^*\xi)(\mathcal{L}_{\mathsf{P}}),\varphi_{(s^*\xi)(\mathcal{L}_{\mathsf{P}})})$,
            through the quotient
            $\mathcal{G}_{\mathsf{P}}\rightarrow\mathcal{L}_{\mathsf{P}}$, is
            $\mathcal{L}_{\det,\underline{\chi}^{\mathcal{L}_{\mathsf{P}}}}^H$-stable,
            with $\underline{\chi}^{\mathcal{L}_\mathsf{P}}$ induced by the
            isomorphism
            $\mathcal{L}_\mathsf{P}\cong\mathcal{L}(\boldsymbol{\lambda}_\eta)$
            as the
            restriction of $\underline{\chi}$.
    \end{enumerate}
    These reductions have the same extension
    $((s^*\xi)(\mathcal{L}_{\mathsf{P}}),\varphi_{(s^*\xi)(\mathcal{L}_{\mathsf{P}})})$
    to
    the Levi-factor, up to an isomorphism.
\end{theorem}

We call such reductions \textit{Jordan--H{\"o}lder Higgs reductions}
of $(\xi,\varphi)$.

\begin{proof}
    We know that $\mathrm{Ad}(\xi)$ is rationally split by
    Lemma \ref{lem_reesconstructionparahoric} and \cite[Lemma
    3.8]{heinloth_hilbertmumfordstability}. Thus, for $C\geqslant0$
    chosen as in Lemma
    \ref{lem_stabilityequivalenceparahorichiggs},
    $d^{\underline{\chi},\varphi,C}:=(d_{\mathrm{Ad}(\xi)_{\eta},\mathbf{T}_{\eta}}^{\underline{\chi},\varphi,C})_{\mathbf{T}_\eta}$,
    for all maximal tori $\mathbf{T}_{\eta}$ of
    $\mathrm{Ad}(\xi)_{\eta}$, forms a compatible family of
    complementary polyhedra for $\mathrm{Ad}(\xi)_{\eta}$, using Lemma
    \ref{lem_complementarypolyhedronparahorichiggs}. Using Lemma
    \ref{lem_stabilityequivalenceparahorichiggs}
    \ref{lem_stabilityequivalenceparahorichiggs1}, we have that
    $\mathrm{Ad}(\xi)_\eta$ is
    $d^{\underline{\chi},\varphi,C}$-semistable, so Theorem
    \ref{thm_jhgroupscheme} \ref{thm_jhgroupscheme1} is applicable,
    which gives us
    Jordan--H{\"o}lder parabolics $\mathbf{P}$ of $\mathrm{Ad}(\xi)$,
    with respect to $d^{\underline{\chi},\varphi,C}$, that admit a
    Higgs reduction. We
    now use Lemma \ref{lem_reesconstructionparahorichiggs} to associate
    parabolic Higgs reductions $s:X\rightarrow\xi/\mathcal{G}_{\mathsf{P}}$ to
    Jordan--H{\"o}lder parabolics $\mathbf{P}$ of $\mathrm{Ad}(\xi)$,
    with respect to $d^{\underline{\chi},\varphi,C}$. Since $\mathbf{P}$
    is Jordan--H{\"o}lder, Lemma
    \ref{lem_stabilityequivalenceparahorichiggs}
    \ref{lem_stabilityequivalenceparahorichiggs1} translates the
    $(d^{\underline{\chi},\varphi,C})_{\underline{R}(\mathbf{P})}$-stability
    of the Levi-factor $(s^*\xi)(\mathcal{L}_{\mathsf{P}})$ into
    $\mathcal{L}_{\det,\underline{\chi}^{\mathcal{L}_\mathsf{P}}}^H$-stability.
    Furthermore, the $\underline{\chi}$-numerical Higgs invariants of
    $\mathbf{P}$
    are $0$, i.e., for every vertex
    $\mathfrak{v}\in\pi_0(t(\mathrm{Ad}(s^*\xi)))$, we have
    $n_{\underline{\chi}}(\mathrm{Ad}(s^*\xi),\mathfrak{v})=0$. For
    the uniqueness of the Levi-factor, for two Jordan--H{\"o}lder Higgs
    reductions $s_1:X\rightarrow\xi/\mathcal{G}_{\mathsf{P}_1}$ and
    $s_2:X\rightarrow\xi/\mathcal{G}_{\mathsf{P}_2}$, there exists a maximal
    torus $\mathbf{T}_\eta$ inside the corresponding
    Jordan--H{\"o}lder parabolics
    $(\mathbf{P}_{1})_\eta$ and $(\mathbf{P}_{2})_\eta$, by \cite[Exp. XXVI,
    4.1.1]{sga3}. Then Theorem \ref{thm_jhgroupscheme}
    \ref{thm_jhgroupscheme2} implies that there exists an
    isomorphism between $\mathbf{P}_{1}$ and $\mathbf{P}_{2}$, thus the
    Levi-factors $(s_1^*\xi)(\mathcal{L}_{\mathsf{P}_1})$ and
    $(s_2^*\xi)(\mathcal{L}_{\mathsf{P}_2})$ are isomorphic.
\end{proof}

Theorem \ref{thm_jhparahorichiggs} gives us an equivalent
characterization of $\mathcal{L}_{\det,\underline{\chi}}$-polystability,
generalizing Lemmas \ref{lem_polystabilityequivalenceprincipal} and
\ref{lem_polystabilityequivalenceparahoric}.

\begin{lemma}[(Equivalent characterization of
    $\mathcal{L}_{\det,\underline{\chi}}^H$-polystability)]\label{lem_polystabilityequivalenceparahorichiggs}
    Let $\underline{\chi}$ be an admissible stability parameter, and
    let $(\xi,\varphi)$ be an
    $\mathcal{L}_{\det,\underline{\chi}}^H$-semistable parahoric Higgs
    $\mathcal{G}$-torsor. Then $(\xi,\varphi)$ is
    $\mathcal{L}_{\det,\underline{\chi}}^H$-polystable if and only if:
    \begin{enumerate}[label=(\roman*)]
            \sloppy
        \item\label{lem_polystabilityequivalenceparahorichiggs1}
            $(\xi,\varphi)$ is
            isomorphic to the extension of an
            $\mathcal{L}_{\det,\underline{\chi}^{\mathcal{L}_\mathsf{P}}}^H$-stable
            parahoric Higgs $\mathcal{L}_{\mathsf{P}}$-torsor
            $(\xi^{\mathcal{L}_{\mathsf{P}}},\varphi^{\mathcal{L}_{\mathsf{P}}})$,
            where $\mathcal{L}_{\mathsf{P}}$ is a Levi-factor of a
            parabolic of $\mathcal{G}$.

        \item\label{lem_polystabilityequivalenceparahorichiggs2} For all Higgs
            reductions $s:X\rightarrow\xi/\mathsf{P}'$ for which
            $(\xi^{\mathcal{L}_{\mathsf{P}}},\varphi^{\mathcal{L}_{\mathsf{P}}})$
            is
            isomorphic to the Levi-factor
            $((s^*\xi)(\mathcal{L}^{\mathsf{P}'}),\varphi_{(s^*\xi)(\mathcal{L}^{\mathsf{P}'})})$,
            the reduction $s$ is admissible in the sense of Theorem
            \ref{thm_jhparahorichiggs}
            \ref{thm_jhparahorichiggs1}.
    \end{enumerate}
\end{lemma}

\begin{proof}
    Using Lemma \ref{lem_stabilityequivalenceparahorichiggs}
    \ref{lem_stabilityequivalenceparahorichiggs2}, $\xi$ is
    $\mathcal{L}_{\det,\underline{\chi}}^H$-polystable if and only if the
    complementary polyhedra
    $(d_{\mathrm{Ad}(\xi)_\eta,\mathbf{T}_{\eta}}^{\underline{\chi},\varphi,C})_{\mathbf{T}_\eta}$
    from Lemma
    \ref{lem_complementarypolyhedronparahorichiggs} are polystable. By
    \ref{lem_polystabilityequivalencegroupscheme3} from Lemma
    \ref{lem_polystabilityequivalencegroupscheme}, and by the existence of
    Jordan--H{\"o}lder Higgs reductions in Theorem
    \ref{thm_jhparahorichiggs}, the equivalence follows.
\end{proof}

\begin{remark}[(Bundle moduli
    problem)]\label{rem_bundlemoduliproblemparahorichiggs}
    Using Remark \ref{rem_geometryparahorichiggs}, for any
    admissible stability parameter $\underline{\chi}$, we claim that
    $(\mathcal{M}\textsc{Bun}_{\mathcal{G}}^H,d^{\underline{\chi},H})$
    forms a bundle moduli
    problem, in the sense of Definition
    \ref{def_bundlemoduliproblem}, where $d^{\underline{\chi},H}$ is induced by
    $(d^{\underline{\chi},\varphi,C}_{\mathrm{Ad}(\xi)_{\eta},\mathbf{T}_{\eta}})_{C}$
    from
    Lemma \ref{lem_complementarypolyhedronparahorichiggs},
    where $C\geqslant0$ is parametrized by values where Lemma
    \ref{lem_stabilityequivalenceparahorichiggs} applies.
    The properties of Definition \ref{def_bundlemoduliproblem}
    \ref{def_bundlemoduliproblem1}
    \ref{def_bundlemoduliproblem11}-\ref{def_bundlemoduliproblem14}
    follow from Lemmas \ref{lem_complementarypolyhedronparahorichiggs} and
    \ref{lem_stabilityequivalenceparahorichiggs}, where we
    define facets $P$ of $\Phi(\mathrm{Ad}(\xi)_\eta,\mathbf{T}_\eta)$ to be
    admissible, in the sense of Definition
    \ref{def_bundlemoduliproblem} \ref{def_bundlemoduliproblem1}
    \ref{def_bundlemoduliproblem11}, if they admit a Higgs reduction.
    For Definition \ref{def_bundlemoduliproblem}
    \ref{def_bundlemoduliproblem1} \ref{def_bundlemoduliproblem15}, we use Lemma
    \ref{lem_stabilityequivalenceparahorichiggs}, which implies that
    $\mathcal{L}_{\det,\underline{\chi}}^H$ has compatible stability to
    $(\mathcal{M}\textsc{Bun}_{\mathcal{G}}^H,d^{\underline{\chi},H})$,
    in the sense of Definition
    \ref{def_bundlemoduliproblem} \ref{def_bundlemoduliproblem2}.
\end{remark}

To study the Jordan--H{\"o}lder stratification of
$(\mathcal{M}\textsc{Bun}_{\mathcal{G}}^H,d^{\underline{\chi},H})$,
let $\mathrm{JH}_{\mathcal{G}}$ be
the set of isomorphism classes of Levi-factors of $\mathcal{G}$. For
$\mu\in\mathrm{JH}_{\mathcal{G}}$, we define
$(\mathcal{M}\textsc{Bun}_{\mathcal{G}}^H)^{d^{\underline{\chi},H}-\mathrm{sst}}_\mu(K)$
as the
set of $K$-points of $d^{\underline{\chi},H}$-semistable parahoric
Higgs $\mathcal{G}$-torsors $(\xi,\varphi)$ with a Jordan--H{\"o}lder Higgs
reduction $s:X\rightarrow\xi/\mathcal{G}_{\mathsf{P}}$, such that
$\mathcal{L}_{\mathsf{P}}$
is in the isomorphism class $\mu$. By construction, it is clear that
this induces the same Jordan--H{\"o}lder stratification as in Subsection
\ref{subsec_jhstratifications}. The following generalizes Theorems
\ref{thm_jhstratificationprincipal} and \ref{thm_jhstratificationparahoric}.

\begin{theorem}[(Jordan--H{\"o}lder
    stratification)]\label{thm_jhstratificationparahorichiggs}
    Let $K$ be algebraically closed of characteristic $0$, and let
    $\underline{\chi}$ be an admissible stability parameter. The
    Jordan--H{\"o}lder stratification of
    $(\mathcal{M}\textsc{Bun}_{\mathcal{G}}^H,d^{\underline{\chi},H})$,
    from Subsection
    \ref{subsec_jhstratifications}, consists of locally closed substacks.
    Moreover,
    $(\mathcal{M}\textsc{Bun}^H_{\mathcal{G}},d^{\underline{\chi},H})$
    exhibits a compatible good moduli space
    $\pi:(\mathcal{M}\textsc{Bun}_{\mathcal{G}})^{d^{\underline{\chi},H}-\mathrm{sst}}\rightarrow
    (M\textsc{Bun}^H_{\mathcal{G}})^{d^{\underline{\chi},H}-\mathrm{sst}}$,
    from Definition \ref{def_jhstratification} \ref{def_jhstratification2}.
\end{theorem}

\begin{proof}
    We have a good moduli space
    $\pi:(\mathcal{M}\textsc{Bun}_{\mathcal{G}}^H)^{d^{\underline{\chi},H}-\mathrm{sst}}\rightarrow
    (M\textsc{Bun}_{\mathcal{G}}^H)^{d^{\underline{\chi},H}-\mathrm{sst}}$,
    then for
    $\mu\in\mathrm{JH}_{\mathcal{G}}$,
    $(\mathcal{M}\textsc{Bun}_{\mathcal{G}}^H)^{d^{\underline{\chi},H}-\mathrm{sst}}_\mu$
    is the
    $\pi$-preimage of a subspace
    $(M\textsc{Bun}_{\mathcal{G}}^H)^{d^{\underline{\chi},H}-\mathrm{sst}}_\mu$
    of
    $(M\textsc{Bun}_{\mathcal{G}}^H)^{d^{\underline{\chi},H}-\mathrm{sst}}$
    defined as
    follows:
    $(M\textsc{Bun}_{\mathcal{G}}^H)^{d^{\underline{\chi},H}-\mathrm{sst}}_\mu$
    is
    the image of a morphism
    $\iota:(M\textsc{Bun}_{\mathcal{L}_{\mathsf{P}}}^H)^{d^{\underline{\chi}^{\mathcal{L}_{\mathsf{P}}},H}-\mathrm{st}}\rightarrow
    (M\textsc{Bun}_{\mathcal{G}}^H)^{d^{\underline{\chi},H}-\mathrm{sst}}$
    extending the
    structure group through the canonical inclusion
    $\mathcal{L}_{\mathsf{P}}\hookrightarrow\mathcal{G}$. Thus,
    $(\mathcal{M}\textsc{Bun}_{\mathcal{G}}^H)^{d^{\underline{\chi},H}-\mathrm{sst}}_\mu$
    forms a
    substack of
    $(\mathcal{M}\textsc{Bun}_{\mathcal{G}}^H)^{d^{\underline{\chi},H}-\mathrm{sst}}$.
    To show that
    $(\mathcal{M}\textsc{Bun}_{\mathcal{G}}^H)^{d^{\underline{\chi},H}-\mathrm{sst}}_\mu$
    is a
    locally closed substack, we claim that $\iota$ is an immersion of
    algebraic spaces.
    Since $d^{\underline{\chi},H}$-polystable points in
    $(\mathcal{M}\textsc{Bun}^H_{\mathcal{G}})^{d^{\underline{\chi},H}-\mathrm{sst}}_\mu(K)$
    correspond uniquely to $K$-points in
    $(M\textsc{Bun}^H_{\mathcal{L}_{\mathsf{P}}})^{d_{M\textsc{Bun}_{\mathcal{L}_{\mathsf{P}}}}-\mathrm{st}}(K)$,
    $\iota$ is a monomorphism.
    For a $K$-scheme $T$ and a family
    $(\xi,\varphi)\in M\textsc{Bun}_{\mathcal{G}}^H(T)$ of parahoric Higgs
    $\mathcal{G}$-torsors, let $T'$ be the locus in $T$, for which
    $\xi/\mathcal{G}_{\mathsf{P}}\rightarrow T'\times X$ admits a
    section $s$ such that
    $\varphi:T'\times
    X\rightarrow\mathrm{ad}(\xi)\otimes\mathcal{K}(D)$ factorizes
    through $\varphi_{\mathsf{P}}:T'\times
    X\rightarrow\mathrm{ad}(s^*\xi)\otimes\mathcal{K}(D)$, then
    $(\xi_{T'},\varphi_{T'})$ over $T'$ is precisely in the image of
    $\iota$. It suffices to show that $T'$ is locally closed in $T$:
    By trivializing locally inside $T'\times X$, sections of
    $\xi/\mathcal{G}_{\mathsf{P}}\rightarrow T'\times X$ map into
    Bruhat cells of
    $\mathcal{G}/\mathcal{G}_{\mathsf{P}}$, which are locally closed.
    Furthermore, the Higgs fields $\varphi$ that can
    factorize into $\varphi_{\mathsf{P}}$ form the kernel of the quotient
    $\mathrm{ad}(\xi)\otimes\mathcal{K}(D)\rightarrow(\mathrm{ad}(\xi)\otimes\mathcal{K}(D))/(\mathrm{ad}(s^*\xi)\otimes\mathcal{K}(D))$,
    which is closed in $\mathrm{ad}(\xi)\otimes\mathcal{K}(D)$. Hence, $T'$ is
    locally closed in $T$. Thus, $\iota$ is an immersion of algebraic
    spaces locally of finite type, so
    $(\mathcal{M}\textsc{Bun}_{\mathcal{G}}^H)^{d^{\underline{\chi},H}-\mathrm{sst}}_\mu$
    is a
    locally closed substack of
    $(\mathcal{M}\textsc{Bun}_{\mathcal{G}}^H)^{d^{\underline{\chi},H}-\mathrm{sst}}$.

    The fact that $\pi$ is a compatible good moduli space comes from
    the correspondence between $d^{\underline{\chi},H}$-polystable
    points in
    $(\mathcal{M}\textsc{Bun}^H_{\mathcal{G}})^{d^{\underline{\chi},H}-\mathrm{sst}}_\mu(K)$
    and $K$-points in
    $(M\textsc{Bun}_{\mathcal{L}_{\mathsf{P}}}^H)^{d^{\underline{\chi}^{\mathcal{L}_{\mathsf{P}}},H}-\mathrm{st}}$,
    which implies that quasi-S-equivalence implies S-equivalence.
\end{proof}

\begin{example}[(Jordan--H{\"o}lder filtrations of a parahoric Higgs
    torsor)]\label{ex_jhfiltrationsparahorichiggs}
    We return to Example \ref{ex_jhfiltrationsparahoric} and introduce a Higgs
    field $\widetilde{\varphi}$ to the parahoric
    $\widetilde{\mathcal{G}}$-torsor $\widetilde{\mathrm{Fr}(F)}$.
    The difference between the
    $\mathcal{L}_{\det,\underline{\widetilde{\chi}}}$-stability conditions
    of $\widetilde{\mathrm{Fr}(F)}$ and the
    $\mathcal{L}_{\det,\underline{\widetilde{\chi}}}^H$-stability
    conditions of $(\widetilde{\mathrm{Fr}(F)},\widetilde{\varphi})$
    completely depend on whether $\widetilde{\varphi}$ preserves
    $L_0$ or $L'$, giving the following cases:
    \begin{enumerate}[label=(\roman*)]
        \item $\widetilde{\varphi}$ preserves $L_0$ and $L'$, so all
            parabolic reductions of $\widetilde{\mathrm{Fr}(F)}$ are Higgs
            reductions. Thus, we have the same stability conditions
            as in Example \ref{ex_jhfiltrationsparahoric}, i.e., $C:=0$ from
            Lemma
            \ref{lem_stabilityequivalenceparahorichiggs} can
            be chosen, and the convex hull of
            $d^{\underline{\widetilde{\chi}},\widetilde{\varphi},C}_{\mathrm{Ad}(\widetilde{\mathrm{Fr}(F)})_\eta,\widetilde{\mathbf{T}}_\eta}$
            can be identified with the interval in $\mathbb{R}$:
            \begin{equation*}
                [-1/2,-l'-1/2]
            \end{equation*}

        \item $\widetilde{\varphi}$ only preserves $L_0$, then the
            convex hull of
            $d^{\underline{\widetilde{\chi}},\widetilde{\varphi},C}_{\mathrm{Ad}(\widetilde{\mathrm{Fr}(F)})_\eta,\widetilde{\mathbf{T}}_\eta}$
            can be identified with the interval in $\mathbb{R}$:
            \begin{equation*}
                [-1/2,-l'-1/2+C]
            \end{equation*}
            where $C>l'+1/2$ for $C$ from Lemma
            \ref{lem_stabilityequivalenceparahorichiggs}.
        \item $\widetilde{\varphi}$ only preserves $L'$, then the
            convex hull of
            $d^{\underline{\widetilde{\chi}},\widetilde{\varphi},C}_{\mathrm{Ad}(\widetilde{\mathrm{Fr}(F)})_\eta,\widetilde{\mathbf{T}}_\eta}$
            can be identified with the interval in $\mathbb{R}$:
            \begin{equation*}
                [-1/2-C,-l'-1/2]
            \end{equation*}
            where there is no restriction on $C\geqslant0$ from Lemma
            \ref{lem_stabilityequivalenceparahorichiggs}.
        \item $\widetilde{\varphi}$ does not preserve $L_0$ nor $L'$,
            then the convex hull of
            $d^{\underline{\widetilde{\chi}},\widetilde{\varphi},C}_{\mathrm{Ad}(\widetilde{\mathrm{Fr}(F)})_\eta,\widetilde{\mathbf{T}}_\eta}$
            can be identified with the interval in $\mathbb{R}$:
            \begin{equation*}
                [-1/2-C,-l'-1/2+C]
            \end{equation*}
            where $C>l'+1/2$ for $C$ from Lemma
            \ref{lem_stabilityequivalenceparahorichiggs}.
    \end{enumerate}
    From these cases, the only difference between
    $\mathcal{L}_{\det,\underline{\widetilde{\chi}}}$-stability conditions
    of $\widetilde{\mathrm{Fr}(F)}$ and the
    $\mathcal{L}_{\det,\underline{\widetilde{\chi}}}^H$-stability
    conditions of $(\widetilde{\mathrm{Fr}(F)},\widetilde{\varphi})$ arises
    from whether $L_0$ is preserved by $\widetilde{\varphi}$ or not.
\end{example}

\begin{acknowledgements}
    The author thanks Jochen Heinloth and Johan Martens for helpful
    discussions and support.
\end{acknowledgements}

\printbibliography[heading=bibnumbered]

\end{document}